\documentclass[table,x11names]{article}

\usepackage{graphicx}
\usepackage[margin=1in]{geometry}
\usepackage{mathtools,amssymb,amsthm}
\usepackage{url}
\usepackage{verbatim}

\usepackage{array}
\newcolumntype{P}[1]{>{\centering\arraybackslash}p{#1}}

\usepackage{hyperref}
\hypersetup{
    colorlinks = true,
    final=true,
    plainpages=false,
    pdfstartview=FitV,
    pdftoolbar=true,
    pdfmenubar=true,
    pdfencoding=auto, 
    psdextra, 
    bookmarksopen=true,
    bookmarksnumbered=true,
    breaklinks=true,
    linktocpage=true
}
\usepackage[nameinlink, capitalise, noabbrev]{cleveref}

\usepackage{csquotes}

\usepackage{caption,subcaption}
\usepackage{booktabs}

\usepackage{xparse}
\let\realItem\item 
\makeatletter
\NewDocumentCommand\myItem{ o }{%
   \IfNoValueTF{#1}%
      {\realItem}
      {\realItem[#1]\def\@currentlabel{#1}}
}
\makeatother

\usepackage[shortlabels]{enumitem}    
\setlist[enumerate]{
    before=\let\item\myItem,       
    label=\textnormal{(\arabic*)}, 
    widest=(3A)                    
}

\usepackage[color=yellow]{todonotes}
\usepackage{todonotes}
\makeatletter
\define@key{todonotes}{TR}[]{%
    \setkeys{todonotes}{author=TR,color=blue!40,size=scriptsize}}%
\makeatother 

\usepackage[natbib, 
maxcitenames=3, 
maxbibnames=10,  
sorting=nyvt,
giveninits,
url=false,
doi=false,
isbn=false,
sortcites,
defernumbers,
backref,
backend=biber
]{biblatex}

\usepackage{pgfplots}
\usepackage[]{tikz}
\pgfplotsset{compat=1.9}

\crefname{algocf}{Algorithm}{Algorithms}
\usepackage{algorithm}
\usepackage{algpseudocode}

\usepackage{authblk}

\newtheorem{proposition}{Proposition}
\newtheorem{theorem}{Theorem}
\newtheorem{lemma}{Lemma}
\newtheorem{example}{Example}

\theoremstyle{definition}
\newtheorem{definition}{Definition}
\newtheorem{assumption}{Assumption}

\theoremstyle{remark}
\newtheorem{remark}{Remark}

\crefname{assumption}{Assumption}{Assumptions}

\newcommand{\abs}[1]{\left|#1\right|}
\newcommand{\norm}[1]{\left\|#1\right\|}

\newcommand{\tmpdiv}{\operatorname{div}}
\renewcommand{\div}{\tmpdiv}
\DeclareMathOperator*{\argmin}{arg\,min}

\DeclareMathOperator{\supp}{supp}
\DeclareMathOperator{\proj}{proj}

\DeclareMathOperator{\tr}{tr}
\newcommand{\st}{\,:\,}
\DeclareMathOperator{\Lip}{Lip}
\DeclareMathOperator{\I}{I}

\newcommand{\de}{\;\mathrm{d}}
\newcommand{\R}{\mathbb{R}}
\newcommand{\N}{\mathbb{N}}
\renewcommand{\rho}{\varrho}
 
\newcommand{\xhat}{\hat{x}}
\newcommand{\prob}{\mathcal{P}}
\newcommand{\E}{\mathbf{E}}
\newcommand{\T}{\mathcal{T}}

\newcommand{\mamu}{m_{\alpha}^{*}(\mu_t)}
\newcommand{\xmin}{x_{\phi}}
\newcommand{\nphiy}{\nabla \phi^*(y)}
\newcommand{\mphi}{m_{\phi}}
\newcommand{\Mphi}{M_{\phi}}
\newcommand{\dualball}{B_r^*(\xhat)}
\renewcommand{\H}{\mathcal{H}}
\newcommand{\normF}[1]{\|#1\|_F}

\newcommand{\mm}{\phi}
\newcommand{\mmel}{\phi_\text{\scriptsize en}}
\newcommand{\mmnl}{\phi_\text{\scriptsize nl}}
\newcommand{\mmell}{\tilde{\phi}_\text{\scriptsize en}}
\newcommand{\obj}{J}
\newcommand{\chara}{\iota}
\newcommand{\hyp}{\mathcal{H}}
\newcommand{\sph}{\mathcal{S}}

\newcommand{\quadr}{\mathcal{Q}}
\newcommand{\regp}{\lambda}

\newcommand{\en}[1]{\mathbf{#1}}
\renewcommand{\d}{\mathrm{d}}
\newcommand{\meas}{b}
\newcommand{\noise}{\varepsilon}
\newcommand{\noiselvl}{\delta}

\definecolor{cocoabrown}{rgb}{0.82, 0.41, 0.12}
\definecolor{guppiegreen}{rgb}{0.0, 1.0, 0.5}
\definecolor{grullo}{rgb}{0.66, 0.6, 0.53}
\definecolor{indiagreen}{rgb}{0.07, 0.53, 0.03}
\definecolor{brilliantrose}{rgb}{1.0, 0.33, 0.64}
\definecolor{midnightblue}{rgb}{0.1, 0.1, 0.44}
\definecolor{pyellow}{RGB}{221,170,51}
\definecolor{capri}{rgb}{0.0, 0.75, 1.0}
\definecolor{mirror}{named}{brilliantrose}

\definecolor{cadmiumgreen}{rgb}{0.0, 0.42, 0.24}
\definecolor{hotpink}{rgb}{1.0, 0.41, 0.71}
\definecolor{cgreen}{RGB}{0, 180, 100}
\hypersetup{
    colorlinks=true,
    linkcolor=cadmiumgreen,
    citecolor=cgreen,
    urlcolor=blue,
    linktocpage
}

\numberwithin{equation}{section}

\pgfplotsset{
    MirrorCBO/.style={
        mirror, mark=*, 
        mark repeat={100}, 
        mark options={solid,fill=mirror}
        },
    ProxCBO/.style={
        capri, mark=o, 
        mark repeat={100}, 
        mark options={solid,fill=capri, 
                      line width=1pt, scale=1.5}
        },
    PenalizedCBO/.style={
        cocoabrown, mark=triangle, 
        mark repeat={100},
        mark options={solid,fill=cocoabrown}
        },
    DualCBO/.style={
        blue, mark=pentagon*, 
        mark repeat={100},
        mark options={solid,fill=blue}
        },
    Drift/.style={
        indiagreen, mark=square*, 
        mark repeat={100},
        mark options={solid,fill=indiagreen, line width=1pt}
        },
    Combi/.style={
        pyellow, mark=square, 
        mark repeat={100},
        mark options={solid,fill=pyellow, line width=1pt, scale=1.5}
        },
    CBO/.style={
        black, mark=pentagon, 
        mark repeat={100}, 
        mark options={solid,fill=black, scale=2, line width=1pt}
        },
    Hyper/.style={
        midnightblue, mark=diamond, 
        mark repeat={100}, 
        mark options={solid,fill=midnightblue,}
        },
}
\title{MirrorCBO: A consensus-based optimization method\\in the spirit of mirror descent}
\date{\today}

\author[1]{Leon Bungert}
\author[2]{Franca Hoffmann}
\author[2]{Doh Yeon Kim}
\author[3]{Tim Roith}

\affil[1]{Institute of Mathematics, Center for Artificial Intelligence and Data Science (CAIDAS), University of Würzburg, Emil-Fischer-Str. 40, 97074 Würzburg, Germany}
\affil[2]{Department of Computing and Mathematical Sciences, Caltech, 1200 E California Blvd. MC 305-16, Pasadena, USA}
\affil[3]{Helmholtz Imaging, Deutsches Elektronen-Synchrotron DESY, Notkestr. 85, 22607 Hamburg, Germany}

\begin{document}


\clearpage

\maketitle
\begin{abstract}
    In this work we propose MirrorCBO, a consensus-based optimization (CBO) method which generalizes standard CBO in the same way that mirror descent generalizes gradient descent. 
    For this we apply the CBO methodology to a swarm of dual particles and retain the primal particle positions by applying the inverse of the mirror map, which we parametrize as the subdifferential of a strongly convex function $\phi$. In this way, we combine the advantages of a derivative-free non-convex optimization algorithm with those of mirror descent.
    As a special case, the method extends CBO to optimization problems with convex constraints. 
    Assuming bounds on the Bregman distance associated to $\phi$, we provide asymptotic convergence results for MirrorCBO with explicit exponential rate. 
    Another key contribution is an exploratory numerical study of this new algorithm across different application settings, focusing on (i) sparsity-inducing optimization, and (ii) constrained optimization, demonstrating the competitive performance of MirrorCBO. We observe empirically that the method can also be used for optimization on (non-convex) submanifolds of Euclidean space, can be adapted to mirrored versions of other recent CBO variants, and that it inherits from mirror descent the capability to select desirable minimizers, like sparse ones. We also include an overview of recent CBO approaches for constrained optimization and compare their performance to MirrorCBO.
\par\vskip\baselineskip\noindent
\textbf{Keywords}: consensus-based optimization, mirror descent, gradient-free optimization, global optimization, mean-field analysis, Fokker--Planck equations\\
\textbf{AMS Subject Classification}: 35B40, 35Q84, 35Q89, 35Q90, 65K10, 90C26, 90C56

\end{abstract}

\tableofcontents

\section{Introduction}

\subsection{Consensus-based optimization}

Since its introduction in \cite{pinnau2017consensus} the consensus-based optimization (CBO) method has sparked a plethora of work, including rigorous theoretical analysis as well as a good number of extensions and generalizations of the original method such that by now it is fair to speak of a family of consensus-based methods.
The purpose of this paper is to introduce MirrorCBO, a versatile generalization of the original CBO method using inspiration from mirror descent.

The original CBO method is an interacting particle system with the purpose of minimizing a given cost function $J:\R^d\to\R$ in a gradient-free manner. 
This can be desirable when $J$ is not differentiable, a black-box function with non-accessible or expensive to evaluate gradient, or a heavily non-convex function in which case gradient information is not conclusive for finding global minima.
CBO evolves an ensemble of $N\in\N$ particles $\{x^{(i)}\}_{i = 1}^N \subset \R^{d}$ towards a common consensus-point and is additionally driven by weighted noise which is encoded by the following system of stochastic differential equations
\begin{align}\label{eq: cbo}
   \d x_t^{(i)} = - \left(x_t^{(i)} - m_\alpha(\rho_t^N)\right)\de t + \sqrt{2\sigma^2}|x_t^{(i)} - m_{\alpha}(\rho_t^N)|\de W_t^{(i)}\,.
\end{align}
Here $\{t\mapsto W_t^{(i)}\}_{i = 1}^N$ are independent $\R^d$-valued Brownian motions, $\sigma\geq 0$ controls the magnitude of the noise, and 
\begin{align}\label{eq:weighted_mean}
    m_{\alpha}(\rho_t^N):= \frac{\sum_{i = 1}^N x_t^{(i)}\exp(-\alpha J(x_t^{(i)}))}{\sum_{j = 1}^N \exp(-\alpha J(x_t^{(j)}))}
\end{align}
is the consensus point, depending only on the empirical measure of the particle positions $\rho_t^N:=\frac{1}{N}\sum_{i=1}^N\delta_{x_i^{(t)}}$.
Depending on an inverse temperature parameter $\alpha>0$, the consensus point is a weighted convex combination of the particle positions, where particles with a small value of $J$ get more weight than others.

By now, the theoretical properties of \labelcref{eq: cbo} are relatively well understood. 
With a convergence analysis of the corresponding mean-field equation it was shown in \cite{carrillo2018analyticalframeworkconsensusbasedglobal,fornasier2024consensus} under relatively mild assumptions on $J$ that the Fokker--Planck equation associated to \labelcref{eq: cbo} concentrates close to the global minimizer of $J$.
While \cite{carrillo2018analyticalframeworkconsensusbasedglobal} relied on variance-decay techniques to prove that exponentially fast, the Fokker--Planck equation concentrates at a Dirac distribution which is then shown to be located close to the global minimizer of $J$ for $\alpha$ sufficiently large, the proof framework developed in \cite{fornasier2024consensus} allows to directly estimate the Wasserstein-2 distance between the ensemble and the Dirac distribution at the minimizer and prove that it decays exponentially fast until it hits a given accuracy level.
The approach of \cite{fornasier2024consensus} turned out to be very versatile and was later used to analyze many other consensus-based approaches, see, e.g., \cite{bungert2024polarized,fornasier2024pdeframeworkconsensusbasedoptimization,fornasier2022convergence,trillos2024cb2oconsensusbasedbileveloptimization}.
It is important to emphasize that both convergence results do not need any convexity assumptions on $J$.
The computational complexity of minimizing non-convex functions is hidden in the fact that \cite{carrillo2018analyticalframeworkconsensusbasedglobal} provides \emph{mean-field} convergence results that are not valid for finitely many particles.
For convergence results of the actual time-discretized particle method and estimates of the computational complexity we refer to~\cite{fornasier2024consensus}.

\begin{figure}[t]
\centering
\includegraphics[width=.8\textwidth]{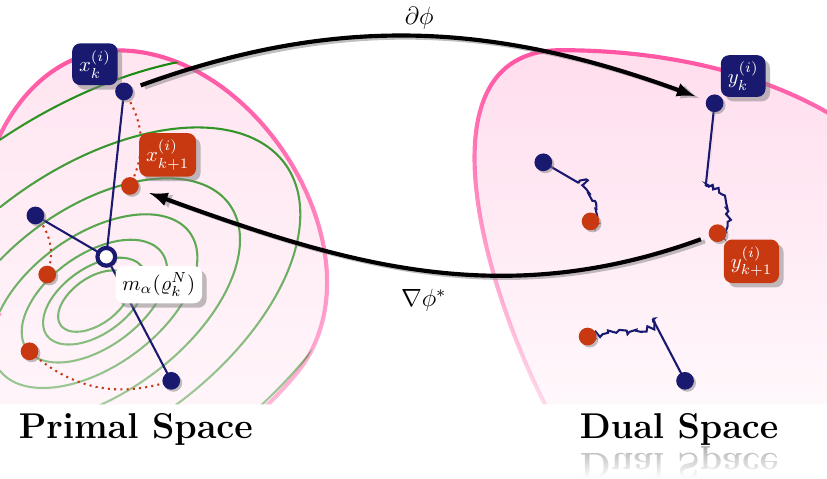}
\caption{A visual representation of the proposed MirrorCBO algorithm.}
\label{fig:conceptual_figure}
\end{figure}

\subsection{Mirror descent and MirrorCBO}
\label{sec:mirror_descent}

The purpose of this paper is to extend CBO in the spirit of mirror descent.
For this we first recap the basics of mirror descent, originally proposed in \cite{nemirovskij1983problem}. 
To derive it we start with gradient descent which performs the iteration $x_{k+1}=x_k - \tau\nabla J(x_k)$ where $\tau>0$ is a step size. 
This can be equivalently written through the following minimization problem
\begin{align}\label{eq:GD}
    x_{k+1} = \argmin_{x\in\R^d}
    \left\lbrace\frac{1}{2}\abs{x-x_k}^2+\tau\langle\nabla J(x_k),x-x_k\rangle\right\rbrace
\end{align}
which asks the next iterate $x_{k+1}$ to be close to the previous one while at the same time minimizing the linearization of the cost function $J$.
From this perspective it becomes clear that gradient descent minimizes the cost function $J$ with respect to the \emph{Euclidean distance}.
Mirror descent replaces this distance by a so-called Bregman distance that it parametrized by a convex function $\phi:\R^d\to[0,\infty
]$, called the \emph{distance generating function}:
\begin{align}\label{eq:MD}
    x_{k+1} = \argmin_{x\in\R^d}
    \left\lbrace D_\phi^{y_k}(x,x_k)+\tau\langle\nabla J(x_k),x-x_k\rangle\right\rbrace.
\end{align}
Here $y_k\in\partial\phi(x_k)$ is a subgradient and the Bregman distance is defined as $$D_\phi^{y_k}(x,x_k)=\phi(x)-\phi(x_k)-\langle y_k,x-x_k\rangle\,.$$
Note that Bregman distances are not distances in the strict sense since they do not satisfy the triangle inequality and are generally not symmetric.
The optimality condition of the convex optimization problem \labelcref{eq:MD} reads $y_k-\tau\nabla J(x_k)\in \partial \phi(x_{k+1})$. 
If $\mm$ is strongly convex \labelcref{eq:MD} has a unique solution, characterized by
\begin{align*}
x_{k+1} = \nabla\phi^*(y_k-\tau\nabla J(x_k)),
\qquad
y_k\in\partial\phi(x_k),
\end{align*}
where $\phi^*$ is the convex conjugate of $\phi$ to be defined later on in \labelcref{eq:convex_conjugate} (for more details on its properties, see \cref{sec:mirrormap}). If, additionally, $\mm$ is differentiable, this algorithm is referred to as \emph{mirror descent}, which can be traced back to \cite{nemirovskij1983problem}.
The nomenclature comes from the fact that the primal variables are \enquote{mirrored} to a dual space through the \emph{mirror map} $\partial\phi$, where a gradient step is performed for the dual variables.
Subsequently, the latter are mirrored back through $\nabla\phi^*$ to obtain the primal variables of interest, see \cref{fig:conceptual_figure}.

Notably, for the quadratic function $\phi(x)=\frac12\abs{x}^2$ mirror descent coincides precisely with gradient descent since in this case the Bregman distance $D_\phi^{y_k}(x,x_k)=\frac12\abs{x-x_k}^2$ coincides with the squared Euclidean one. If $\mm$ is not differentiable, the above algorithm is a special case of so-called \enquote{unified mirror descent} \cite{juditsky2023unifying}. The $y$ variables are typically referred to as dual variables, whereas the $x$ variables are the primal ones. 
By selecting a special sequence of dual variables $y$, we obtain the iteration
\begin{align}\label{eq:MD_iteration}
    y_{k+1}=y_k-\tau\nabla J(x_k),
    \qquad 
    x_{k+1}=\nabla\phi^*(y_{k+1}).
\end{align}
The update \labelcref{eq:MD_iteration} is referred to as \emph{lazy mirror descent} or \emph{Nesterov's dual averaging} \cite{nesterov2009primal}. 
It can be fully expressed in terms of the dual variables $y$, and only for evaluating $\nabla J(x_k)$, we need to invoke $\nabla\phi^*$. 
In fact more general versions, like in \cite{juditsky2023unifying} allow an arbitrary dual increment $y_{k+1} = y_k - \zeta_k$, with $\zeta_k$ not necessarily depending on $x_k$. Throughout this work, we will adapt this \enquote{lazy} notion of mirror descent, dropping the adjective \enquote{lazy}, and use \labelcref{eq:MD_iteration} as the reference algorithm to derive MirrorCBO.

In the case where $\mm$ is differentiable, the two formulations introduced so far are equivalent. When we additionally consider constrained optimization, this changes again. Namely, in the lazy setting, we can introduce constraints, by choosing the function $\phi + \iota_C$, where $\iota_C$ denotes the characteristic function\footnote{$\iota_C(x)=0$ if $x\in C$ and $+\infty$ else.} of a closed convex set $C\subset\R^d$. For $\phi=\frac{1}{2}\abs{x}^2$, this yields the update $y_{k+1} = y_k - \tau\nabla J(\proj_C(y_k))$. In \cite[Remark 9.6]{beck2017first} the constrained mirror iteration is allowed to choose a subgradient $y_k\in\partial (\phi + \iota_C)(x_k)$.
Often, however, authors make the concrete distinction for mirror descent (and differentiable $\phi$) to set $y_k = \nabla\phi (x_k)$ and then using $\nabla(\phi + \iota_C)^*$ for the inverse mirror map.  For $\phi=\frac{1}{2}\abs{x}^2$, this instead yields projected gradient descent $x_{k+1} = \proj_C(x_k - \tau \nabla J(x_k))$. We review these differences in more detail in \cref{sec:constraints_numerics}. We also want to highlight that for online optimization and the \emph{Follow-The-Regularized-Leader} algorithm the relation between mirror descent and dual averaging becomes more intricate, see \cite{mcmahan2017survey}, which however is beyond the scope of the present work.

Sending the time step size $\tau>0$ in \labelcref{eq:MD_iteration} to zero, mirror descent converges to the ordinary differential equation
\begin{align}\label{eq:MF}
    \dot y(t)=-\nabla J(\nabla\phi^*(y(t)))
\end{align}
which for quadratic $\phi$ is precisely the gradient flow of $J$.
Using the chain rule and assuming sufficient regularity of $\phi^*$, the primal variables $x(t):=\nabla\phi^*(y(t))$ undergo the flow $\dot x(t) = -D^2\phi(x)^{-1}\nabla J(x(t))$ which can be interpreted as the ``natural gradient flow'' \cite{amari1998natural} of $J$ where the Hessian matrix $D^2\phi$ encodes a metric on $\R^d$.

While this primal formulation allows for a nice interpretation, it is quite limiting in terms of the regularity assumptions that are required on $\phi$.
In particular, the dual formulations of mirror descent as in \labelcref{eq:MD,eq:MD_iteration,eq:MF} allow for distance generating functions $\phi$ which are merely strongly convex. This was exploited in the context of inverse problems, where the method is typically called Bregman iterations \cite{osher2005iterative,burger2016bregman}, and in the context of learning sparse neural networks \cite{bungert2022bregman,heeringa2023learning}.

The proposed MirrorCBO method, explained in more detail in \cref{sec:dynamics} below, is the following modification of \labelcref{eq: cbo}
\begin{align}\label{eq:mirrorcbo-intro}
   \d y_t^{(i)} 
   = 
   - \left(x_t^{(i)} - m_\alpha(\rho_t^N)\right)\de t + \sqrt{2\sigma^2}|x_t^{(i)} - m_{\alpha}(\rho_t^N)|\de W_t^{(i)}\,,
   \quad
   \text{where
   }
   x_t^{(i)} 
   =
   \nabla\phi^*(y_t^{(i)})
\end{align}
and $m_\alpha(\rho_t^N)$ is the weighted mean of the primal particles $\{x_t^{(i)}\}_{i=1}^N$ defined as in \labelcref{eq:weighted_mean}.
We refer to \cref{fig:conceptual_figure} for an illustration of MirrorCBO.
It is important to mention that, while in this paper we propose and analyze a mirror version of the original CBO method as proposed in \cite{pinnau2017consensus}, our approach directly generalizes to most other CBO variants, e.g., polarized CBO \cite{bungert2024polarized}, CBO with anisotropic noise \cite{fornasier2022convergence}, consensus-based sampling \cite{carrillo2022consensus}, etc. 
For example, for the more general model 
\begin{align}\label{eq: mcbo_general}
   \d y_t^{(i)} 
   = 
   - \left(x_t^{(i)} - m_\alpha(\rho_t^N)\right)H\left(J(x_t^{(i)})-J(m_\alpha(\rho_t^N))\right)\de t + \sqrt{2\sigma^2} S\left(x_t^{(i)} - m_{\alpha}(\rho_t^N)\right) \de W_t^{(i)}\,,
\end{align}
where $H:\R\to\R$ is a Lipschitz continuous cut-off function approximating the Heaviside function and $S: \R^d \rightarrow \R^{d \times d}$ is a Lipschitz continuous function with linear growth modeling the noise, we can still apply our proof framework and obtain analogous results regarding well-posedness and long-time asymptotics. 
Note that by taking $H\equiv 1$ and $S(u)=\abs{u}\I$ we obtain the basic MirrorCBO above, for $S(u) = \operatorname{diag}(u)$ we obtain a mirror version of CBO with anisotropic noise studied in \cite{carrillo2021consensus, fornasier2022convergence}. For details on the extension of our convergence results to this case, refer to \cref{rmk: aniso_longtime_asymptotics}; well-posedness of the corresponding particle and mean field system is discussed in \cref{rmk: aniso_wellposed_particle,rmk: aniso_wellposed_mf}, respectively. 

\subsection{Our contributions}
\label{sec:contributions}

In this work, we introduce a new algorithm: MirrorCBO. Our contributions are both theoretical and numerical in nature: (1) we provide well-posedness and convergence guarantees for this new algorithm, and (2) we demonstrate the performance of MirrorCBO in a number of application settings, benchmarking it against relevant alternatives in the literature. To do so, we provide a comparative overview of different constrained CBO methods, which we believe is a useful contribution in itself given the multitude of new approaches that have been proposed in the past 5 years, and for which a comprehensive benchmarking is still lacking. 

\paragraph{Theoretical study.} 
Building on the approaches in \cite{carrillo2018analyticalframeworkconsensusbasedglobal, fornasier2024consensus}, we prove well-posedness of the continuous-time algorithm dynamics of MirrorCBO \labelcref{eq:mirrorcbo-intro}, as well as for the corresponding mean-field partial differential equation. To extend the results in \cite{carrillo2018analyticalframeworkconsensusbasedglobal} to the MirrorCBO setting, we leverage the strong convexity condition of the distance generating function $\phi$, which renders its convex conjugate Lipschitz continuous, in order to transfer analysis previously done in the primal space to the dual space. 
Furthermore, we introduce a suitable Lyapunov functional based on the Bregman divergence that allows us to show exponential convergence for the mean-field law of the MirrorCBO particles towards the global minimizer of the objective function. The convergence proof is largely based on the approach in \cite{fornasier2024consensus}, but again transferred to the dual space via the mirror map.
Thanks to our assumptions, this Lyapunov functional is equivalent to the Wasserstein-2 distance between the mean-field law and the Dirac distribution located at the global minimum of the objective. In \cref{sec:numex} we briefly mention interesting open questions for the asymptotic behavior of the mean-field law of MirrorCBO for choices of mirror maps for which these assumptions are not satisfied. Throughout, we also comment on how these theoretical results can be extended to mirror versions of more recent CBO variants.
One of the main contributions in our theoretical analysis is the construction of a suitable Lyapunov functional that naturally generalizes the classical quadratic Lyapunov function used in standard CBO.
Our analysis extends theoretical guarantees to the mirror framework, where duality and non-Euclidean geometry play central roles. Proving Wasserstein stability and deriving quantitative Laplace principles in this context requires to carefully exploit convexity properties of the mirror map, and to handle translations between primal and dual spaces in ways that are technically demanding and not directly transferable from existing results.

Consistent with other studies on the long-time asymptotics of CBO algorithms, our convergence analysis presented in \cref{sec: convergence_analysis} focuses on the mean-field description of the original algorithm. This analytical framework is advantageous, as it allows the utilization of established tools from partial differential equations (PDE) and stochastic differential equations (SDE). However, a natural question arises regarding the validity of such analyses when directly applied to finite particle dynamics. Several works have addressed this question, including direct analyses for finite-particle systems by Ha and collaborators \cite{doi:10.1142/S0218202522500245, ha-timediscrete-cbo2021},
or by focusing on the connection between the mean-field model and particle-based algorithms \cite{fornasier2024consensus}. Specifically, in the proof of \cite[Theorem 3.8]{fornasier2024consensus}, the authors decompose the error into three distinct components: the first is attributed to time discretization and is manageable through classical numerical analysis; the second emerges from the quantitative mean-field limit approximation; and the third component corresponds precisely to the convergence analysis conducted in the mean-field regime. In \cref{sec: convergence_analysis}, we target this third component in the context of our MirrorCBO algorithm, for more details also see \cref{sec:dynamics}.

\paragraph{Numerical study.}
Numerous works established CBO as a powerful optimization algorithm. E.g., in the context of inverse problems, \cite{riedl2024leveraging} employs CBO for a compressed sensing task. 
In this regard, we explore the capabilities of MirrorCBO with the sparsity promoting elastic net functional, in \cref{sec:sparse}, where we also evaluate numerically how the rich theory of Bregman iterations for compressed sensing problems transfers to the CBO setting. 
A further successful application of CBO is gradient-free training of neural networks, as demonstrated in \cite{carrillo2021consensus,fornasier2022convergence}. The previously mentioned sparse MirrorCBO allows us to transfer the efficient learning framework from \cite{bungert2022bregman} to the gradient-free setting.

As explained above, mirror descent is often employed for constrained optimization. In \cref{sec:simplex}, we use this methodology for optimizing on the simplex. By doing so we also introduce the CBO variant of exponentiated gradient descent \cite{kivinen1997exponentiated}. Beyond constraints on convex sets, there have been numerous recent works introducing CBO variants for optimization on hypersurfaces or manifolds $\hyp$ \cite{Fornasier2020ConsensusbasedOO,CBO_sphere, CBO_Stiefel,CBO_reg_giacomo, CBO_reg_jose, carrillo2024interacting,trillos2024cb2oconsensusbasedbileveloptimization}, which we review in detail in \cref{sec:constraints_numerics}. Choosing the distance generating function $\mm = \frac{1}{2} \abs{x}^2 + \iota_\hyp$ is not a common application of mirror descent, since for most hypersurfaces $\hyp$ this functions is non-convex. From an algorithmic point of view, the mirror machinery can still be employed, since merely the projection (or an approximation) $\nabla\phi^*(y)=\proj_\hyp(y)$ is necessary for the implementation. MirrorCBO performs surprisingly well on various hypersurfaces, mostly outperforming comparable methods. While this setting is not within the analytical scope of this work, it is an attractive algorithm, since the mechanism is almost identical to standard CBO (thus allowing to easily employ the full bag of CBO tricks and modifications), and the constraints can be outsourced to the projection. An even simpler and strongly related algorithm is proximal or projected gradient descent. E.g., if $\hyp$ is an affine hyperplane, then mirror descent with the above choice of distance generating function is equivalent to this algorithm. Projected CBO has already been introduced in \cite{CBO_finance_projection}, with constraints on convex sets. Our numerical contribution here is to also establish this optimizer with hypersurface constraints, which is very similar to hypersurface MirrorCBO in spirit and concerning numerical performance and simplicity. As a byproduct, this section serves as a review and benchmark of current constrained CBO methods, tested on hyperplanes, spheres, quadrics and the Stiefel manifold. Notably, for the sphere constraint, we are able to match the performance of \cite{Fornasier2020ConsensusbasedOO} on a phase retrieval task.

For a suitable choice of mirror map, mirror descent can be used to accelerate gradient descent via a pre-conditioner. In \cref{sec:precon}, we briefly comment why it does not make sense to use MirrorCBO for this purpose.

A topic which closely relates to constrained optimization is constrained sampling. 
In this context, mirror methods have been successfully used, see, e.g., \cite{zhang2020wassersteincontrolmirrorlangevin, ahn2021efficient, Hsieh2018MirroredLD, pmlr-v167-li22b} which work with the so-called mirror Langevin algorithm. 
A natural follow-up work would be to combine the consensus-based sampling approach from \cite{carrillo2022consensus} with our MirrorCBO method to obtain a consensus-based analogue of the mirror Langevin algorithm.

\subsection{Outline and notation}

After introducing our notation, the remainder of the paper is structured as follows. In \cref{sec:Mirror_CBO}, we introduce the proposed MirrorCBO method as a set of coupled stochastic differential equations, we discuss the corresponding mean field equation and state all technical assumptions for our analysis. 
Then in \cref{sec:wellposedness,sec: convergence_analysis}, we investigate the well-posedness of the particle method and the mean-field process as well as the convergence properties of the latter towards the global minimizer of the loss function. 
Lastly in \cref{sec:numerics}, we perform extensive numerical experiments and comparisons with existing methods. 
There we also elaborate more on previous constrained CBO methods, and the benefits that MirrorCBO has over those alternatives. 

Throughout the paper, we denote $\N:= \{ 0, 1,2, \dots\}$ and  $\R^d$ as the $d$-dimensional Euclidean space with its usual Euclidean norm denoted as $|\cdot|.$ 
For a matrix $A\in\R^{N\times d}$, we denote its Frobenius norm by $\normF{A}$.
The set of Borel probability measures over $\R^d$ is denoted as $\prob (\R^d)$, and those with bounded $p$-th moment ($p \geq 1$) are denoted by $\prob_p (\R^d)$. Also, we define the \mbox{Wasserstein-$p$} distance $W_p(\mu_0,\mu_1)$ as a metric to measure distances between two Borel probability measures $\mu_0, \mu_1 \in \prob_p(\R^d)$ via
\begin{align*}
    W_p^p(\mu_0,\mu_1):= \inf\left\{  \iint_{\R^d \times \R^d} |x - \bar{x}|^p \de\pi(x,\bar{x}) \,, \pi \in \Pi(\mu_0, \mu_1) \right\}\,,
\end{align*}
where $\Pi(\mu_0, \mu_1)$ denotes the set of all couplings of $\mu_0$ and $\mu_1$.
By convention,  we use $x\in \R^d$ to denote primal variables and we use $y\in\R^d$ for dual ones.
Correspondingly, we denote by $\rho, \mu \in \prob (\R^d)$ the respective laws of primal and dual particles.

\section{The MirrorCBO dynamics}
\label{sec:Mirror_CBO}

\subsection{Particle Dynamics}
\label{sec:dynamics}

Here we will introduce our MirrorCBO algorithm, which builds on the CBO dynamics from \cite{pinnau2017consensus} using ideas from mirror descent.
At an informal level, the MirrorCBO algorithm can be summarized in these steps:
First, given an ensemble of primal particles $\{x_t^{(i)}\}_{i=1}^N$ with a consensus point, we evolve the corresponding dual particles $\{y_t^{(i)}\}_{i=1}^N$ through the usual CBO dynamics towards the consensus point for an infinitesimal amount of time. 
Second, we map the dual particles back into the primal space by applying $\nabla \phi^*$, where $\phi$ is a strongly convex function. 

At the level of an interacting particle system, $\{x_t^{(i)}\}_{i=1}^N$ parametrized by time $t\geq 0$ the MirrorCBO dynamics take the form of the following system of stochastic differential equations

\begin{align}\label{eq:particle_x_y}
        \d y_t^{(i)} &= -\left(x_t^{(i)} - m_\alpha(\rho_t^N)\right)\d t + \sqrt{2\sigma^2}\abs{x_t^{(i)}-m_\alpha(\rho_t^N)}\d W_t^{(i)},
        \quad
        i=1,\dots,N,
\end{align}
where we define the particles in primal space and their consensus point as
\begin{align*}
    x_t^{(i)} := \nabla\phi^\ast(y_t^{(i)}),
    \quad
    i=1,\dots,N,
    \qquad
    m_\alpha(\rho_t^N) := \frac{\sum_{i = 1}^N w_{\alpha}(x_t^{(i)})x_t^{(i)}}{\sum_{i = 1}^N w_{\alpha}(x_t^{(i)})}.
\end{align*}

Here $t\mapsto W_t^{(i)}$ for $i\in\{1,\dots,N\}$ denote independent Brownian motions and we use the notation
\begin{align*}
    w_{\alpha}(x) := \exp(-\alpha J(x)),\qquad x\in\R^d.
\end{align*}
We can equivalently express the dynamics~\labelcref{eq:particle_x_y} in terms of the dual variable as

\begin{align}\label{eq:particle_y}
        \d y_t^{(i)} &= -\left(\nabla\phi^\ast(y_t^{(i)}) - m_\alpha^*(\mu_t^N)\right)\d t  + \sqrt{2\sigma^2}\abs{\nabla\phi^\ast(y_t^{(i)})-m_\alpha^*(\mu_t^N)} \d W_t^{(i)},
        \qquad
        i=1,\dots,N,
\end{align}
where we define the dual weighted mean of the empirical measure $\mu_t^N:=\frac{1}{N}\sum_{i=1}^N\delta_{y_t^{(i)}}$ via
\begin{align*}
    m_\alpha^*(\mu_t^N)
    :=
    \frac{\int_{\R^d}  w_{\alpha}^{*}(y) \nabla \phi^*(y)  \de\mu_t^N(y)}{\| w_{\alpha}^{*}\|_{L^1(\mu_t^N)} }
    = \frac{\sum_{i = 1}^N w_{\alpha}^{*}(y_t^{(i)}) \nabla \phi^* (y_t^{(i)}) }{\sum_{i = 1}^N w_{\alpha}^{*}(y_t^{(i)})}
\end{align*}
and, analogously, we use the notation
\begin{align*}
    w_{\alpha}^{*}(y):= \exp(-\alpha J \circ \nabla \phi^*(y))\,.
\end{align*}
Here, the dual weighted mean $m_\alpha^*(\mu)$ is equal to the consensus point $m_\alpha(\rho)$, as long as the dual distribution $\mu$ is related to the  primal distribution $\rho$ via $(\nabla\phi^*)_\sharp\mu = \rho$. Indeed,
\begin{align*}
    m_\alpha^*(\mu) = \frac{\int_{\R^d}  w_{\alpha}^{*}(y) \nabla \phi^*(y)  \de\mu(y)}{\| w_{\alpha}^{*}\|_{L^1(\mu)} }
    =
    \frac{\int_{\R^d}  w_{\alpha}(x) x\de\rho(x)}{\| w_{\alpha}\|_{L^1(\rho)} }
    =
    m_\alpha(\rho).
\end{align*}

Assuming propagation of chaos, the associated McKean--Vlasov process is given by
\begin{align}\label{eq: meanfield-mirrorcbo-sde}
        \d y_t= -\left(\nabla\phi^\ast(y_t) - m_\alpha^*(\mu_t)\right)\d t  + \sqrt{2\sigma^2}\abs{\nabla\phi^\ast(y_t)-m_\alpha^*(\mu_t)} \d W_t
\end{align}
where $t\mapsto W_t$ is $\R^d$-valued Brownian motion and $\mu_t := \operatorname{Law}(y_t)$ is the law of the stochastic process $y_t$ and satisfies the associated Fokker--Planck equation
\begin{align}\label{eq: meanfield-mirrorcbo-pde}
    \partial_t\mu_t(y) &= \div\Big(\mu_t(y)(\nabla\phi^\ast(y)-m_{\alpha}^*(\mu_t))\Big) + \sigma^2\Delta\left(\mu_t(y)\abs{\nabla\phi^\ast(y)-m_{\alpha}^*(\mu_t)}^2\right).
\end{align}
In \cref{sec:wellposedness} we shall prove well-posedness of the particle system \labelcref{eq:particle_y} and the mean-field equations \labelcref{eq: meanfield-mirrorcbo-sde,eq: meanfield-mirrorcbo-pde}.
Furthermore, convergence of the mean-field law to the global minimizer of the cost function $J$ will be shown in \cref{sec: convergence_analysis}.

 For practical implementations, understanding the behavior of the algorithm at the particle level is crucial. Implementation of MirrorCBO evolves particles $\left\{x_{\Delta t}^{(i)}:= \nabla\phi^\ast(y_{\Delta t}^{(i)})\right\}_{i=1}^N$ via an Euler-Maruyama time-discretization of the dynamics \eqref{eq:particle_x_y}. The convergence of MirrorCBO can thus be assessed by decomposing the Wasserstein-2 distance between the empirical measure of the discrete-time particles $\left\{x_{\Delta t}^{(i)}\right\}$ and the target Dirac located at the global minimizer $\xhat:=\argmin J$ of the objective function $J$ into three components: (i) time-discretization error, (ii) propagation of chaos error, and (iii) optimization error:
    \begin{align*}
        \E\left[ \left\| \frac{1}{N} \sum_{i = 1}^N x_{\Delta t}^{(i)} - \hat{x} \right\|_2^2 \right] \leq \underbrace{\E\left[ \left\|  \frac{1}{N} \sum_{i = 1}^N x_{\Delta t}^{(i)} - x_t^{(i)} \right\|_2^2  \right]}_{\rm (i) \,\text{time-discretization error}}
        + 
        \underbrace{\E\left[ \left\|  x_t^{(i)} - \overline{x}_t^{(i)} \right\|_2^2  \right]}_{\rm (ii)\, \text{propagation of chaos error}}
        + 
        \underbrace{\E\left[ \left\|  \overline{x}_t^{(i)} - \hat{x} \right\|_2^2  \right]}_{\rm (iii)\, \text{optimization error}}\,.
    \end{align*}
Here, $\overline{x}_t^{(i)}$ are i.i.d. mean-field particles with law $\rho_t:=(\nabla\phi^*)_\sharp\mu_t$ and $\mu_t$ solving \eqref{eq: meanfield-mirrorcbo-pde}.  
The main focus of the theoretical analysis in this paper is to provide rigorous guarantees for the optimization error, which by \cref{ass:bregman_phi} below is upper-bounded by the Lyapunov function studied in our main result \cref{thm: exp_decay_V}. 
The time-discretization error can be bounded in terms of the time-step size $\Delta t$, following for instance a similar approach as in \cite{fornasier2024consensus}. 
Here, the main challenge is the non-Lipschitzness of the coefficients in \labelcref{eq: meanfield-mirrorcbo-sde} which requires one to restrict oneself to a high-probability event on which the trajectories of \labelcref{eq: meanfield-mirrorcbo-sde} remain bounded. 
For the propagation of chaos error, we remark that extending existing results for the CBO algorithm such as \cite[Proposition 3.11]{fornasier2024consensus} to the MirrorCBO setting seems achievable by deriving appropriate moment bounds and using the the Lipschitzness of $\nabla \phi^*$, but the technical details of the analysis in the mirror framework would go beyond the scope of this work.
Together, these three errors would achieve error control for MirrorCBO similar to \cite[Theorem 3.8]{fornasier2024consensus} for CBO, at least in finite time-intervals. For convergence guarantees as $t\to\infty$, uniform-in-time estimates for (i) and (ii) are needed. These could be achieved for instance by building on the recent works \cite{UiT-CBO} (uniform-in-time propagation of chaos for CBO) and \cite{MT-model} (uniform-in-time transition from discrete-time to continuous-time dynamics for the Motsch--Tadmor model).

\subsection{Assumptions}
\label{sec:assumptions}

In this section, we outline the key assumptions on the cost function $J: \R^d \rightarrow \R$ and the distance generating function $\phi: \R^d \rightarrow [0,\infty]$ which we require for our well-posedness and convergence results.

\subsubsection{Cost function}

To ensure well-posedness of the CBO dynamics and convergence to a global minimizer, we impose the following assumptions on $J$ which coincide with those in \cite{fornasier2024consensus}.
First, for well-posedness of either the particle dynamics or the mean-field process, we pose
\begin{assumption}\label{ass: J}
\phantom{}
    \begin{enumerate}
    \item[(1A)] $J$ is locally Lipschitz continuous. \label{as: locally_lip}
    \item[(1B)] $J$ is bounded from below. Denote $\underline{J}:= \inf_{x\in \R^d} J(x)$. \label{as: J_inf}
    \item[(1C)] \label{as: generalized_locally_lipschitz} There exists constants $L_J, c_u > 0$ such that for all $x,\tilde x\in\R^d$:
    \begin{align} 
        &|J(x) - J(\tilde x)| \leq L_J(1 + |x| + |\tilde x|) |x - \tilde x| \,, \label{as:GenLip1} \\
        &J(x) - \underline{J} \leq c_u (1 + |x|^2) \,.\label{as:GenLip2}
    \end{align}
    \item[(1D)] \label{as: J_sup} One of the following holds:
        \begin{enumerate}
            \item[(1D-i)] $J$ is bounded from above:  $\overline{J}:= \sup_{x\in \R^d} J(x)\,,$ or,  \label{eq: J_sup_1}
            \item[(1D-ii)] $J$ admits a polynomial growth: There exists $c_l > 0$ and $ l \geq 0$ such that 
            $$J(x) - \underline{J} \geq c_l ( |x|^l - 1)\,.$$\label{eq: J_sup_2}
        \end{enumerate}
\end{enumerate}
\end{assumption}
\begin{remark}
    Condition \labelcref{as:GenLip1} implies \ref{as: locally_lip}, but the converse is not true. We list them here separately as some well-posedness results require \ref{as: locally_lip} only.
\end{remark}
To analyze the long-time asymptotics of the particles, we require the following assumption which guarantees the existence of a unique minimizer around which $J$ has polynomial growth.
\begin{assumption}\label{ass: J2}
\phantom{}
\begin{enumerate}
 \item[(2A)] There exists $\hat x\in\R^d$ such that $J(\hat x)=\inf_{x\in\R^d} J(x) = \underline J$.
\item[(2B)] \label{as: J_inverse_continuity} There exists $J_{\infty}, R, \eta,\nu > 0$
  such that 
\begin{subequations}
\begin{equation}
    |x - \xhat| \leq (J(x) - \underline{J})^{\nu}/\eta \quad \text{for all } x\in B_R(\xhat)\,,\label{eq: inverse_conti-1}
\end{equation}
\begin{equation}
    J(x) - \underline{J} > J_{\infty} \quad\text{for all } x\in (B_R(\xhat))^{c} \,. \label{eq: inverse_conti-2}
\end{equation}
\end{subequations}
\end{enumerate}
\end{assumption}
Rigorous analysis of convergence of the CBO algorithm has been well studied in previous literatures. We refer the interested reader to \cite{carrillo2018analyticalframeworkconsensusbasedglobal,Fornasier2020ConsensusbasedOO,fornasier2024consensus,huang2023global}. 
In the initial analysis conducted in \cite[Theorem 4.1]{carrillo2018analyticalframeworkconsensusbasedglobal}, the assumptions required for $J$ were more restrictive.
Assuming $J \in \mathcal{C}^2(\R^d)$ and certain bounds on the Hessian and gradient, the authors were able to prove mean-field convergence to a consensus point. 
These assumptions were later relaxed to local Lipschitz continuity in \cite[Theorem 3.7]{fornasier2024consensus}. Our \cref{ass: J2} is consistent with those presented therein.

\subsubsection{Mirror map}\label{sec:mirrormap}

A key component in the MirrorCBO dynamics is a proper and strongly convex function $\phi: \R^d \rightarrow [0,\infty]$, which is used to define the so-called mirror map. This function is often referred to as the distance-generating function, see e.g., \cite{beck2003mirror}.
For well-posedness of our MirrorCBO dynamics, we shall make the following mild assumption:
\begin{assumption}\label{ass: phi0}
The function $\phi:\R^d\to[0,\infty]$ is proper, lower semi-continuous, and strongly convex.
\end{assumption}

For a proper function $\phi: \R^d  \rightarrow [0,\infty]$ we define its convex conjugate or Legendre transform $\phi^*: \R^d \rightarrow [-\infty,\infty]$ via
\begin{align}\label{eq:convex_conjugate}
    \phi^* (y) := \sup_{x\in\R^d}\langle x, y \rangle - \phi(x).
\end{align}

It is a well-known result that strong convexity of $\phi$ implies (in fact, is equivalent) to the Lipschitz continuity of $\nabla\phi^*$.
\begin{lemma}\label{lem: convex_lip}
Suppose that \cref{ass: phi0} is satisfied, where $\phi$ is $\lambda$-strongly convex for some $\lambda>0$, meaning that
\begin{align*}
    \frac{\lambda}{2} \abs{\nabla\phi^*(y)- \hat x}^2
    & \leq 
    \phi(\hat x) - \phi(x) - \langle y, \hat x-x\rangle .
\end{align*}    
Then, $\nabla \phi^*$ is $\frac{1}{\lambda}$-Lipschitz, meaning that it holds
\begin{align*}
    \left|  \nabla \phi^*(y) - \nabla \phi^*(\hat{y}) \right| &\leq \frac{1}{\lambda}\left| y - \hat{y} \right|.
\end{align*}
\end{lemma}
\begin{proof}
    The result can be found, e.g., in \cite[Theorem 1]{zhou2018fencheldualitystrongconvexity}.
\end{proof}
Since $\phi$ is strongly convex, it possesses a unique minimizer which we denote by $\xmin := \argmin_{x\in\R^d}\phi(x)$ and which will make an appearance later.
The subdifferential $\partial\phi:\R^d\rightrightarrows\R^d$ of a proper, convex function $\phi$ is a set-valued map, defined as
\begin{align*}
    \partial\phi(x) := \left\{y\in\R^d\st\phi(x)+\langle y,\hat x-x\rangle\leq\phi(\hat x) \;\forall \hat x\in\R^d\right\},\quad x\in\R^d.
\end{align*}
If $\phi$ is differentiable at $x\in\R^d$, then $\partial\phi(x)=\{\nabla\phi(x)\}$. 
In our method, $\partial\phi$ plays the role of the {mirror map}.
One has the well-known subdifferential inversion formula, see \cite[Cor. 23.5.1]{rockafellar1970convex}
\begin{align*}
    y \in \partial\phi(x) \iff x \in \partial\phi^*(y),\qquad x,y\in\R^d.
\end{align*}
Thanks to \cref{lem: convex_lip} the gradient map $y\mapsto\nabla\phi^*(y)$ is Lipschitz continuous and, in particular, single-valued which implies
\begin{align*}
    y \in \partial\phi(x) \iff x = \nabla\phi^*(y),\qquad x,y\in\R^d
\end{align*}
and $\nabla\phi^* \circ \partial\phi = \operatorname{id}\vert_{\R^d}$.
Since the MirrorCBO method in \labelcref{eq:particle_y} (and in fact also mirror descent as introduced in \cref{sec:mirror_descent}) just depends on $\nabla\phi^*$ and not explicitly on $\partial\phi$, the non-differentiability of $\phi$ is no issue.

We want to emphasize that other works related to mirror descent \cite{nemirovskij1983problem} (see also \cite{bubeck2015convex} for a survey) and its variants including recent studies in mirror Langevin algorithms \cite{ahn2021efficient, Chewi2020ExponentialEO, pmlr-v167-li22b, Hsieh2018MirroredLD} typically assume that $\phi$ is at least differentiable and often even two times continuously differentiable for the sampling approaches.
In our work we will not make use of such regularity assumptions.
While for the mirror Langevin algorithm for sampling these assumptions seem to be unavoidable since they are even required to have a well-defined algorithm, for mirror descent it seems unnecessary to assume differentiability of $\phi$. 

For analyzing the long time behavior of MirrorCBO we shall pose an additional assumption on the function $\phi$ for which we first introduce the notion of a Bregman distance.
\begin{definition}[Bregman distance]\label{def:bregman_distance}
    Given a proper function $\phi: \R^d \rightarrow [0, \infty]$, the Bregman distance of points $\hat x,x \in \R^d$ with respect to $\phi$ and a subgradient $y \in \partial\phi(x)$ is defined as:
    \begin{align*}
        D_{\phi}^y(\hat x, x) :=\phi(\hat x) - \phi(x) - \langle y, \hat x-x\rangle  \,.
    \end{align*}
    \end{definition}
%
The following assumption depends on a point $\hat x\in\R^d$, which we shall later choose as global minimizer of the cost function $J$ from \cref{ass: J2}.
\begin{assumption}\label{ass:bregman_phi}
    Fix $\hat x\in\R^d$. 
    We assume that the function $\phi:\R^d\to[0,\infty]$ is proper, lower semi-continuous, and strongly convex, and that there exist constants $0<m_\phi\leq M_\phi<\infty$ such that 
    \begin{align}\label{eq: bregman_phi}
        m_\phi \abs{\nabla\phi^*(y)- \hat x}^2
        \leq 
        D_\phi^y(\hat x , \nabla\phi^*(y))
        \leq 
        M_\phi
        \abs{\nabla\phi^*(y)- \hat x}^2
        \qquad 
        \forall y\in\R^d.
    \end{align}
\end{assumption}
\begin{remark}
    Taking into account that in the mean field regime CBO methods essentially minimize the quadratic function $x\mapsto\abs{x-\hat x}^2$ (cf. \cite{fornasier2024consensus}), \cref{ass:bregman_phi} corresponds to the usual relative convexity and smoothness assumptions which connects the cost function and the distance generating function in the convergence analysis of mirror descent.
\end{remark}
\begin{remark}\label{rem:bregman_distance}
As noted above, for strongly convex $\phi$ we have that $y \in \partial\phi(\nabla\phi^*(y))$. 
Consequently, it holds
\begin{align}
    D_{\phi}^{y}(\hat x,\nabla\phi^*(y)) =
    \phi(\hat x) - \phi(\nabla\phi^*(y)) - \langle y,\hat x - \nabla\phi^*(y)\rangle.
\end{align}%
\end{remark}
Note that the lower bound in \labelcref{eq: bregman_phi} is equivalent to $2\mphi$-strong convexity of $\phi$, cf. \cref{lem: convex_lip}.
However, the upper bound in \labelcref{eq: bregman_phi} is more restrictive and in the next example we will discuss cases where it is satisfied. 
For a discussion on asymptotic behavior without this upper bound, see \cref{sec:numex}.
We emphasize that \cref{ass:bregman_phi} does not demand any differentiability of $\phi$.
\begin{example}[Choices of $\phi$]\label{ex: bregman_phi}
    We now give a couple of examples in which \cref{ass:bregman_phi} holds true.
    As mentioned above, the lower bound is equivalent to strong convexity which is a standard assumption for distance generating functions in mirror methods.
    So the interesting part is the upper bound.
    \begin{enumerate}
        \item If $\phi(x)=\frac{1}{2}\abs{x}^2$, then \cref{ass:bregman_phi} holds true with $m_\phi=M_\phi=\frac{1}{2}$.
        This is the case of standard CBO.
        \item  Let $\H\subset\R^d$ be a linear subspace and let $\phi(x)=\frac{1}{2}\abs{x}^2+\iota_\H(x)$.
            In this case it holds $\nabla\phi^*(y)=\proj_\H(y)\in \H$ for all $y\in\R^d$. If $\hat x\notin \H$, then \cref{ass:bregman_phi} fails since $ D_\phi^y(\hat x , \nabla\phi^*(y))=+\infty$. If $\hat x\in\H$, we have
        \begin{align*}
            D_\phi(\hat x,\nabla\phi^*(y))
            &=
            \frac{1}{2}\abs{\hat x}^2-\frac{1}{2}\abs{\proj_\H(y)}^2
            -
            \langle y,\hat x - \proj_\H(y)\rangle
            \\
            &=
            \frac{1}{2}\abs{\proj_\H(y)-\hat x}^2
            -\abs{\proj_\H(y)}^2
            +
            \langle \hat x,\proj_\H(y)\rangle
            -
            \langle y,\hat x-\proj_\H(y)\rangle
            \\
            &=
            \frac{1}{2}\abs{\proj_\H(y)-\hat x}^2
            +
            \langle y-\proj_\H(y),\proj_\H(y)\rangle
            +
            \langle\hat x,\proj_\H(y)-y\rangle
            \\
            &=
            \frac{1}{2}\abs{\proj_\H(y)-\hat x}^2
            +
            \langle y-\proj_\H(y),\proj_\H(y)-\hat x\rangle
            \\
            &= 
            \frac{1}{2}\abs{\proj_\H(y)-\hat x}^2
            =
            \frac12\abs{\nabla\phi^*(y)-\hat x}^2,
        \end{align*}
        using orthogonality of the projection.
        Hence if $\hat x\notin \H$, \cref{ass:bregman_phi} holds true with $m_\phi= M_\phi=\frac{1}{2}$.
        \item Let $\phi(x)=\frac{1}{2}\abs{x}^2+\lambda\abs{x}_1$ where $\lambda>0$ and $\abs{x}_1:=\sum_{i=1}^d\abs{x_i}$ is the $1$-norm. 
        Then there exists $\delta>0$ (depending only on $\hat x$ and $d$) such that \cref{ass:bregman_phi} holds with $m_\phi=M_\phi=\frac{1}{2}$ whenever $\abs{\nabla\phi^*(y)-\hat x}<\delta$.
        This was shown in \cite{bungert2022bregman}.
         \item If $\phi\in \mathcal C^2(\R^d)$ is such that $\lambda \I_{d\times d} \preccurlyeq D^2\phi(x) \preccurlyeq \Lambda \I_{d\times d}$ for all $x\in\R^d$, then using Taylor's formula, \cref{ass:bregman_phi} is satisfied with $m_\phi=\lambda/2$ and $M_\phi=\Lambda/2$.
    \end{enumerate}
\end{example}


\section{Well-posedness}
\label{sec:wellposedness}

In this section, we prove the well-posedness of the MirrorCBO method. 
We analyze both the particle-based dynamics and the mean-field equation, focusing on the existence and uniqueness of solutions for \labelcref{eq:particle_y,eq: meanfield-mirrorcbo-sde}. 
The results will build upon previous well-posedness results for CBO methods~\cite{carrillo2018analyticalframeworkconsensusbasedglobal, fornasier2024consensus}.

Before discussing further, let us briefly comment on each of the assumptions from \cref{sec:assumptions} in relation to the well-posedness of \labelcref{eq:particle_y,eq: meanfield-mirrorcbo-sde}. First, \mbox{\ref{as: locally_lip}} is required for the well-posedness of particle system \labelcref{eq:particle_y}, while \ref{as: J_inf},~\ref{as: generalized_locally_lipschitz}, and~\ref{as: J_sup} are required for the well-posedness of the mean-field process \labelcref{eq: meanfield-mirrorcbo-sde}. 
\cref{ass: phi0} is required for well-posedness of both \labelcref{eq:particle_y,eq: meanfield-mirrorcbo-sde}.

\subsection{Well-posedness of the particle system}

Well-posedness for the interacting particle system of standard CBO \labelcref{eq: cbo} was shown in \cite[Theorem 2.1]{carrillo2018analyticalframeworkconsensusbasedglobal} by proving that the drift and diffusion term $x_t^{(i)} - m_{\alpha}(\rho_t^N)$ and $|x_t^{(i)} - m_{\alpha}(\rho_t^N)|$ are locally Lipschitz and have a linear growth property.
This allows one to apply classical arguments for strong solutions of stochastic differential equations, see, e.g., \cite[Theorem 3.1]{durrett1996stochastic} or \cite[Theorem 3.5]{khasminskii_stochastic_2011}.

Our well-posedness proof follows the same lines of argumentation; we first check that the drift and noise coefficients of \labelcref{eq:particle_y} are locally Lipschitz and grow at most linearly. 
In what follows, we denote elements of $\R^{dN}$ by capital letters (e.g., $Y$) and $\normF{Y}$ refers to their Frobenius norm. 
\begin{lemma}\label{lem: condition_check_wellposedenss_particle}
    Suppose that \mbox{\ref{as: locally_lip}} holds and let $\phi: \R^d \rightarrow [0,\infty]$ satisfy \cref{ass: phi0}. 
    Fix $N \in \N$, and $\alpha > 0$. 
    Let $Y:= (y^{(1)}, y^{(2)}, \cdots, y^{(N)}) \,, \hat{Y}:=  (\hat{y}^{(1)}, \hat{y}^{(2)}, \cdots, \hat{y}^{(N)}) \in \R^{dN}\,,$ and define
    \begin{align}\label{eq:mirror_drift}
        b^{(i)}(Y):= - \left(\nabla \phi^*(y^{(i)}) - m_{\alpha}^*(\mu^N)\right) \,,
    \end{align}
    where $\mu^N = \frac{1}{N}\sum_{i = 1}^N \delta_{y^{(i)}}\,.$ 
    If $\normF{Y},\normF{\hat{Y}} \leq k$ for some $k<\infty
    $, then the following holds:
    \begin{subequations}
    \begin{align}
        &|b^{(i)}(Y)| \leq L\left(|y^{(i)}| + \normF{Y}\right)\,, \label{eq: linear_mirrorcbo_particle}\\ 
        &|b^{(i)}(Y) - b^{(i)}(\hat{Y})| \leq L|y^{(i)} - \hat{y}^{(i)}| + c_{b, i}(Y) \normF{Y - \hat{Y}} \label{eq: lipschitz_mirrorcbo_particle}\,,
    \end{align}
    \end{subequations}
    where $L =\Lip(\nabla \phi^*)$, 
       \begin{align*}
        c_{b, i}(Y)
        &:=  
        L\left( 1 + \frac{\left( 1 + \sqrt{2} \right) c_k   }{N} \sqrt{N\left(L |\hat{y}^{(i)}| + |\xmin|\right)^2 + \left(L\normF{\hat{Y}}  +  |\xmin|\sqrt{N}\right)^2}   \right) 
        \text{ with }\\
        c_k &:= 
        \alpha\|\nabla J\|_{L^{\infty}(B_{k})} \exp\left(\alpha \|J - \underline{J}\|_{L^{\infty}(B_{k})}\right)\,,
    \end{align*}
    and $B_{k}=\{y\in\R^d\st\abs{y}\leq L k + \abs{\xmin}\}$.
\end{lemma}
\begin{proof}
    The proof is very similar to \cite[Lemma 2.1]{carrillo2018analyticalframeworkconsensusbasedglobal}.
    For completeness we elaborate on the details.
    
    Thanks to the strong convexity of $\phi$ from \cref{ass: phi0} there exists $L := \Lip(\nabla\phi^*) < + \infty$ such that $\left| \nabla \phi^*(y) - \nabla \phi^*(\hat{y})\right| \leq L|y - \hat{y}|$ for all $y, \hat{y} \in \R^{d}$. 
    For \labelcref{eq: linear_mirrorcbo_particle} we use this to compute
    \begin{align*}
       \abs{b^{(i)}(Y)}
       &=
       \abs{\nabla\phi^*(y^{(i)}) - m_\alpha^*(\mu^N)}
       \\
       &=
       \frac{\abs{\sum_{j=1}^N\exp\left(-\alpha J\circ\nabla\phi^*(y^{(j)})\right)(\nabla\phi^*(y^{(i)})-\nabla\phi^*(y^{(j)}))}}{\sum_{j=1}^N\exp(-\alpha J\circ\nabla\phi^*(y^{(j)}))}
       \\
       &\leq 
       L
       \frac{\sum_{j=1}^N\exp\left(-\alpha J\circ\nabla\phi^*(y^{(j)})\right)\abs{y^{(i)}-y^{(j)}}}{\sum_{j=1}^N\exp(-\alpha J\circ\nabla\phi^*(y^{(j)}))}
       \leq 
       L\left(
       |y^{(i)}|
       +
       \normF{Y}
       \right)
    \end{align*}
For \labelcref{eq: lipschitz_mirrorcbo_particle}, we reformulate $b^{(i)}(Y) - b^{(i)} (\hat{Y})$ as follows:
\begin{align*}
   b^{(i)}(Y) - b^{(i)} (\hat{Y}) &= \frac{\sum_{j=1}^{N} \left( \nabla \phi^*(y^{(i)}) - \nabla \phi^*(y^{(j)}) \right) w_{\alpha}^{*} (y^{(j)}) }{ \sum_j  w_{\alpha}^{*} (y^{(j)} )} - \frac{\sum_{j=1}^{N} \left( \nabla \phi^*( \hat{y}^{(i)}) - \nabla \phi^*(\hat{y}^{(j)}) \right) w_{\alpha}^{*}( \hat{y}^{(j)} )}{ \sum_j  w_{\alpha}^{*} (\hat{y}^{(j)} )} = \sum_{l = 1}^3 A_l \,,
\end{align*}
where 
\begin{align*}
&A_1:= \frac{ \sum_{j=1}^{N} \left( \nabla \phi^*(y^{(i)}) - \nabla \phi^*( \hat{y}^{(i)}) + \nabla \phi^*( \hat{y}^{(j)}) - \nabla \phi^*( y^{(j)})  \right) w_{\alpha}^{*}(y^{(j)}) }{ \sum_j w_{\alpha}^{*} (y^{(j)})  } \,, \\
&A_2:= \frac{ \sum_{j=1}^{N} \left( \nabla \phi^*( \hat{y}^{(i)}) - \nabla \phi^*( \hat{y}^{(j)}) \right) \left(  w_{\alpha}^{*}(y^{(j)}) -  w_{\alpha}^{*}(\hat{y}^{(j)}) \right) }{ \sum_j w_{\alpha}^{*} (y^{(j)})  } \,, \\
&A_3:= \sum_{j=1}^{N} \left(\nabla \phi^*(\hat{y}^{(i)}) - \nabla \phi^*(\hat{y}^{(j)}) \right) w_{\alpha}^{*}(\hat{y}^{(j)}) \frac{\sum_j \left( w_{\alpha}^{*}(\hat{y}^{(j)}) - w_{\alpha}^{*}(y^{(j)}) \right) }{  \sum_j w_{\alpha}^{*}(y^{(j)}) \sum_j w_{\alpha}^{*}(\hat y^{(j)}) } \,.
\end{align*}
By definition of the convex conjugate and convexity of $\phi$, we have $\nabla\phi^*(0)=\argmin \phi = \xmin$.
Using this with the Lipschitz continuity of $\nabla\phi^*$ we get
\begin{align*}
    |\nabla \phi^*(y)| 
    &\leq 
    |\nabla \phi^*(y) - \nabla \phi^*(0)| +  |\nabla \phi^*(0)| \leq L | y| + | \xmin |\qquad\forall y\in\R^d\,.
\end{align*} 
Now we denote $x^{(i)}:=\nabla\phi^*(y^{(i)})$, $\hat x^{(i)}:=\nabla\phi^*(\hat{y}^{(i)})$, $X:=(x^{(1)},\dots,x^{(N)})$, $\hat X:=(\hat x^{(1)},\dots,\hat x^{(N)})$, and observe that $\normF{Y}\leq k$ implies $\abs{y^{(i)}}\leq k$ and hence $\abs{x^{(i)}}\leq Lk+\abs{\xmin}$ for all $i=1,\dots,N$.
An analogous bound holds for $\hat x^{(i)}$.
Hence, with this change of variables we get verbatim as in the proof of \cite[Lemma 2.1]{carrillo2018analyticalframeworkconsensusbasedglobal},
\begin{align*}
    &|A_1| \leq |x^{(i)} - \xhat^{(i)}| + \normF{X - \hat{X}}\,,\\
    &|A_2| \leq \frac{\sqrt{2}c_k }{N} \normF{X - \hat{X}} \sqrt{N|\xhat^{(i)}|^2 +\normF{\hat{X}}^2 } \,, \\
    &|A_3| \leq \frac{c_k}{N} \normF{X - \hat{X}} \sqrt{N|\xhat^{(i)}|^2 +  \normF{\hat{X}}^2  } \,.
\end{align*}
Using the Lipschitz continuity of $\nabla\phi^*$ and Hölder's inequality for sums it is easy to show that $\normF{\hat X}\leq L \normF{\hat Y} + \abs{\xmin}\sqrt{N}$.
Furthermore, it also holds $\normF{X-\hat X}\leq L\normF{Y-\hat Y}$.
Using these two estimates it follows
\begin{align*}
    &|A_1| \leq L|y^{(i)} - \hat{y}^{(i)}| + L \normF{Y - \hat{Y}}\,,\\
    &|A_2| \leq \frac{\sqrt{2}c_k L }{N} \normF{Y - \hat{Y}} \sqrt{N\left(L\abs{y^{(i)}}+\abs{\xmin}\right)^2 + \left(L\normF{Y} + \sqrt{N}|\xmin|\right)^2 } \,, \\
    &|A_3| \leq \frac{c_k L }{N} \normF{Y - \hat{Y}} \sqrt{N\left(L\abs{y^{(i)}}+\abs{\xmin}\right)^2 + \left(L\normF{Y} + \sqrt{N}|\xmin|\right)^2 } \,.
\end{align*}
Combining, we have \labelcref{eq: lipschitz_mirrorcbo_particle} as claimed.
\end{proof}
Using \cref{lem: condition_check_wellposedenss_particle} we can use standard results to infer the well-posedness of the particle evolution for MirrorCBO when the number of particles is fixed. 
Furthermore, we also have (exponential) bounds for higher order moments of the solution. 
\begin{theorem}[Well-posedness of the interacting particle system]\label{thm: well-posedness-particle-mirror}
    Suppose that $J$ satisfies \mbox{\ref{as: locally_lip}}, assume that $\phi: \R^d \rightarrow [0,\infty]$ satisfies \cref{ass: phi0}, and let $N\in\N$ be fixed.
    Then, given any initial configuration $Y_0:= (y_0^{(i)})_{i = 1}^N \in \R^{dN}$ such that $\E|Y_0|^{2n} < \infty$ for some $n\in\N_0$, there exists a pathwise unique strong solution $Y_t:= (y_t^{(i)})_{i = 1}^N$ for all $t \geq 0$ to the MirrorCBO dynamics \labelcref{eq:particle_y} with $\E|Y_t|^{2n} < \infty$ for all $t\geq 0$.  
\end{theorem}
\begin{proof}  
We begin by rewriting \labelcref{eq:particle_y} as
\begin{align}\label{eq: particle_system_mirrorcbo}
    &\d Y_t= B_N(Y_t)\de t + \sqrt{2\sigma^2} \Sigma_N(Y_t) \de W_t\,\qquad\text{ where} \\
    &B_N(Y):= (b^{(1)}(Y), \cdots, b^{(N)}(Y))^T\,, \notag \\
    &\Sigma_N(Y):= \operatorname{diag}(|b^{(1)}(Y)|I_d, \cdots, |b^{(N)}(Y)|I_d) \notag 
\end{align}
and the drift term $b^{(i)}(Y)$ is defined as in \labelcref{eq:mirror_drift}.

From \labelcref{eq: linear_mirrorcbo_particle,eq: lipschitz_mirrorcbo_particle} in \cref{lem: condition_check_wellposedenss_particle} we conclude that the drift and diffusion terms $B_N$ and $\Sigma_N$ are locally Lipschitz continuous and satisfy a linear growth condition. 
By standard results like \cite[Theorems 3.1, 3.2, Chapter 5]{durrett1996stochastic} or \cite[Theorem 3.5]{khasminskii_stochastic_2011} we can conclude existence of a unique solution if we find $C_N > 0$ and $\Psi \in \mathcal{C}^2(\R^{dN}, [0, \infty)) $ such that 
\begin{align*}
    &\Psi(Y) \rightarrow \infty \text{ as } |Y| \rightarrow \infty\,,\\
    &\langle B_N(Y), \nabla \Psi(Y)\rangle + 2\sigma^2 \tr\left( \Sigma(Y)^T D^2 \Psi(Y) \Sigma(Y) \right) \leq C_N \Psi(Y)\,.
\end{align*}
    Taking $\Psi(Y) = \frac12\|Y\|_F^2$ for $Y\in\R^{dN}$ and denoting $\psi(y)=\frac{1}{2}\abs{y}^2$ for $y\in\R^d$ we have following inequalities: 
\begin{align*}
    \langle\nabla \psi (y^{(i)}), b^{(i)} (Y) \rangle
    &= 
    -
    \left\langle 
    y^{(i)},
    \frac{\sum_{j=1}^{N} (\nabla \phi^*(y^{(i)}) - \nabla \phi^*(y^{(j)})) w_{\alpha}^*(y^{(j)}) }{ \sum_j w_{\alpha}^*(y^{(j)})}
    \right\rangle
    \\
    &\leq |y^{(i)}| \frac{\sum_{j=1}^{N} L|y^{(i)} - y^{(j)}| w_{\alpha}^*(y^{(j)})}{ \sum_j w_{\alpha}^*(y^{(j)})} \leq L\left( |y^{(i)}|^2 + |y^{(i)}|\normF{Y} \right)\,,\\
    |b^{(i)}(Y)|^2 
    &= 
    \left| \frac{\sum_{j=1}^{N} (\nabla \phi^*(y^{(i)}) - \nabla \phi^*(y^{(j)})) w_{\alpha}^*(y^{(j)}) }{ \sum_j w_{\alpha}^*(y^{(j)}) } \right|^2 \leq 2L^2\left( |y^{(i)}|^2 + \normF{Y}^2 \right)\,.
\end{align*}
Therefore, we have 
\begin{align*}
\langle B_N(Y), \nabla \Psi(Y)\rangle + 2\sigma^2 \tr\left( \Sigma(Y)^T D^2 \Psi(Y) \Sigma(Y) \right) & \leq \sum_i \left[ L\left( |y^{(i)}|^2 + |y^{(i)}|\normF{Y} \right) + 4\sigma^2 L^2 d\left( |y^{(i)}|^2 + \normF{Y}^2 \right) \right] \\
& \leq 2L\left(1 + \sqrt{N} + 4\sigma^2 L d(1 + N) \right) \frac12\normF{Y}^2 = C_N \Psi(Y)\,.
\end{align*}
Finally, using the linear growth bound \labelcref{eq: linear_mirrorcbo_particle} from \cref{lem: condition_check_wellposedenss_particle} we can apply \cite[Theorem 7.1.2]{arnold1974stochastic} (upon noticing that this result does not depend on the Lipschitz constant of the drift but only on the linear growth constant) to obtain the existence of a constant $C_{n,N} < + \infty$ such that 
\begin{align*}
    \E |Y_t|^{2n} \leq \left(1+\E |Y_0|^{2n}\right)
    \exp\left(C_{n,N} t \right)  \,.
\end{align*}
\end{proof}
\begin{remark}\label{rmk: aniso_wellposed_particle}
    For the anisotropic noise case, choosing $S(u) = \operatorname{diag}(u)$ in \labelcref{eq: mcbo_general}, we have instead
    \begin{align*}
        \Sigma_N(Y):= \operatorname{diag} \left(\operatorname{diag}(b^{(1)}(Y)), \cdots, \operatorname{diag}(b^{(N)}(Y)) \right)\,,
    \end{align*}
    and thus 
    \begin{align*}
        \tr\left( \Sigma(Y)^T D^2 \Psi(Y) \Sigma(Y) \right) \leq 2 L^2 ( 1+N ) \normF{Y}^2\,.
    \end{align*}
    In particular, $C_{N}$ is no longer dependent on the dimension $d$. 
    For a general noise model we have
    \begin{align*}
        \Sigma_N(Y):= \operatorname{diag} \left(S(b^{(1)}(Y)), \cdots, S(b^{(N)}(Y)) \right)\,,
    \end{align*}
    with a function $S:\R^d\to\R^{d\times d}$ that has linear growth, i.e., $\tr\left(S(u)^T S(u)\right)\leq C_S\abs{u}^2$ for some constant $C_S>0$, in which case we get
    \begin{align*}
        \tr\left( \Sigma(Y)^T D^2 \Psi(Y) \Sigma(Y) \right) \leq C_S \sum_{i=1}^N\abs{b^{(i)}(Y)}^2 
        \leq 
        2 C_S L^2(1+N)\normF{Y}^2.
    \end{align*}
    As a result, we also obtain well-posedness of the particle system in this case applying the same arguments as in the proof of \cref{thm: well-posedness-particle-mirror}.
\end{remark}
\subsection{Well-posedness of the mean-field system}

Well-posedness of the mean field process for standard CBO was investigated in \cite[Theorems 3.1 and 3.2]{carrillo2018analyticalframeworkconsensusbasedglobal}. 
Here, we extend this theory to the MirrorCBO setting relying on the Lipschitz continuity of $y\mapsto\nphiy$ that follows from \cref{ass: phi0}. 
Denoting
\begin{align*}
    m_\alpha^*(\mu) = \frac{\int_{\R^d}  w_{\alpha}^{*}(y) \nabla \phi^*(y)  \de\mu(y)}{\| w_{\alpha}^{*}\|_{L^1(\mu)} }\,,
\end{align*}
the mean-field process corresponding to \labelcref{eq:particle_y} is the McKean--Vlasov process \labelcref{eq: meanfield-mirrorcbo-sde}, given by
\begin{align*}
        \d y_t= -\left(\nabla\phi^\ast(y_t) - m_\alpha^*(\mu_t)\right)\d t  + \sqrt{2\sigma^2}\abs{\nabla\phi^\ast(y_t)-m_\alpha^*(\mu_t)} \d W_t
\end{align*}
where $\mu_t := \operatorname{Law}(y_t)$ satisfies the associated nonlocal and nonlinear Fokker--Planck equation~\labelcref{eq: meanfield-mirrorcbo-pde}:
\begin{align*}
    \partial_t\mu_t(y) &= \div\Big(\mu_t(y)\left(\nabla\phi^\ast(y)-m_{\alpha}^{*}(\mu_t)\right)\Big) + \sigma^2\Delta\left(\mu_t(y)\abs{\nabla\phi^\ast(y)-m_{\alpha}^{*}(\mu_t)}^2\right).
\end{align*}
First, we define weak solutions of the Fokker--Planck equation by shifting all spatial derivatives onto test functions.
\begin{definition}\label{def:weak_solutions}
    We say that $\mu_t \in \mathcal C([0,T],\mathcal P(\R^d))$ is a weak solution to the Fokker--Planck equation \labelcref{eq: meanfield-mirrorcbo-pde} if the following holds for all $\xi \in \mathcal C_{c}^{\infty}(\R^{d})$:
    \begin{align}\label{eq:weak_form_fp}
    \begin{split}
        \frac{\d}{\d t}\int_{\R^d}\xi(y)\de\mu_t(y) 
        &= 
        - \int_{\R^d}\langle \nabla \xi(y), \nabla \phi^{*}(y) - m_\alpha^*(\mu_t)\rangle \de\mu_t(y) 
        \\
        &\qquad 
        + \sigma^2 \int_{\R^d} \Delta \xi(y) \left|\nabla \phi^{*}(y) - m_\alpha^*(\mu_t)\right|^2\de\mu_t(y)\,.
    \end{split}
    \end{align}
\end{definition}
The following stability estimate of the weighted average in terms of the \mbox{Wasserstein-$2$} distance is one of the key elements for well-posedness of~\labelcref{eq:weak_form_fp} in \cref{thm: wellposedness-pde}.
It follows from the Lipschitz continuity of $\nabla\phi^*$ and the corresponding result for standard CBO.
\begin{lemma}[Wasserstein stability estimate]\label{lem: stability_weightedmean}
   Fix $T > 0$. Suppose that $J$ satisfies \ref{as: J_inf}, \ref{as: generalized_locally_lipschitz}, and $\phi$ satisfies \cref{ass: phi0}. Then for $\mu, \hat{\mu} \in \mathcal{P}_4(\R^d)$ with $\int_{\R^d}\abs{y}^4\de\mu,\int_{\R^d}\abs{y}^4\de\hat\mu\leq K$ it holds 
   \begin{align}\label{eq: Wasserstein_stability}
       |m_{\alpha}^*(\mu) - m_{\alpha}^*(\hat{\mu})| \leq c_m W_2(\mu, \hat{\mu} )\,,
   \end{align}
   where $c_m = c_K L \left( 1 + \alpha L_J (1 + c_K ) p_K \right)$, $L = \Lip(\nabla \phi^*)$, $c_K:= \exp\left( \alpha c_u \left(1 + \sqrt{4L^4K+4|x_\phi|^4}\right) \right)$, $L_J$ is the local Lipschitz constant of $J$ from \ref{as: generalized_locally_lipschitz}, and $p_K:=2\sqrt{4L^4K+4|x_\phi|^4}$.
\end{lemma}
\begin{proof}
We have that
\begin{align*}
    m_\alpha^*(\mu) - m_\alpha^*(\hat\mu)
    =
    m_\alpha(\rho) - m_\alpha(\hat\rho)
\end{align*}
where $\rho=(\nabla\phi^*)_\sharp\mu$ and $\hat\rho=(\nabla\phi^*)_\sharp\hat\mu$. 
Let $\gamma\in\Pi(\mu,\hat\mu)$ be an arbitrary coupling, define $\pi :=(\nabla\phi^*\times\nabla\phi^*)_\sharp\gamma$.
Further, note that $\int_{\R^d}\abs{y}^4\d\mu,\int_{\R^d}\abs{y}^4\d\hat\mu\leq K$ and $\nabla\phi^*(0)=x_\phi$ implies 
\begin{align*}
    &\int_{\R^d}\abs{x}^4\d\rho = \int_{\R^d}\abs{\nabla\phi^*(y)}^4\d\mu
    \leq \int_{\R^d}\left(\abs{\nabla\phi^*(y)-\nabla\phi^*(0)}+\abs{x_\phi}\right)^4\d\mu\\
    &\qquad \leq \int_{\R^d}\left(L\abs{y}+\abs{x_\phi}\right)^4\d\mu
    \le 4L^4K+4|x_\phi|^4\,,
\end{align*}
with the same bound for $\int_{\R^d}\abs{x}^4\de\hat\rho$,
thanks to \cref{ass: phi0}.
It is easy to check that $\pi\in\Pi(\rho,\hat\rho)$ and following the proof of \cite[Lemma 3.2]{carrillo2018analyticalframeworkconsensusbasedglobal} we obtain
\begin{align*}
    \abs{m_\alpha^*(\mu) - m_\alpha^*(\hat\mu)}
    &=
    \abs{m_\alpha(\rho) - m_\alpha(\hat\rho)}
    \\
    &\leq 
    c_K \left( 1 + \alpha L_J (1 + c_K) p_K \right)
    \left(
    \iint\abs{x-\hat x}^2\de\pi(x,\hat x)
    \right)^\frac{1}{2},
    \\
    &=    
    c_K \left( 1 + \alpha L_J (1 + c_K) p_K \right)
    \left(
    \iint\abs{\nabla\phi^*(y)-\nabla\phi^*(\hat y)}^2\de\gamma(y,\hat y)
    \right)^\frac{1}{2}
    \\
    &\leq 
    L
    c_K \left( 1 + \alpha L_J (1 + c_K) p_K \right)
    \left(
    \iint\abs{y-\hat y}^2\de\gamma(y,\hat y)
    \right)^\frac{1}{2}\,.
\end{align*}
Taking the infimum over all couplings $\gamma\in\Pi(\mu,\hat\mu)$ we get the desired result.

\end{proof}
Using this stability result and a fixed-point argument, we can prove well-posedness of the mean-field equation for initial conditions with finite fourth moment. 
\begin{theorem}[Well-posedness of the mean-field system]\label{thm: wellposedness-pde}
    Suppose that $J$ satisfies \ref{as: J_inf},~\ref{as: generalized_locally_lipschitz}, and \ref{as: J_sup}. Furthermore, suppose that $\phi: \R^d \rightarrow [0,\infty]$ satisfies \cref{ass: phi0}.
    Then, for fixed $T>0$ and given $\mu_0 \in \mathcal{P}_4(\R^d)$, there exists a unique strong solution $y_t \in \mathcal{C}([0, T], \R^d)$ to \labelcref{eq: meanfield-mirrorcbo-sde}, where $\mu_t := \operatorname{Law}(y_t)\in \mathcal{C}([0, T], \mathcal{P}_4(\R^d))$ is a weak solution to \labelcref{eq: meanfield-mirrorcbo-pde} in the sense of \cref{def:weak_solutions}.

    Moreover, $\mu_t$ also satisfies \labelcref{eq:weak_form_fp} in \cref{def:weak_solutions} for all test functions 
    $\xi \in \mathcal C^{1,1}_*(\R^d)$ and almost every time $t>0$, where
    \begin{align*}
        \mathcal C^{1,1}_*(\R^d):=\left\{\xi \in \mathcal C^{1,1}(\R^d)\st
        \exists C\geq 0,\,
        \abs{\nabla \xi(y)}\leq C\left(1+\abs{y}\right)\;\forall y\in\R^d,\;\abs{\Delta\xi}\leq C\;\text{a.e. in }\Omega\right\}\,.
    \end{align*}
\end{theorem}
\begin{proof}
    We follow the proof of \cite[Theorem 3.1]{carrillo2018analyticalframeworkconsensusbasedglobal} and give a proof sketch of each step therein. 
    The main idea is to use the Leray--Schauder fixed point theorem for an auxiliary SDE which gives rise to a linear Fokker--Planck equation. 
    Here we focus on the first case (i) in \ref{as: J_sup}, meaning that we assume that the objective $J$ is bounded.
    The case (ii) of polynomial growth is handled just like in \cite[Theorem 3.2]{carrillo2018analyticalframeworkconsensusbasedglobal}.
\\

    \fbox{Step 1 (Auxiliary SDE)} 
     Let $\mu_0 \in \mathcal{P}_4(\R^d)\,, \text{ and } u \in \mathcal{C}([0, T], \R^d)\,,$ be given.
    By a similar argument as in \cref{thm: well-posedness-particle-mirror}, we can apply standard results to obtain the existence of a unique solution $y_t \in \mathcal{C}([0, T], \R^d)$ to 
    \begin{align}\label{eq: aux_sde}
        \d y_t = - (\nabla \phi^*(y_t) - u_t ) \de t + \sqrt{2\sigma^2} |\nabla \phi^*(y_t) - u_t | \de W_t
    \end{align} 
    with $y_0 \sim \mu_0$.
    Furthermore by the regularity of $y_t$ and using It\^o's lemma, the law $\nu_t = \operatorname{Law}(y_t)$ satisfies $ \nu \in \mathcal{C}([0, T], \mathcal{P}_2(\R^d))$ and, furthermore, for all $\xi \in \mathcal{C}_{c}^\infty
    (\R^d)$ and all $t>0$ it holds
    \begin{align}\label{eq: aux_pde}
        \frac{\d}{\d t} \int_{\R^d} \xi \de\nu_t = \int_{\R^d} \left( -\langle\nabla \phi^*(y) - u_t, \nabla \xi\rangle + \sigma^2 |\nabla \phi^*(y) - u_t|^2 \Delta \xi \right) \de\nu_t\,.
    \end{align}
    Since $\mu_0 \in \mathcal{P}_4(\R^d)$, and $\phi$ is strongly convex, we can invoke classical SDE theory \cite[Chapter 7]{arnold1974stochastic} to obtain that the solution $y_t$ to \labelcref{eq: aux_sde} satisfies $\E|y_t|^4 \leq (1 + \E|y_0|^4) e^{ct}$ for some $ c< \infty$. 
Therefore, there exists $K <  \infty$ such that  \begin{align}\label{eq: y_moment_bound}
    \sup_{t \in [0, T]} \int_{\R^d} |y|^4 \de\nu_t \leq K\,.
\end{align} 
In particular, this shows that $\nu_t\in\mathcal P_4(\R^d)$ for all $t\in[0,T]$.
We will now verify that we can use less regular test functions $\xi\in \mathcal C^{1,1}_*(\R^d)$ and still get \labelcref{eq: aux_pde} for almost every $t>0$.
For this let $\xi\in \mathcal C^{1,1}_*(\R^d)$ and note that in particular $\xi \in W^{2,p}_\mathrm{loc}(\R^d)$ for every $1\leq p<\infty$.
Convolving $\xi$ with a smooth mollifier we obtain a sequence of smooth functions converging to $\xi$ uniformly and also in $W^{2,p}_\mathrm{loc}(\R^d)$ for all $1\leq p<\infty$.
Then It\^o's formula holds in integral form, meaning that almost surely and for all $t>0$ we have
\begin{align*}
    \xi(y_t) - \xi(y_0)
    &=
    -\int_0^t 
    \langle 
    \nabla\xi(y_s)
    , 
    \nabla\phi^*(y_s)-u_s
    \rangle 
    \de s
    +
    \sqrt{2\sigma^2}
    \int_0^t
    \abs{\nabla\phi^*(y_s)-u_s}
    \langle
    \nabla\xi(y_s)
    ,
    \de W_s
    \rangle
    \\
    &\qquad
    +
    \sigma^2
    \int_0^t 
    \Delta\xi(y_s)
    \abs{\nabla\phi^*(y_s)-u_s}^2
    \de s.
\end{align*}
Using $\xi\in \mathcal C^{1,1}_*(\R^d)$ as well as $\nu_t\in\mathcal P_4(\R^d)$ and the strong convexity of $\phi$, the second integral vanishes in expectation by \cite[(4.4.14) e)]{arnold1974stochastic}.
Applying also Fubini's theorem one therefore obtains for all $t\geq 0$ that
\begin{align*}
    \int_{\R^d} \xi\de(\nu_t-\nu_0)
    =
    \int_0^t 
    \left[-
    \int_{\R^d} \langle
    \nabla\xi(y),
    \nabla\phi^*(y)-u_s
    \rangle
    \de\nu_s(y)
    +
    \sigma^2
    \int
    \Delta\xi(y)
    \abs{\nabla\phi^*(y)-u_s}^2
    \de\nu_s(y)
    \right]
    \de s.
\end{align*}
This shows that the function $t\mapsto\int_{\R^d}\xi\de\nu_t$ is absolutely continuous on $[0,\infty)$ and therefore for almost every $t>0$ we have that \labelcref{eq:weak_form_fp} holds.

    \fbox{Step 2 (Compact Map)}
    Note that $m_{\alpha}^*(\nu_t) \in \mathcal{C}([0, T], \R^d)$. We define a map $\T: \mathcal{C}([0, T], \R^d) \rightarrow \mathcal{C}([0, T], \R^d)$ as $\T (u):= m_{\alpha}^*(\nu)$, where $\nu\in\mathcal C\left([0,T],\mathcal P_4(\R^d)\right)$ solves the auxiliary SDE \labelcref{eq: aux_sde}. We will prove that this map is compact. 
    
    For this we remind the reader that $y \mapsto \nabla \phi^*(y)$ is $L$-Lipschitz continuous for some $L > 0$ thanks to the strong convexity of $\phi$ from \cref{ass: phi0} and that we denote $\xmin:= \argmin_{x \in \R^d} \phi(x)$.
    Using It\^o's isometry, Jensen's inequality, and the Lipschitz continuity we have for $0 < s < t < T$:
    \begin{align}\label{eq: y_compact_embedding}
        \E|y_t - y_s|^2
        &= 
        \E\left[
        \abs{
        -\int_s^t(\nabla\phi^*(y_\tau)-u_\tau)\de\tau 
        +
        \sqrt{2\sigma^2} 
        \int_s^t
        \abs{\nabla\phi^*(y_\tau)-u_\tau}
        \de W_\tau
        }^2
        \right] 
        \notag
        \\
        &= 
        \E\left[
        \abs{\int_s^t(\nabla\phi^*(y_\tau)-u_\tau)\de\tau}^2
        \right] 
        +
        2\sigma^2 
        \E\left[
        \abs{\int_s^t
        \abs{\nabla\phi^*(y_\tau)-u_\tau}
        \de W_\tau}^2
        \right] 
        \notag
        \\
        &\leq 
        \abs{t-s}
        \E\left[
       \int_s^t \abs{\nabla\phi^*(y_\tau)-u_\tau}^2\de\tau
        \right] 
        +
        2\sigma^2 
        \E\left[
        \int_s^t
        \abs{\nabla\phi^*(y_\tau)-u_\tau}^2
        \de \tau 
        \right] 
        \notag
        \\
        &\leq  
        \left(T+2\sigma^2\right)
        \E\left[\int_s^t 3\left(L^2\abs{y_\tau}^2 + \abs{\xmin}^2 + \abs{u_t}^2\right)  d\tau\right]
        \notag
        \\
        &\leq 
        3\left(T+2\sigma^2\right)
        \left(
        L^2 \sqrt{K}
        +
        \abs{\xmin}^2
        +
        \norm{u}_\infty^2
        \right)
        \abs{t-s}
        =: c^2\abs{t-s}
    \end{align}
    with $\norm{u}_\infty^2:=\sup_{t\in[0,T]}|u_t|$.
    Therefore we have $W_2(\nu_t, \nu_s) \leq c|t - s|^{\frac{1}{2}}$. 
    Using the Wasserstein stability estimate from \cref{lem: stability_weightedmean} we get 
    \begin{align*}
        |m_{\alpha}^*(\nu_t) - m_{\alpha}^*(\nu_s)| \leq c_mW_2(\nu_t, \nu_s) \leq c_mc|t - s|^{\frac{1}{2}}\,.
    \end{align*} 
    Therefore, $m_{\alpha}^*(\nu_t)$ is H\"older continuous with exponent $1/2$. By the compact embedding of $\mathcal{C}^{0, 1/2}([0, T], \R^d) \hookrightarrow \mathcal{C}([0, T], \R^d)\,,$ we can conclude that $\T$ is compact. 
    Note that $\mathcal T$ is also a continuous map from $\mathcal C([0,T],\R^d)$ to itself which is an immediate consequence of the Wasserstein-stability result \cref{lem: stability_weightedmean} and the stability of \labelcref{eq: aux_sde} with respect to the $u_t$-variable. 
    Showing the latter follows similarly to the uniqueness proof in Step 4 below and we leave the details to the reader.

\fbox{Step 3 (Existence of a fixed point)} 
Fix $\tau \in [0, 1]$ and consider $u \in \mathcal{C}([0, T], \R^d)$ that satisfies $u = \tau \T u $.
In particular, this implies the existence of $\mu $ solving \labelcref{eq: aux_pde} with $u=\tau m_\alpha^*(\mu)$. For the purpose of showing that $u$ is uniformly bounded, we observe 
\begin{align}\label{eq: u-y-bound}
    |u_t|^2 = \tau^2 |m_{\alpha}^*(\mu_t)|^2 \leq \tau^2 e^{\alpha (\overline{J} - \underline{J})} \int_{\R^d} |\nabla\phi^*(y)|^2 \de\mu_t
    \le \tau^2 2L^2 e^{\alpha (\overline{J} - \underline{J})} \left(1 + \int_{\R^d} |y|^2 \de\mu_t\right)
    \,.
\end{align}
Using \labelcref{eq: y_moment_bound} as well as $\mathcal P_4(\R^d) \subset \mathcal P_2(\R^d)$, we conclude that $u$ is uniformly bounded. 
Hence, the Leray--Schauder fixed point theorem implies the existence of a fixed point $u \in \mathcal C([0,T],\R^d)$ such that $\T(u) = u = m_{\alpha}^*(\mu_t)$, where $\mu$ is the corresponding solution of \labelcref{eq: aux_pde}. In particular, \labelcref{eq: y_moment_bound} implies
\begin{align}\label{eq: step3-bounds}
    \sup_{t \in [0, T]} \int_{\R^d} |y|^4 \de\mu_t \leq K \,.
\end{align}

\fbox{Step 4 (Uniqueness)} For two fixed points $u, \hat{u}$ of $\T$, and the corresponding processes $y_t, \hat{y}_t$ that solves \labelcref{eq: aux_sde} (which then also solve \labelcref{eq: meanfield-mirrorcbo-sde}), we define $z_t:= y_t - \hat{y}_t\,.$ Note that by Step 3, both $u, \hat{u}$ and  $\mu_t, \hat{\mu}_t$ satisfy \labelcref{eq: step3-bounds}. Hence,
\begin{align*}
    \max\left\lbrace 
    \sup_{t \in [0, T]}
    \int_{\R^d} |y|^4 \de\mu_t 
    ,
    \sup_{t \in [0, T]}
    \int_{\R^d} |y|^4 \de\hat\mu_t
    \right\rbrace 
    \leq K \,, \qquad 
    \|u \|_{\infty}\,, \|\hat{u} \|_{\infty} < q\,.    
\end{align*} 
Therefore, using a synchronous coupling of $y_t$ and $\hat y_t$ we have 
\begin{align*}
    z_t &= z_0 - \int_0^t \left(\nabla \phi^* (y_s) - \nabla \phi^* (\hat{y}_s)\right)\de s + 
    \int_0^t \left( u_s - \hat{u}_s \right) \de s \\
    & + \sqrt{2\sigma^2} \int_0^t \left(| \nabla \phi^* (y_s)  - u_s | - | \nabla \phi^* (\hat{y}_s) - \hat{u}_s |  \right) \de W_s\,.
\end{align*}
Also note that by definition of $\T$, we have 
$u = m_{\alpha}^*(\mu_t)\,, \hat{u} = m_{\alpha}^*(\hat{\mu}_t)\,.$ By It\^o's isometry and $L$-Lipschitzness of $\nabla \phi^*$ we thus have 
\begin{align*}
    \E|z_t|^2 & \leq 2\E|z_0|^2 + 8(t  + 2\sigma^2  ) L^2 \int_0^t \E |z_s|^2 \de s + 4t \int_0^t |m_{\alpha}^*(\mu_t) - m_{\alpha}^*(\hat{\mu}_t)|^2 \de s \\
    & \leq  2\E|z_0|^2 + 8(t + 2\sigma^2) L^2 \int_0^t \E |z_s|^2 \de s + 4t c_m^2 \int_0^t \E|z_s|^2 \de s \,,
\end{align*}
where the last line follows by \cref{lem: stability_weightedmean}. 
Hence by Gr\"onwall's inequality and $\E|z_0|^2 = 0$, we have $\E|z_t|^2 = 0$ for all $t \in [0, T]$. 
Furthermore, using \cref{lem: stability_weightedmean} again we have $u = \hat u$.
\end{proof}
\begin{remark}\label{rmk: aniso_wellposed_mf}
    To prove the well-posedness of the more general CBO model \labelcref{eq: mcbo_general}, it suffices to slightly modify Steps 1, 2 and 4 in the proof above. For example in \labelcref{eq: aux_sde,eq: y_compact_embedding}, we will change $\abs{\nphiy - u_t}$ to $S(\nphiy - u_t)\,.$ For $S: \R^d \rightarrow \R^{d \times d}$ with linear growth the same argument as in the proof \cref{thm: wellposedness-pde} can be applied. This was also discussed in \cite[Remark 1]{fornasier2022convergence} for the anisotropic case.  
\end{remark}

\section{Consensus formation for the Bregman distance}\label{sec: convergence_analysis}

Building on the well-posedness study from the previous section, we now turn our focus to the convergence of the MirrorCBO dynamics. 
An interesting feature of the algorithm is that although the actual evolution occurs in the dual space, the associated primal distribution eventually converges to the global minimum in the primal space.
Our main theorem in this section is \cref{thm: exp_decay_V} below.
We adapt the analysis tools for the original CBO algorithm, introduced in \cite{fornasier2024consensus}, to the MirrorCBO setting by utilizing Bregman distances and exploiting \cref{ass:bregman_phi}.

Conventionally, the convergence analysis for standard CBO involves a two major ingredients. 
First, exponential convergence of the particles towards a consensus point is established, typically by proving exponential decay of the variance or an appropriate Lyapunov functional. Second, the proximity of this consensus point to the global minimizer is ensured through the quantitative Laplace principle, where parameters can be suitably chosen to achieve arbitrary accuracy for sufficiently large values of the inverse temperature parameter~$\alpha$. 
Our convergence analysis for MirrorCBO follows this framework, circumventing substantial difficulties stemming from the presence of the mirror map.
We define a suitable Lyapunov functional taking the effect of the mirror map into account (\cref{def:Lyapunov}) and derive a decay estimate (\cref{prop: dynamics_V}).
A central issue here is that the presence of the mirror map does not allow one to prove precise convergence rates for the Lyapunov function, as opposed to the case of standard CBO, due to potentially degenerate diffusion.
Subsequently, we leverage the quantitative Laplace principle (\cref{prop: m-xhat_bound}) to complete the proof of our main convergence result (\cref{thm: exp_decay_V}). 

To motivate our approach, we will first briefly discuss the convergence analysis of \cite{fornasier2024consensus} which relies on analyzing the time dynamics of the Lyapunov function
\begin{align}\label{def: original_lyapunov}
    \mathcal V(\rho_t) := \frac{1}{2} \int_{\R^d} |x - \xhat|^2 \de\rho_t\,,
\end{align}
where $\xhat:=\argmin J$.
Note that this Lyapunov function equals a multiple of the squared \mbox{Wasserstein-$2$} distance between the distribution $\rho_t$ and $\delta_{\xhat}$, the Dirac delta distribution concentrated on the global minimum. 
The authors in \cite{fornasier2024consensus} then showed that for any given target accuracy $\varepsilon$, there exist parameters $\alpha$ and $\sigma$ such that the Lyapunov function decays exponentially fast until it reaches the target accuracy.
Besides an energy decay estimate, the key technical ingredient to prove this result is a quantitative Laplace principle. 
Invoking \labelcref{eq: inverse_conti-1,eq: inverse_conti-2} from \cref{ass: J2}, and denoting $J_{r}:= \sup_{x \in B_{r}(\xhat)} J(x)\,,$ for $r\in(0,R]$, this principle states that for suitably chosen $q>0$ and $r=r_q>0$,
\begin{align}\label{eq: m-xhat-original}
    | m_\alpha(\rho_t) - \xhat| \leq \frac{(q + J_r)^{\nu}}{\eta} + \frac{\exp(-\alpha q)}{\rho_t(B_r(\xhat))}\int_{\R^d} |x - \xhat|\de\rho_t(x) \quad \text{ for all } r \in (0, R]\,.
\end{align}
Due to continuity of $J$ from \ref{as: locally_lip}, the first term can be controlled well by choosing $q$ small enough such that $r$ and therefore also $J_{r_q}$ are small enough. To bound the second term, one needs to prove a lower bound on $\rho_t(B_r(\xhat))$, i.e., the amount of mass that is distributed near the global minimizer~$\xhat$. 

\subsection{Definition and properties of the Lyapunov function}

For generalizing this proof framework to MirrorCBO, we first need to find a suitable Lyapunov function. 
While it would be tempting to just use the original Lyapunov function $\mathcal{V}$ from \labelcref{def: original_lyapunov} for the primal distribution $\rho_t := (\nabla\phi^*)_\sharp\mu_t$, it turns out that the evolution of $\rho_t$ is not regular enough to study the time dynamics of such $\mathcal V$. 
Therefore, we need a dual Lyapunov function that depends solely on the dual distribution $\mu_t$ which satisfies the Fokker--Planck equation \labelcref{eq: meanfield-mirrorcbo-pde}.
Obvious choices would be $\frac{1}{2}\int_{\R^d}\abs{y-\hat y}^2\de\mu_t(y)$ for $\hat y\in\partial\phi(\hat x)$ or $\frac12\int_{\R^d}\abs{\nabla\phi^*(y)-\hat x}^2\de\mu_t(y)$ but it turns out that their temporal evolutions are not amenable to analysis. 
A suitable choice turns out to involve the Bregman distance between $\nabla\phi^*(y)$ and $\hat x$, as defined in \cref{def:bregman_distance}.
\begin{definition}\label{def:Lyapunov}
    For $\hat x \in \R^d$ we define a Lyapunov functional $V(t)$ as follows: \begin{align}\label{eq:Lyapunov}
        V(t):= \int_{\R^d} D_\phi^{y}(\hat    x,\nabla\phi^*(y))\de\mu_t(y).
    \end{align}
\end{definition}
We note that for $\phi(x)=\frac12\abs{x}^2$, the function $V$ in \labelcref{eq:Lyapunov} coincides with the Lyapunov function $\mathcal V$ defined in \labelcref{def: original_lyapunov} for the purpose of analyzing the long term behavior of standard CBO.
This fact can be seen, e.g., from \cref{rem:bregman_distance}.

In this way, our choice of $V(t)$ naturally generalizes the classical quadratic Lyapunov function used in standard CBO. The crucial insight behind this definition comes from the structure of the distance generating function $\phi$ and its dual $\phi^*.$ In standard CBO, the quadratic Lyapunov function measures distances directly in the Euclidean geometry. However, in MirrorCBO, dynamics evolve in both primal and dual spaces, and using the Bregman divergence associated with $\phi$ to define the Lyapunov functional allows to capture the geometry induced by the mirror map.
Moreover, the time derivative of the Lyapunov functional can be conveniently controlled thanks to its structure and \cref{ass:bregman_phi} (see \cref{lem: grad_lap_phi,prop: dynamics_V} below). This is a key motivation for choosing this specific form for $V(t)$.

Our Lyapunov function $V$ is informative in the following way: 
If $V(t)=0$ then $\mu_t$-almost every $y\in\R^d$ satisfies $y\in\partial\phi(\hat x)$, where $\xhat$ is the global minimizer of $J$.
Note that, unless $\phi$ is differentiable, this does not imply that $\mu_t$ is a concentrated measure in a single point. 
Instead, it is concentrated on the set-valued subdifferential of $\phi$ at the minimizer $\hat x$.
At the same time, using that $\phi$ is strongly convex we have $\nabla\phi^*(y)=\hat x$ for $\mu_t$-almost every $y$ which means that the primal particles are concentrated on the global minimizer $\hat x$ of~$J$.

Furthermore, using \cref{ass:bregman_phi} one observes that $V(t)$ is bounded from above and from below by a multiple of the squared \mbox{Wasserstein-$2$} distance $W_2(\rho_t,\delta_{\hat x})$ of the distribution $\rho_t:=(\nabla\phi^*)_\sharp\mu_t$ and $\delta_{\hat x}$.
Hence, any decay estimate for $V(t)$ directly converts into one for this Wasserstein distance.

\begin{remark}
Making a change of variables from $\mu_t$ to $\rho_t := (\nabla\phi^*)_\sharp\mu_t$ we can formally rewrite $V(t)$ as
\begin{align*}
    V(t) = \int_{\R^d} D_\phi^{p(x)}(\hat x,x)\de\rho_t(x)
\end{align*}
for $p(x)\in\partial\phi(x)$. 
However, the precise choice of subgradient $p(x)$ is not clear unless $\phi$ is differentiable which is why we stick to the rigorous definition of $V$ in \labelcref{eq:Lyapunov} which also works for non-differentiable $\phi$.
\end{remark}

For convenience, we present a table that matches the definition and results including estimates that help prove the main theorem in \cite{fornasier2024consensus} with our corresponding statements, see~\cref{table: estimate_comparison}. 

\begin{table}[h!]
\begin{center}
    \begin{tabular}{l|ll}
 & MirrorCBO 
 & Original CBO \cite{fornasier2024consensus} 

\\
\hline
Lyapunov Functional & \labelcref{eq:Lyapunov} & \labelcref{def: original_lyapunov} 
\\
\hline
Energy Decay Estimate 
& \cref{prop: dynamics_V}   
& Lemma 4.1, 4.2                        
\\
\hline
Quantitative Laplace Principle  
& \cref{prop: m-xhat_bound}   
& Proposition 4.5                       
\\
\hline
Lower bound on probability mass around \\ global minimum for all $t \geq 0$ 
& \cref{prop: lower_bound_B} & Proposition 4.6    
\\ 
\hline
Main Theorem & \cref{thm: exp_decay_V} 
& Theorem 3.7  
    \end{tabular}
\end{center}
\caption{estimate comparison}
\label{table: estimate_comparison}
\end{table}

\begin{lemma}\label{lem: grad_lap_phi}
    Suppose \cref{ass:bregman_phi} is satisfied. Then, for almost every $y \in \R^d$ we have 
\begin{equation} \begin{aligned}\label{eq: grad_lap_phi}
    \nabla D_\phi^y(\hat x,\nabla\phi^*(y))
    &= \nabla\phi^*(y)-\hat x \,, \\
    \Delta 
    D_\phi^y(\hat x,\nabla\phi^*(y))
    &=
    \tr(D^2\phi^*(y)) \,, \\
    \abs{\nabla D_\phi^y(\hat x,\nabla\phi^*(y))}
    &\leq 
    C_1(\abs{\xmin}+\abs{\hat x}+\abs{y})
    \\
   \left| \Delta  D_\phi^y(\hat x,\nabla\phi^*(y)) \right| & \leq C_2 \,,
\end{aligned}
\end{equation}
where $C_1 := \max\left\lbrace1,\frac{1}{2\mphi}\right\rbrace$, $C_2:= \frac{d}{2\mphi}$, and the first identity actually holds for all $y\in\R^d$.
\end{lemma}
\begin{proof}
    Recap that \cref{ass:bregman_phi} means that $\phi$ is strongly convex and hence that $\phi^*$ is in $\mathcal{C}^{1,1}(\R^d)$.
    In particular, $D^2\phi^*$ is Lebesgue almost everywhere well-defined and bounded.
    Using \cref{rem:bregman_distance} we get for almost every $y\in\R^d$ that
\begin{align*}
    \nabla D_\phi^y(\hat x,\nabla\phi^*(y))
    &=
    \nabla_y
    \left(\phi(\hat x) - \phi(\nabla\phi^*(y)) - \langle y,\hat x - \nabla\phi^*(y)\rangle\right)
    \\
    &=
    -D^2\phi^*(y)y
    -
    \hat x
    +
    \nabla\langle y,\nabla\phi^*(y)\rangle
    \\
    &=
    -D^2\phi^*(y)y
    -
    \hat x
    +
    \nabla\phi^*(y)
    +
    D^2\phi^*(y)y
    =
    \nabla\phi^*(y)-\hat x,
\end{align*}
    where we used $\phi(\nabla\phi^*(y))=\langle \nabla\phi^*(y),y\rangle-\phi^*(y)$. This follows from \cite[Proposition 17.27]{Bauschke_Combettes_2017}. 
    We also used 
    \begin{align*}
        \partial_i \langle y,\nabla\phi^*(y)\rangle
        &=
        \partial_i\sum_{j=1}^d y_j\partial_j\phi^*(y)
        =
        \sum_{j=1}^d
        \left(\delta_{ij}\partial_j\phi^*(y)
        +
        y_j\partial_{ij}^2\phi^*(y)\right)
        =
        \partial_i\phi^*(y) + \sum_{j=1}^d
        y_j\partial_{ij}^2\phi^*(y)
        \\
        &=
        \left(\nabla\phi^*(y)+D^2\phi^*(y)y\right)_i,
        \qquad
        i=1,\dots,d,
    \end{align*}
    and hence $\nabla\langle y,\nabla\phi^*(y)\rangle=\nabla\phi^*(y)+D^2\phi^*(y)y$. 
    Since $y\mapsto\nabla\phi^*(y)-\hat x$ is a fortiori continuous, the identity $\nabla D_\phi^y(\hat x,\nabla\phi^*(y))$ indeed holds for all $y\in\R^d$.
    Furthermore, using that $y\mapsto\nabla\phi^*(y)$ is $\frac{1}{2\mphi}$-Lipschitz continuous thanks to \cref{lem: convex_lip,rem:bregman_distance}, we obtain
    \begin{align*}
        \abs{\nabla D_\phi^y(\hat x,\nabla\phi^*(y))} 
        &= 
        \abs{\nabla\phi^*(y)-\hat x}
        \leq 
        \abs{\hat x}
        +
        \abs{\nabla\phi^*(0)}
        +
        \abs{\nabla\phi^*(y)-\nabla\phi^*(0)}
        \leq 
        \abs{\hat x} + \abs{x_\phi} + \frac{1}{2\mphi}\abs{y}
        \\
        &\leq 
        C_1(\abs{\xmin}+\abs{\hat x}+\abs{y})
    \end{align*}
where $x_\phi$ is the unique minimizer of $\phi$ and $C_1 := \max\left\lbrace 1, \frac{1}{2\mphi} \right\rbrace$. 
Taking the divergence in the identity $\nabla D_\phi^y(\hat x,\nabla\phi^*(y))=\nabla\phi^*(y)-\hat x$, we get for almost all $y\in\R^d$ that
\begin{align*}
    \Delta 
    D_\phi(\hat x,\nabla\phi^*(y))
    =
        \div\left(\nabla\phi^*(y)-\xhat\right)
        =
        \tr(D^2\phi^*(y))
    \end{align*} 
and by the Lipschitz continuity of $\nabla\phi^*$ we have that $\tr(D^2\phi^*(y))\leq  \frac{d}{2m_\phi}$, where $m_\phi$ is the strong convexity constant of $\phi$ from \cref{ass:bregman_phi}. Therefore,
\begin{align*}
    \abs{\Delta D_\phi^y(\hat x,\nabla\phi^*(y))} \leq C_2
\end{align*}
for the constant $C_2 = \frac{d}{2\mphi}\,.$
\end{proof}

\begin{proposition}\label{prop: dynamics_V} 
Under \cref{ass:bregman_phi} it holds for almost every $t>0$ that
    \begin{align}
        \label{eq:dV_lower_bound}
        \frac{\d}{\d t}V(t) 
        &\geq 
        -M_{\phi}^{-1}V(t) - | m_\alpha^*(\mu_t) - \xhat| \sqrt{{M_{\phi}^{-1}}V(t)}\,,
        \\
        \label{eq:dV_upper_bound}
        \frac{\d}{\d t}V(t) 
        &\leq 
        -\frac{2m_\phi - d\sigma^2}{2m_\phi M_\phi}
        V(t)
        +
        \left(
        1+
        \frac{d\sigma^2}{m_\phi}
        \right)
        \abs{
        \mamu-\hat x
        }
        \sqrt{m_\phi^{-1}V(t)}
        +
        \frac{d\sigma^2}{2m_\phi}
        \abs{m_\alpha^*(\mu_t) - \hat x}^2\,.
    \end{align}
\end{proposition}

\begin{remark}
    Note that, in contrast to the analysis of standard CBO in \cite{fornasier2024consensus}, the lower and upper convergence for the dynamics of the Lyapunov function $t\mapsto V(t)$ obtained in \cref{prop: dynamics_V} do not match: the lower rate is larger by a constant of $\frac{d\sigma^2}{2\mphi\Mphi}$.
    This is due to the lack of a non-trivial lower bound for the trace of the almost everywhere defined positive semi-definite Hessian~$D^2\phi^*$.
\end{remark}

\begin{proof}[Proof of \cref{prop: dynamics_V}]
    Since $V(t) = \int_{\R^d} D_\phi^y(\hat x,\nabla\phi^*(y))\de\mu_t(y)$ we aim to apply \cref{def:weak_solutions} to the test function $\xi(y) = D_\phi^y(\hat x,\nabla\phi^*(y))$. 
    Thanks to \cref{lem: grad_lap_phi} the test function $\xi$ lies in $\mathcal C^{1,1}_*(\R^d)$ and as a consequence of \cref{thm: wellposedness-pde} we get for almost every $t>0$ that
    \begin{align*}
        &\phantom{{}={}}
        \frac{\d}{\d t}
        V(t)
        \\
        &=
        -\int_{\R^d} \langle\nabla D_\phi(\hat x,\nabla\phi^*(y)),\nabla\phi^*(y)-\mamu\rangle\de\mu_t(y)
        +
        \sigma^2
        \int_{\R^d} 
        \left(\Delta 
        D_\phi(\hat x,\nabla\phi^*(y))\right)
        \abs{\nabla\phi^*(y)-\mamu}^2
        \de \mu_t(y)
        \\
        &=
        -\int_{\R^d} \langle\nabla\phi^*(y)-\hat x,\nabla\phi^*(y)-\mamu\rangle\de\mu_t(y)
        +
        \sigma^2
        \int_{\R^d} 
        \tr(D^2\phi^*(y))
        \abs{\nabla\phi^*(y)-\mamu}^2
        \de \mu_t(y)
        \\
        &=
        -\int_{\R^d} \abs{\nabla\phi^*(y)-\hat x}^2
        \de\mu_t(y)
        +
        \int_{\R^d}\langle\nabla\phi^*(y)-\hat x,\mamu-\hat x\rangle\de\mu_t(y)
        \\
        &\qquad
        +
        \sigma^2
        \int_{\R^d} 
        \tr(D^2\phi^*(y))
        \abs{\nabla\phi^*(y)-\mamu}^2
        \de \mu_t(y)\,.
\end{align*}
We first show the lower bound then the upper bound. Using $\tr(D^2\phi^*(y))\geq 0$ for almost every $y\in\R^d$ and the Cauchy--Schwarz inequality we have
\begin{align*}
    \frac{\d}{\d t}V(t)
        &\geq -\int_{\R^d} \abs{\nabla\phi^*(y)-\hat x}^2
        \de\mu_t(y)
        -
        \abs{
        \mamu-\hat x
        }
        \left(\int_{\R^d}\abs{\nabla\phi^*(y)-\hat x}^2\de\mu_t(y)\right)^{1/2}
        \\ 
        & \geq 
        -M_{\phi}^{-1}V(t) - |\mamu - \xhat| \sqrt{{M_{\phi}^{-1}}V(t)}\,,
\end{align*}
where the second inequality holds thanks to \cref{ass:bregman_phi}.

Now we show the upper bound.
Using Hölder's inequality and expanding the square it holds
\begin{align*}
        \frac{\d}{\d t}V(t)
        &\leq 
        -\int_{\R^d} \abs{\nabla\phi^*(y)-\hat x}^2
        \de\mu_t(y)
        +
        \abs{
        \mamu-\hat x
        }
        \left(\int_{\R^d}\abs{\nabla\phi^*(y)-\hat x}^2\de\mu_t(y)\right)^{1/2}
        \\
        &\qquad
        +
        \sigma^2
        \int_{\R^d} 
        \tr(D^2\phi^*(y))
        \left(
        \abs{\nabla\phi^*(y)-\hat x}^2
        +
        \abs{\hat x-\mamu}^2
        +
        2
        \langle
        \nabla\phi^*(y)-\hat x,
        \hat x-\mamu
        \rangle
        \right)
        \de \mu_t(y).
    \end{align*}
    Hence by using $\tr(D^2\phi^*(y))\leq\frac{d}{2m_\phi}$, we obtain
    \begin{align*}
        \frac{\d}{\d t}
        V(t)
        &\leq 
        -\left(1-\frac{d\sigma^2}{2m_\phi}\right)
        \int_{\R^d}\abs{\nabla\phi^*(y)-\hat x}^2\de\mu_t(y)
        +
        \left(
        1+
        \frac{d\sigma^2}{m_\phi}
        \right)
        \abs{
        \mamu-\hat x
        }
       \left(\int_{\R^d}\abs{\nabla\phi^*(y)-\hat x}^2\de\mu_t(y)\right)^{1/2} \notag
        \\
        &\qquad
        +
        \frac{d\sigma^2}{2m_\phi}
        \abs{\mamu - \hat x}^2 \notag
        \\
        &\leq 
        -\frac{2m_\phi - d\sigma^2}{2m_\phi M_\phi}
        V(t)
        +
        \left(
        1+
        \frac{d\sigma^2}{m_\phi}
        \right)
        \abs{
        \mamu-\hat x
        }
        \sqrt{m_\phi^{-1}V(t)}
        +
        \frac{d\sigma^2}{2m_\phi}
        \abs{m_\alpha^*(\mu_t) - \hat x}^2.
    \end{align*}
\end{proof}
\begin{remark}
    We point out that the constants in front of $V(t)$ in the lower and upper bound \labelcref{eq:dV_lower_bound,eq:dV_upper_bound} are different. 
    This is due to the lack of a positive lower bound for the trace of the Hessian $\tr(D^2\phi^*(y))$ which in the case of standard CBO just equals the dimension $d$. 
    However, such a bound would imply that $\phi^*$ is strongly convex and hence $\phi$ has a Lipschitz continuous gradient. 
    Since our analysis is supposed to allow for non-smooth and merely strongly convex distance generating functions $\phi$, we refrain from posing such assumptions.
    For the subsequent analysis this discrepancy does not cause any problems.
\end{remark}

\subsection{Quantitative Laplace principle and lower bound for the mass}

We proceed with a quantitative Laplace principle which allows us to control the term $\abs{m_\alpha^*(\mu_t) - \hat x}$ that appears in \cref{prop: dynamics_V} and describes the distance between the consensus point of MirrorCBO and the global minimizer $\hat x$ of the cost function $J$.
The following result is a direct consequence of the quantitative Laplace principle proved in \cite{fornasier2024consensus}.
\begin{proposition}[Quantitative Laplace principle]\label{prop: m-xhat_bound}
Suppose \cref{ass: J2} holds with $\eta>0$ and $J_\infty>0$, and without loss of generality assume that $\underline{J} = 0$.
Let $\mu\in \mathcal{P}(\R^d)$ and define $\rho := (\nabla\phi^*)_\sharp\mu\in\mathcal
P(\R^d)$. Define 
    \begin{align*}
        J_{r}:= \sup_{x \in B_{r} (\xhat)} J(x)\,,
    \end{align*} 
where $B_{r}:= \{x: |x - \xhat| \leq r \}.$  Fix $ q > 0\,,r=r_q \in (0, R]$ so that $q + J_{r_q} \leq J_{\infty}.$ Then for any $\alpha > 0$,
    \begin{align*}
        |m_\alpha^*(\mu)- \xhat| \leq \frac{(q + J_{r_q})^{\nu}}{\eta} + \frac{e^{-\alpha q}}{\rho(B_{r_q}(\xhat))} \int_{\R^d} |\nphiy - \xhat|\de\mu(y)\,.
    \end{align*}
\end{proposition}
\begin{proof}
The result is a direct consequence of \cite[Proposition 4.5]{fornasier2024consensus} and a change of variables. 
Since by definition of $\rho$ it holds $m_\alpha^*(\mu)=m_\alpha(\rho)$ we can use this result to get
\begin{align*}
    \abs{m_\alpha^*(\mu)-\hat x}
    &=
    \abs{m_\alpha(\rho)-\hat x}
    \leq 
    \frac{(q + J_{r_q})^{\nu}}{\eta} + \frac{e^{-\alpha q}}{\rho(B_{r_q}(\xhat))} \int_{\R^d} |x - \xhat|\de\rho(x)
    \\
    &=
    \frac{(q + J_{r_q})^{\nu}}{\eta} + \frac{e^{-\alpha q}}{\rho(B_{r_q}(\xhat))} \int_{\R^d} |\nphiy - \xhat|\de\mu(x).
\end{align*}
\end{proof}
\begin{proposition}[Lower bound on $\rho_t(B_{\bar{r}}(\xhat))$]\label{prop: lower_bound_B}
Suppose that \cref{ass:bregman_phi} holds and fix $T, r, \alpha> 0$.
For the initial distribution $\mu_0\in \mathcal{P}(\R^d)$, let $\mu_t\in \mathcal{C}([0, T], \mathcal{P}(\R^d)$ solve \labelcref{eq:weak_form_fp}. 
Choose $B\geq \sup_{t \in [0, T]} |\mamu - \xhat|$, and define a mollifier $\psi_r$ whose support equals the closure of
\begin{align*}
    B_r^*(\xhat):= \left\{ y \in \R^d:D_{\phi}^y(\xhat, \nphiy) \leq r^2 \right\}    
\end{align*}
as follows:
\begin{align}\label{eq: mollifer}
    \psi_r(y)  
    := 
    \eta_r\left( D_{\phi}^y(\xhat, \nphiy \right) 
    \qquad 
    \text{where}
    \qquad
    \eta_r(t) 
    := 
    \begin{dcases}
        \exp \left( 1- \frac{r^2}{r^2 - t} \right) \quad&\text{if } t \leq r^2\,,
        \\ 
        0 \quad &\text{otherwise}.
    \end{dcases}
\end{align}
Also, fix a constant $c \in (1/2, 1)$ which  satisfies the following for $C:= \sup_{y \in \dualball} \tr(D^2 \phi^*(y) )\,,$
\begin{align}\label{eq: c_condition}
    \Mphi^{-1} c (2 c-1) \geq 2 C (1 - c)^2
\end{align}
and define
\begin{align*}
    p:= \max\left\{  \frac{\left(\sqrt{\mphi}^{-1}\sqrt{c}r + B\right)\sqrt{\mphi}^{-1}\sqrt{c} }{ (1 - c)^2 r}  + \frac{2\sigma^2 \left(m_\phi^{-1}c+  C\right)  \left( \mphi^{-1} c r^2 + B^2\right)}{ ( 1 - c  )^4 r^2 }  \,, \frac{2}{(2c - 1) \sigma^2}  \right\} \,.
\end{align*}
Then for all $t\in [0, T]$ and defining $\bar{r}:=\frac{r}{\sqrt{\mphi}}$, the following holds, where $\rho_t := (\nabla\phi^*)_\sharp\mu_t$: 
\begin{align*}
    \mu_t(B_r^*(\xhat)) 
    \geq 
    \left(\int_{\R^d} \psi_r (y)\de\mu_0 (y) \right) e^{-p t}
    \qquad 
    \text{and}
    \qquad 
    \rho_t(B_{\bar{r}}(\xhat)) 
    \geq 
    \frac12
    \mu_0(B_{r/2}^{*}(\xhat)) e^{-pt}\,.
\end{align*}
\end{proposition}
\begin{proof}
The proof is similar to that of  \cite[Proposition 4.6]{fornasier2024consensus}, or \cite[Lemma 4.6]{carrillo2024interacting}. 
The basic idea is to analyze $\frac{\d}{\d t}\int_{\R^d} \psi_r (y)\de\mu_t(y)$ and use the fact that $\mu_t$ is a weak solution to \labelcref{eq:weak_form_fp} to bound $\frac{\d}{\d t}\int_{\R^d}\psi_r(y) \de\mu_t(y)$ from below. Note that $\psi_r \in \mathcal C^{1, 1}_{*}(\R^d)$ thanks to \cref{lem: grad_lap_phi}.
With some algebra this also allows us to compute $\nabla \psi_r(y)$ and $\Delta \psi_r(y)$, the latter for almost every $y\in\R^d$:
\begin{align}\label{eq: grad_delta_psi}
\begin{split}
     &\nabla \psi_r(y) = -r^2  \frac{ \nabla_y( D_{\phi}^y (\xhat, \nphiy)) }{\left( r^2 - D_{\phi}^y (\xhat, \nphiy) \right)^2 }  \psi_r(y) = -r^2  \frac{\nphiy - \xhat }{\left( r^2 - D_{\phi}^y (\xhat, \nphiy) \right)^2 }  \psi_r(y)\,, \\
    &\Delta \psi_r(y) = r^2
    \frac{\left[\left(2 D_{\phi}^y (\xhat, \nphiy)- r^2\right) |\nphiy - \xhat|^2 -  \tr\left(D^2\phi^*(y) \right) \left( r^2 - D_{\phi}^y (\xhat, \nphiy) \right)^2\right]}{\left( r^2 - D_{\phi}^y (\xhat, \nphiy) \right)^4 }  \psi_r(y)\,.
\end{split}
\end{align}
Thanks to the definition of mollifier \labelcref{eq: mollifer} we have $0 \leq \psi_r(y) \leq 1$ and hence \begin{align}\label{eq: mu_lower_bound}
    \mu_t (\dualball) = \mu_t \left( \left\{ y \in \R^d: D_{\phi}^y (\xhat, \nphiy) \leq r^2 \right\} \right) \geq \int_{\R^d} \psi_r(y) \de\mu_t(y)\,.
\end{align}
For the purpose of estimating the lower bound of the right-hand side of \labelcref{eq: mu_lower_bound} we analyze its time decay using \labelcref{eq: grad_delta_psi}. For notational simplicity, we decompose $\frac{\d}{\d t}\int_{\R^d}\psi_r(y) \de\mu_t(y)$ into two parts as follows: \begin{align*}
    \frac{\d}{\d t}\int_{\R^d} \psi_r(y) \de\mu_t(y)&:= \int_{\R^d} - \langle \nabla \phi^*(y) - \mamu, \nabla \psi_r(y) \rangle \de\mu_t(y) + \sigma^2 \int_{\R^d} |\nabla \phi^*(y) - \mamu|^2 \Delta \psi_r(y) \de\mu_t(y)\\
    & = \int_{\R^d} M_1(y) \de\mu_t(y) + \int_{\R^d} M_2(y) \de\mu_t(y)\,.
\end{align*}
Recall that $\supp(\psi_r) = B_r^*(\xhat)$. Hence, $\nabla \psi_r(y) = \Delta \psi_r(y) = 0$ for $y \in \R^d \setminus  B_r^*(\xhat)\,.$ Therefore, we only have to take care of the integrals over the set $B_r^*(\xhat)$. 
To this end, we decompose it into three parts which involves the following definition of subdomains depending on some fixed positive constant $c < 1$: 
\begin{align}
    &S_1:= \left\{ y \in \R^d: D_{\phi}^y(\xhat, \nabla \phi^*(y) ) > cr^2 \right\}\,,\\
    &S_2:= \Big\{ y \in \R^d: - 2 \langle \nphiy - \mamu, \nphiy - \xhat \rangle \left(r^2 - D_{\phi}^y(\xhat, \nabla \phi^*(y) )\right)^2 \\ \notag 
    & \qquad \qquad \qquad \qquad > \sigma^2 
    (2 c -1) r^2 |\nphiy - \mamu|^2 |\nphiy - \xhat|^2  \Big\}\,.
\end{align}
For such subdomains, we have $\dualball = (S_1^{c} \cap \dualball) \cup (S_1 \cap S_2^{c} \cap \dualball) \cup (S_1 \cap S_2 \cap \dualball)\,.$ We now show that for each subdomain, the integrals of $M_1$ and $M_2$ are well controlled. 
\begin{enumerate}[label=\textbf{\alph*)}]
\item \textbf{$y \in S_1^{c} \cap \dualball$:}
    Note that by definition of the subdomain $S_1$ and using \cref{ass:bregman_phi}, we have:
    \begin{align}\label{eq:relations}
    \begin{split}
        &|\nphiy - \xhat| \leq \sqrt{ \mphi^{-1} D_{\phi}^y(\xhat, \nphiy) } \leq \sqrt{\mphi}^{-1} \sqrt{c}r\,, \\
        &|\nphiy - \mamu| \leq |\nphiy - \xhat| + |\xhat - \mamu| \leq \sqrt{\mphi}^{-1} \sqrt{c} r + B\,, \\
        &|\nphiy - \mamu|^2 \leq  2 \left( \mphi^{-1} c r^2 + B^2\right) \,, \\
        & r^2 - D_{\phi}^y(\xhat, \nphiy) \geq (1-c)r^2 > 0\,.
    \end{split}
    \end{align}
    Recapping the definition of $M_1(y)$ and plugging in \labelcref{eq: grad_delta_psi} we have:
    \begin{align*}
        M_1(y) = r^2 \frac{ \langle \nphiy - \mamu, \nphiy - \xhat \rangle }{\left( r^2 - D_{\phi}^y (\xhat, \nphiy) \right)^2  } \psi_r(y) .
    \end{align*}
     Using \labelcref{eq:relations} and the Cauchy--Schwarz inequality, we have
\begin{align}\label{eq: M1_bound}
    M_1(y) & \geq -r^2 \frac{|\nphiy - \mamu||\nphiy - \xhat|}{ \left( r^2 - D_{\phi}^y (\xhat, \nphiy) \right)^2 }\psi_r(y) \notag\\
    & \geq -\frac{|\nphiy - \mamu||\nphiy - \xhat|}{ ( 1 - c  )^2 r^2}\psi_r(y) \notag\\
    & \geq - \frac{ \left(\sqrt{\mphi}^{-1}\sqrt{c} r  + B \right) \sqrt{\mphi}^{-1} \sqrt{c}}{( 1 - c  )^2 r } \psi_r(y)=: - p_1 \psi_r(y)\,.
\end{align}
For $M_2$ we can use \labelcref{eq: grad_delta_psi,eq:relations} as well as the non-negativity of the Bregman distance to show
\begin{align}\label{eq: M2_bound}
M_2(y) 
&= 
\sigma^2
\abs{\nabla\phi^*(y)-m_\alpha^*(\mu_t)}^2
r^2
\psi_r(y)
\times 
\\
&\qquad\times
\frac{\left[    \left(2 D_{\phi}^y (\xhat, \nphiy)- r^2\right)\abs{\nabla\phi^*(y)-\hat x}^2
-
\tr(D^2\phi^*(y)) \left(r^2 - D_{\phi}^y (\xhat, \nphiy) \right)^2\right]}
{\left(r^2-D_\phi^y(\hat x,\nabla\phi^*(y))\right)^4}
\notag
\\
&\geq 
\sigma^2
\abs{\nabla\phi^*(y)-m_\alpha^*(\mu_t)}^2
r^2
\psi_r(y)
\frac{
-m_\phi^{-1}cr^4
-
r^4  C
}{(1-c)^4r^8}
\notag
\\
&=
-
\sigma^2
\abs{\nabla\phi^*(y)-m_\alpha^*(\mu_t)}^2
\psi_r(y)
\frac{m_\phi^{-1}c + C}{(1-c)^4r^2}
\notag
\\
& \geq  - \frac{ 2\sigma^2 \left(m_\phi^{-1}c+  C\right)  \left( \mphi^{-1} c r^2 + B^2\right)}{ ( 1 - c  )^4 r^2 }  \psi_r(y) =: - p_2  \psi_r(y) \,,
\end{align}
where $C:= \sup_{y \in \dualball} \tr(D^2 \phi^*(y) )\in[0,\infty)$.
Combining this with \labelcref{eq: M1_bound}, we thus have for $ y \in  S_1^{c} \cap \dualball\,,$
\begin{align*}
   M_1(y) + M_2(y) \geq - (p_1 + p_2) \psi_r(y)\,.
\end{align*}
\item \textbf{$y \in S_1 \cap S_2^{c} \cap \dualball$}: 
Using the definition of $M_1(y)$ and $M_2(y)$ and some algebra it is easy to see that $M_1(y)+M_2(y)\geq 0$ is equivalent~to
\begin{align}\label{eq:condition_S2}
\begin{aligned}
    &\phantom{{}={}}
    -\langle \nphiy - \mamu, \nphiy - \xhat \rangle
    \left(r^2 - D_{\phi}^y (\xhat, \nphiy) \right)^2
    \\
    &\qquad +
    \sigma^2
    \abs{\nabla\phi^*(y)-m_\alpha^*(\mu_t)}^2
    \tr(D^2\phi^*(y))
    \left(r^2 - D_{\phi}^y (\xhat, \nphiy) \right)^2
    \\
    &\qquad 
    \leq 
    \sigma^2
    \left(2 D_{\phi}^y (\xhat, \nphiy)- r^2\right)
    \abs{\nabla\phi^*(y)-m_\alpha^*(\mu_t)}^2
    \abs{\nabla\phi^*(y)-\hat x}^2.
\end{aligned}   
\end{align}
Now since $ y \in S_2^{c}$ we have 
\begin{align*}
    &\phantom{{}={}}
    - \langle \nphiy - \mamu, \nphiy - \xhat \rangle \left(r^2 - D_{\phi}^y(\xhat, \nabla \phi^*(y) \right)^2 
    \\
    &\leq 
    \frac{\sigma^2}{2} (2c - 1)r^2  |\nphiy - \mamu|^2 |\nphiy - \xhat|^2.
\end{align*}
Using this and the fact that $(2c-1)r^2\leq 2D_\phi^y(\hat x,\nphiy)-r^2$ since $y\in S_1$, we can therefore upper-bound the first summand on the left hand side of \labelcref{eq:condition_S2} from above by a half of the right hand side. 
It therefore remains to do the same for the second summand.
To this end we use \labelcref{eq: c_condition} to estimate
\begin{align*}
    &\phantom{{}={}}
    \sigma^2
    \abs{\nabla\phi^*(y)-m_\alpha^*(\mu_t)}^2
    \tr(D^2\phi^*(y))
    \left(r^2 - D_{\phi}^y (\xhat, \nphiy) \right)^2
    \\
    &\leq 
    \sigma^2
    C(1-c)^2r^4
    \abs{\nabla\phi^*(y)-m_\alpha^*(\mu_t)}^2
    \\
    &\leq 
    \frac{\sigma^2}{2}
    \Mphi^{-1} c (2 c-1)
    r^4
    \abs{\nabla\phi^*(y)-m_\alpha^*(\mu_t)}^2
    \\
    &\leq 
    \frac{\sigma^2}{2}
    (2 D_{\phi}^y (\xhat, \nphiy)-r^2)
    \abs{\nabla\phi^*(y)-m_\alpha^*(\mu_t)}^2
    \abs{\nphiy-\hat x}^2
\end{align*}
where in the last inequality we used $y\in S_1$ and \cref{ass:bregman_phi}.
Hence, we also estimated the second summand in \labelcref{eq:condition_S2} by half of the right hand side and therefore we proved \labelcref{eq:condition_S2} and hence
\begin{align*}
    M_1(y) + M_2(y) \geq 0 
\end{align*}
for all $ y \in S_1 \cap S_2^{c} \cap \dualball\,.$
\item \textbf{$y \in S_1 \cap S_2 \cap \dualball$}: We have the following estimate on this subdomain: 
\begin{align*}
    &\phantom{{}={}}
    - \langle \nphiy - \mamu, \nphiy - \xhat \rangle \left (r^2 - D_{\phi}^y(\nphiy, \xhat) \right)^2  
    \\
    &\qquad >
    \frac{\sigma^2}{2} (2c - 1)r^2  |\nphiy - \mamu|^2 |\nphiy - \xhat|^2\,.
\end{align*}
Then by the Cauchy--Schwarz inequality, we have
\begin{align*}
    M_1(y) & \geq - r^2 \frac{|\nphiy - \mamu||\nphiy - \xhat|}{ \left( r^2 - D_{\phi}^y (\xhat, \nphiy) \right)^2 }\psi_r(y) \\
    &\geq  
    \frac{2}{\sigma^2(2c-1)}
    \frac{\langle \nphiy - \mamu, \nphiy - \xhat \rangle}{|\nphiy - \mamu||\nphiy - \xhat|}
    \psi_r(y)
    \\
    &\geq 
    -\frac{2}{\sigma^2(2c-1)}
    \psi_r(y)
    =: - p_3 \psi_r(y)\,.
\end{align*}
Next, we show that $M_2(y)\geq 0$, which can alternatively be formulated as the following: 
\begin{align*}
    \left(2 D_{\phi}^y (\xhat, \nphiy)- r^2\right)\abs{\nabla\phi^*(y)-\hat x}^2
    -
    \tr(D^2\phi^*(y)) \left(r^2 - D_{\phi}^y (\xhat, \nphiy) \right)^2
    \geq 0.
\end{align*}
This follows by combining a condition on $c$ \labelcref{eq: c_condition}, and that $y \in S_1$.
\begin{align*}
    \tr(D^2\phi^*(y)) \left(r^2 - D_{\phi}^y (\xhat, \nphiy) \right)^2
    &\leq 
    Cr^4(1-c)^2
    \leq 
    r^4
    \frac{c(2c-1)}{2\Mphi}
    \\
    &\leq 
    \Mphi^{-1} D_\phi^y(\hat x,\nphiy) (2D_\phi^y(\hat x,\nphiy)-r^2)
    \\
    &\leq 
    \left(2 D_{\phi}^y (\xhat, \nphiy)- r^2\right)\abs{\nabla\phi^*(y)-\hat x}^2\,.
\end{align*}
Combining the lower bound for $M_1(y), M_2(y)$ we obtained above, we get
\begin{align*}
    M_1(y) + M_2(y) \geq - p_3 \psi_r(y)\,.
\end{align*}
\end{enumerate}
We thus have the following lower bound on the time decay estimate of the integral evaluation:
\begin{align*}
    & \frac{\d}{\d t}\int_{\R^d}\psi_r(y) \de \mu_t(y) = \frac{\d}{\d t}\int_{\dualball}\psi_r(y) \de \mu_t(y) \\
    &=  \int_{S_1^{c} \cap \dualball} \left(M_1(y) + M_2(y) \right)\de\mu_t(y) + \int_{ S_1 \cap S_2^{c} \cap \dualball} M_1(y) + M_2(y) \de\mu_t(y) 
    \\
    &\qquad 
    + \int_{S_1 \cap S_2 \cap \dualball} \left(M_1(y) + M_2(y) \right)\de\mu_t(y)\\
    & \geq -  \max\{ p_1 + p_2 , p_3 \}\int_{\R^d}\psi_r(y) \de\mu_t(y)=: - p \int_{\R^d}\psi_r(y)\de\mu_t(y)\,.
\end{align*}
Hence, by Gr\"onwall's inequality, we obtain the first claim of the proposition:
\begin{align*}
    \mu_t(B_r^*(\hat x)) \geq 
    \int_{\R^d} \psi_r \de\mu_t \geq 
    \left(\int_{\R^d} \psi_r \de\mu_0\right)e^{-pt}.
\end{align*}
Finally, by the set inclusion coming from \cref{ass:bregman_phi} and using that $\psi_r(y)\geq\frac12$ for $y\in B_{r/2}^*(\hat x)\,,$ which holds since $\eta_r(t)\geq\frac12$ for $t\leq\frac{r^2}{4}\,,$ we have
\begin{align*}
    \rho_t(B_{\bar{r}}(\xhat)) 
    &= 
    \rho_t\left(\left\{ x \in \R^d: |x - \xhat| \leq \frac{r}{\sqrt{\mphi}} \right\} \right)  
    \geq 
    \mu_t\left( \left\{ y \in \R^d: D_{\phi}^y(\xhat, \nphiy) \leq r^2 \right\} \right) 
    \\
    &= 
    \mu_t(B_r^*(\xhat))  
    \geq 
    \left(\int_{\R^d} \psi_r \de\mu_0\right)e^{-pt}
    \geq 
    \frac12
    \mu_0(B_{r/2}^{*}(\xhat)) e^{-pt} \,.
\end{align*}
\end{proof}

\subsection{The main convergence theorem}

We are now ready to prove our main convergence result.

\begin{theorem}
\label{thm: exp_decay_V}
Suppose that 
\cref{ass: J,ass: J2,ass: phi0,ass:bregman_phi}
hold, that $\mu_0 \in \mathcal P_4(\R^d)$, and that $\hat x\in\supp(\nabla\phi^*)_\sharp\mu_0$. 
Suppose without loss of generality that $\underline{J} = 0.$
Fix any $\varepsilon \in (0, V(0))$, and $\tau \in (0, 1)$.  
For $\sigma^2<\frac{2\tau}{2+\tau}\frac{m_\phi}{d}$ define
\begin{align}\label{eq: Tstar}
    T^*:= \frac{2m_{\phi} M_{\phi}}{(1 - \tau)(2m_{\phi} - d\sigma^2)}\log\left(\frac{V(0)}{\varepsilon}\right)\,.
\end{align} 
Then, there exists $\alpha_0 > 0$ (depending on $\varepsilon$ and $\tau$) such that for all $\alpha>\alpha_0$ there exists $T_{\alpha,\varepsilon}\in\left[\frac{1-\tau}{1+\frac{\tau}{2}}T^*,T^*\right]$ such that 
\begin{align}\label{eq: V_bound}
    V(T_{\alpha,\varepsilon}) = \varepsilon \,.
\end{align}
Furthermore, $V(t)$ decays exponentially fast until it satisfies the error bound criteria in \labelcref{eq: V_bound}: 
\begin{align*}
    V(t) \leq V(0) \exp\left( - (1 - \tau) \left(\frac{2m_\phi - d\sigma^2}{2m_\phi M_\phi}\right) t\right)\,,
    \qquad
    \forall t\in[0,T_{\alpha,\varepsilon}].
\end{align*}
\end{theorem}

A few remarks are in order before we elaborate on the proof.

\begin{remark}
    We note that, while the Lyapunov function $V(t)$ for $t>0$ depends on the inverse temperature parameter $\alpha>0$, its initial value $V(0)$ and hence also $T^*$ are independent of $\alpha$.
\end{remark}
\vspace{2mm}
\begin{remark}
    The fact that the upper and lower convergence rates in \cref{prop: dynamics_V} do not match was also encountered in the analysis of CBO with memory terms \cite{riedl2024leveraging}. This represents an additional challenge for the convergence analysis and means that the time $T_{\alpha,\varepsilon}$ at which the target accuracy occurs is not necessarily arbitrarily close to $T^*$, but could potentially be significantly smaller (unless the noise $\sigma^2$ is very close to zero).
    This means the algorithm can potentially converge even faster than predicted by the theory; a potential benefit rather than a shortcoming that will be subject to future investigation.
\end{remark}
\vspace{2mm}
\begin{remark}
    Note that we cannot expect to be able to improve this exponential convergence result to infinite times.
    This is due to the fact that all CBO dynamics are biased (by the presence of the inverse temperature parameter $\alpha$) and do not concentrate around the minimizer but only close to it.
    A potential avenue towards such a result would be an adaptive choice of $\alpha=\alpha(t)\to\infty$ as $t\to\infty$, which is often done in practice and is also still a key open analytical question for CBO and its variants.
\end{remark}
\vspace{2mm}
\begin{remark}
    Note that our result does not exactly reduce to \cite[Theorem 3.7]{fornasier2024consensus} if one chooses $\phi = \frac12\abs{\cdot}^2$ and correspondingly $\mphi=\Mphi=\frac12$ since (i) we use a different scaling of the noise term with $\sqrt{2}$ and, more importantly, (ii) the lower bound \labelcref{eq:dV_lower_bound} in \cref{prop: dynamics_V} has a different constant than the upper bound.
    As mentioned above this is since in its proof we used the estimate $\tr(D^2\phi^*)\geq 0$ which is not sharp for $\phi=\frac12\abs{\cdot}^2$ but sharp for the class of distance generating functions $\phi$ we consider.  
\end{remark}
\vspace{2mm}
\begin{remark}\label{rmk: aniso_longtime_asymptotics}
    By adapting the proof suggested in \cite{fornasier2022convergence}, we can extend our convergence result to MirrorCBO with the anisotropic noise case which arises for the choice $S(u) = \operatorname{diag}(u)$ in \labelcref{eq: mcbo_general}.
    The basic idea is to substitute the mirror version of the $L_2$ ball used in the estimations for \cref{thm: exp_decay_V} with a mirrored $L_{\infty}$ ball to capture the maximum fluctuations in all dimensions. This can be applied to \cref{prop: m-xhat_bound,prop: lower_bound_B} to cover the anisotropic noise version of MirrorCBO.
    As a result, one can extend the convergence proof for \cref{thm: exp_decay_V} to the anisotropic MirrorCBO similar to how the convergence result of the original CBO was extended to CBO with anisotropic noise in \cite[Theorem 2]{fornasier2022convergence}.
\end{remark}
\begin{proof}[Proof of \cref{thm: exp_decay_V}]
    The proof follows very closely to the one in \cite[Theorem 3.7]{fornasier2024consensus} with non-trivial modifications due to the non-matching rates in \cref{prop: dynamics_V}.
    Our goal is to show that (i) for $T_{\alpha, \varepsilon}$ defined as in \labelcref{eq: T_alpha} below, we have $V(T_{\alpha, \varepsilon})=\varepsilon$, and (ii) $V(t)$ decays exponentially for all $t \in [0, T^*]$. For this, we define
    \begin{gather*}
        c(\tau, \sigma):= \min\left\{ 
        \frac{\tau}{2}
        \frac{(2m_{\phi} - d\sigma^2)\sqrt{m_{\phi}}
        }{2M_{\phi}(m_{\phi} + d\sigma^2)}\,, 
        \sqrt{
        \frac{\tau}{2}
        \frac{2m_{\phi} - d\sigma^2}{d\sigma^2 M_{\phi}}}, 
        \frac{\frac{\tau}{2}(2m_\phi-d\sigma^2)-d\sigma^2}{2 m_\phi \sqrt{M_\phi}}
        \right\}\,, \\
        C(t):= c(\tau, \sigma)\sqrt{V(t)}\,,
    \end{gather*}
    where $c(\tau,\sigma)>0$ thanks to the smallness assumption on $\sigma$, and  
    \begin{align}\label{eq: T_alpha}
        T_{\alpha, \varepsilon}:= \sup\left\{t \geq 0: V(s) > \varepsilon\text{ and } |m_{\alpha}^{*}(\mu_s) - \xhat| < C(s) \text{ for all } s\in [0, t] \right\}\,.
    \end{align}    
    We first show that $T_{\alpha, \varepsilon} > 0$.
    Since the functions $t\mapsto V(t)$ and $t\mapsto\abs{\mamu-\xhat}$ are continuous (as a consequence of \cref{thm: wellposedness-pde,lem: stability_weightedmean}), it suffices to argue that $V(0)>\varepsilon$ and $\abs{m_\alpha^*(\mu_0)-\xhat} < C(0)$.
    The first inequality is true by definition of $\varepsilon$.
    The second one needs more work:
    Using \cref{prop: m-xhat_bound} we have    \begin{align}\label{eq: application_mx_bound}
        |m_\alpha^*(\mu_0)- \xhat| \leq \frac{(q_{\varepsilon} + J_{r_{\varepsilon}})^{\nu}}{\eta} + \frac{e^{-\alpha q_{\varepsilon}}}{\rho_0(B_{r_{\varepsilon}}(\xhat))} \int_{\R^d} |\nphiy - \xhat| \de\mu_0(y)\,,
    \end{align}
    where 
    \begin{align} \label{eq: q_choice}
    &q_{\varepsilon}:= \frac{1}{2}\min\left\{ \left( \eta \frac{c(\tau, \sigma) \sqrt{\varepsilon}}{2}\right)^{1/\nu}, \, J_{\infty} \right\} >0\,,    \\
    &r_{\varepsilon}:= \max_{s \in [0, R]} \left\{ \sup_{x \in B_s(\xhat)} J(x) \leq q_{\varepsilon}   \right\} \,.\label{eq: r_choice}
    \end{align}
    Observe that the definition of $q_{\varepsilon}$ and $r_{\varepsilon}$ leads to 
    \begin{align*}
        q_{\varepsilon} + J_{r_{\varepsilon}} \leq 2q_{\varepsilon} \leq J_{\infty}\,.
    \end{align*}
    Also, there exists $s_{q_{\varepsilon}}> 0$ such that $J(x) \leq q_{\varepsilon}$ for all $x \in B_{s_{\varepsilon}}(\xhat)$ by the continuity of $J$. Thus, $r_{\varepsilon} > 0$. That $r_{\varepsilon}\leq R$ and $q_{\varepsilon}> 0$ follows by the construction.
    Therefore, $r_{\varepsilon}$ and $q_{\varepsilon}$ satisfy the condition required to apply \cref{prop: m-xhat_bound} and get \labelcref{eq: application_mx_bound}. Now take $\alpha > 0$ such that \begin{align}\label{eq: alpha}
        \alpha > \alpha_0 
        %
        & := \frac{1}{q_\varepsilon}\left[ \log \left( \frac{4\sqrt{V(0)}}{c(\tau, \sigma)\sqrt{\varepsilon \mphi}}\right) +
        pT^* - \log \left({\mu_0\left (B_{ \frac{\sqrt{\mphi} 
 r_{\varepsilon}}{2}}^{*}(\xhat) \right) } \right) \right] \,,
    \end{align}
    where $T^*$ is defined in \labelcref{eq: Tstar} and we define
    $p$ as in \cref{prop: lower_bound_B} with the values $T:=T_{\alpha, \varepsilon}$, $B:= c(\tau, \sigma) \sqrt{V(0)}$ and $r:=\sqrt{\mphi}r_\varepsilon$. 
    Note that the assumption $\hat x\in\supp(\nabla\phi^*)_\sharp\mu_0$ together with \cref{ass:bregman_phi} implies that $\mu_0(B_{{\sqrt{\mphi} 
 r_{\varepsilon}}/{2}}^{*}(\xhat))>0$ which makes $\alpha_0$ is well-defined.
    It will be clear soon in \labelcref{eq: m_xhat_alpha_initial} that the value of $B$ defined above satisfies the lower bound condition $B\geq \sup_{t \in [0, T_{\alpha,\varepsilon}]} |\mamu - \xhat|$ in \cref{prop: lower_bound_B}. Due to the choice of $\alpha$ in \labelcref{eq: alpha}, note that \begin{align*}
        -\alpha q_{\varepsilon}< \log\left(\frac{c(\tau, \sigma) \sqrt{\varepsilon \mphi}}{4\sqrt{V(0)} }\right)  + \log \rho_0(B_{r_{\varepsilon}}(\xhat))\,,
    \end{align*} 
    where we used $pT^*\geq 0$ as well as \cref{ass:bregman_phi} and the definition of $\rho$ to get
    \begin{align*}
        \rho_0(B_{r_{\varepsilon}}(\xhat)) \geq \mu_0\left (B_{\sqrt{\mphi} 
 r_{\varepsilon} }^{*}(\xhat) \right)  \geq \mu_0\left (B_{ \frac{\sqrt{\mphi} 
 r_{\varepsilon}}{2}}^{*}(\xhat) \right).
    \end{align*}
Therefore, we can continue the chain of inequalities for \labelcref{eq: application_mx_bound} as follows: 
    \begin{align}\label{eq: m_xhat_alpha_initial}
        |m_{\alpha}^*(\mu_0) - \xhat| &\leq \frac{(q_{\varepsilon} + J_{r_{\varepsilon}})^{\nu}}{\eta} + \frac{e^{-\alpha q_{\varepsilon}}}{\rho_0(B_{r_{\varepsilon}}(\xhat))} \int_{\R^d} |\nphiy - \xhat| \de \mu_0(y) \notag
        \\ 
        & \leq \frac{c(\tau, \sigma)\sqrt{\varepsilon}}{2} + \frac{e^{-\alpha q_{\varepsilon}}}{\rho_0(B_{r_{\varepsilon}}(\xhat))} \sqrt{\frac{V(0)}{{\mphi}}} \notag
        \\ 
        & \leq \frac{c(\tau, \sigma)\sqrt{\varepsilon}}{2} + \exp\left(\log\left(\frac{c(\tau, \sigma) \sqrt{\varepsilon {\mphi}}}{4\sqrt{V(0)} }\right) + \log \left( \rho_0(B_{r_{\varepsilon}}(\xhat)) \right)
        \right)
        \frac{\sqrt{V(0)}}{\rho_0(B_{r_{\varepsilon}}(\xhat))\sqrt{{\mphi}}} \notag
        \\ &\leq c(\tau, \sigma) \sqrt{\varepsilon}  < c(\tau, \sigma) \sqrt{V(0)} = B = C(0)\,,
    \end{align}
    where the inequalities follow from \labelcref{eq: q_choice,eq: alpha}. Therefore, we have shown that $T_{\alpha, \varepsilon} > 0$. 
    To analyze the behavior of $t\mapsto V(t)$ on the time interval $[0,T_{\alpha,\varepsilon}]$, we start by showing that an exponential decay estimate is true for all $t \in [0, T_{\alpha, \varepsilon}]$. For this, we use \cref{prop: dynamics_V}, the definition of $T_{\alpha, \varepsilon}$ in \labelcref{eq: T_alpha}, and the first two terms in the definition of $c(\tau,\sigma)$ to get: 
    \begin{align*}
        \frac{\d}{\d t}V(t) & \leq 
        -\left(\frac{2m_\phi - d\sigma^2}{2m_\phi M_\phi}\right)
        V(t)
        +
        \left(
        1+
        \frac{d\sigma^2}{m_\phi}
        \right)
        \abs{
        \mamu-\hat x
        }
        \sqrt{m_\phi^{-1}V(t)}
        +
        \frac{d\sigma^2}{2m_\phi}
        \abs{\hat x - \mamu}^2 
        \\ 
        &\leq 
        - (1 - \tau)\left(\frac{2m_\phi - d\sigma^2}{2m_\phi M_\phi}\right)
        V(t)\,,\qquad\forall t\in[0,T_{\alpha,\varepsilon}].
    \end{align*}
    With the same arguments and using the third term in the definition of $c(\tau,\sigma)$ we also have the lower bound
    \begin{align*}
         \frac{\d}{\d t}V(t) & \geq 
         -M_{\phi}^{-1}V(t) - |\mamu - \xhat| \sqrt{M_{\phi}^{-1}V(t)}
         \geq - \left( M_{\phi}^{-1} + \frac{c(\tau, \sigma) }{\sqrt{
         M_{\phi}} 
         }\right) V(t) 
         \\
         &\geq 
         -\left(1+\frac{\tau}{2}\right)
         \left(\frac{2m_\phi - d\sigma^2}{2m_\phi M_\phi}\right)
         V(t)\,,
         \qquad\forall t\in[0,T_{\alpha,\varepsilon}].
    \end{align*}
    Therefore, Gr\"onwall's inequality shows that for all $t\in[0,T_{\alpha,\varepsilon}]$ it holds
    \begin{align}\label{eq: exp_V_decay_1}
        V(0)\exp\left( - (1 + \tau/2)\left(\frac{2m_\phi - d\sigma^2}{2m_\phi M_\phi}\right) t\right) \leq V(t) \leq V(0)\exp\left( - (1 - \tau)\left(\frac{2m_\phi - d\sigma^2}{2m_\phi M_\phi}\right) t\right)\,,
    \end{align}
    and as a consequence of the upper bound in \labelcref{eq: exp_V_decay_1} and the definitions of $T_{\alpha,\varepsilon}$ and $C(t)$ we get
    \begin{align}\label{eq: C0_bound}
        \max_{t \in [0, T_{\alpha, \varepsilon}]} |\mamu - \xhat| \leq \max_{t \in [0, T_{\alpha, \varepsilon}]} C(t) \leq C(0)\,,
    \end{align}
    which we shall use later in the proof.
    We shall continue by distinguishing three cases: 
    \\
    \fbox{Case 1 ($T_{\alpha, \varepsilon} \geq T^*$)}
    By the definition of $T^*:= \frac{2m_{\phi} M_{\phi}}{(1 - \tau)(2m_{\phi} - d\sigma^2)}\log\left(\frac{V(0)}{\varepsilon}\right)$ as in \labelcref{eq: Tstar} and the upper bound on $V(t)$ from \labelcref{eq: exp_V_decay_1}, we have $V(T^*) \leq \varepsilon\,.$ By definition of $T_{\alpha, \varepsilon}$ in \labelcref{eq: T_alpha}, and continuity of $V(t)$, we have $V(T_{\alpha, \varepsilon}) = \varepsilon$ and $T_{\alpha, \varepsilon} = T^*$.
    \\
    \fbox{Case 2 ($T_{\alpha, \varepsilon} < T^*$ and $V(T_{\alpha, \varepsilon}) \leq \varepsilon$)} 
    By the definition of $T_{\alpha, \varepsilon}$, $V(T_{\alpha, \varepsilon}) \leq \varepsilon$ implies $V(T_{\alpha, \varepsilon})=\varepsilon$ thanks to the continuity of $V(t)$.
    By the lower bound in \labelcref{eq: exp_V_decay_1}, we have \begin{align*}
        V(0)\exp\left( - (1 + \tau/2)\left(\frac{2m_\phi - d\sigma^2}{2m_\phi M_\phi}\right) T_{\alpha, \varepsilon}\right) \leq V(T_{\alpha, \varepsilon}) = \varepsilon
        =V(0)\exp\left( - (1 - \tau)\left(\frac{2m_\phi - d\sigma^2}{2m_\phi M_\phi}\right) T^*\right)
        \,,
    \end{align*}
    where the last equality follows from the definition of $T^*$. This inequality can be reordered to
    \begin{align*}
        \frac{1 - \tau}{1  + \frac{\tau}{2}}T^* \leq T_{\alpha, \varepsilon} < T^*\,.
    \end{align*} 
\fbox{Case 3 ($T_{\alpha, \varepsilon} < T^*$ and $V(T_{\alpha, \varepsilon}) > \varepsilon$)}
This is the only nontrivial case and we will show that it cannot occur by construction. 
Since in this case we assume $V(T_{\alpha,\varepsilon})>\varepsilon$ we will obtain a contradiction to the definition of $T_{\alpha,\varepsilon}$ once we show that $\abs{m_\alpha^*(\mu_{T_{\alpha,\varepsilon}})-\hat x}<C(T_{\alpha,\varepsilon})$.

Recap that by the choice of $\alpha$ in \labelcref{eq: alpha} we have
\begin{align*}
    \exp\left( -\alpha q_{\varepsilon} \right) < \exp \left( \log\left(\frac{c(\tau, \sigma) \sqrt{\varepsilon {\mphi} }}{4\sqrt{V(0)} }\right) - pT^*  + \log \mu_0\left (B_{ \frac{\sqrt{\mphi} 
 r_{\varepsilon}}{2}}^{*}(\xhat) \right)\right) \,.
\end{align*}
Using this together with the fact that $\varepsilon<V(T_{\alpha,\varepsilon})$ and applying \cref{prop: lower_bound_B,prop: m-xhat_bound} 
we have
\begin{align*}
    |m_{\alpha}^{*}(\mu_{T_{\alpha, \varepsilon}}) - \xhat| &\leq \frac{(q_{\varepsilon} + J_{r_{\varepsilon}})^{\nu}}{\eta} + \frac{e^{-\alpha q_{\varepsilon}}}{\rho_{T_{\alpha, \varepsilon}}(B_{r_{\varepsilon}}(\xhat))} \int_{\R^d} |\nphiy - \xhat| \de \mu_{T_{\alpha, \varepsilon}}(x) 
    \\ 
    & \leq \frac{c(\tau, \sigma)\sqrt{V(T_{\alpha, \varepsilon})}}{2} + \frac{e^{-\alpha q_{\varepsilon}}}{\rho_{T_{\alpha, \varepsilon}}(B_{r_{\varepsilon}}(\xhat))} \sqrt{\frac{V(T_{\alpha, \varepsilon})}{{\mphi}}} 
    \\ 
    & \leq \frac{c(\tau, \sigma)\sqrt{V(T_{\alpha, \varepsilon})}}{2} + 
    \frac{2e^{-\alpha q_{\varepsilon}} }{ \mu_0\left (B_{ \frac{\sqrt{\mphi} 
 r_{\varepsilon}}{2}}^{*}(\xhat) \right)   e^{-pT^*}} 
 \cdot \sqrt{\frac{V(T_{\alpha, \varepsilon})}{{\mphi}}} 
    \\ 
    & \leq  \frac{c(\tau, \sigma)\sqrt{V(T_{\alpha, \varepsilon})}}{2} 
    + 
    \frac{\sqrt{\varepsilon {\mphi}} } { 2\sqrt{V(0)}} 
 \cdot \sqrt{\frac{V(T_{\alpha, \varepsilon})}{{\mphi}}} \\
    &< c(\tau, \sigma) \sqrt{V(T_{\alpha, \varepsilon})} = C(T_{\alpha, \varepsilon})\,.
    \end{align*}
    Note that we used \cref{prop: lower_bound_B} in the third inequality to bound $\rho_{T_{\alpha, \varepsilon}}$ from below, and we used $\varepsilon < V(0)$ in the last one.
Hence, we derived the desired contradiction, and thus Case 3 cannot happen.
\end{proof}

\section{Applications and numerical results}
\label{sec:numerics}

In this section, we evaluate the numerical performance of the considered MirrorCBO algorithm. To do so, we first present the algorithm and details on its implementation. Here, and in the following, we denote by $\en{x}_{k} = (x^{(1)}_k, \ldots, x^{(N)}_k)$ the ensemble at iteration $k$ and analogously the dual ensemble $\en{y}_{k}$. With a slight abuse of notation, we also perform binary operations like $+,-,*,/$ and apply functions defined on $\R^d$ on whole ensembles, where the evaluation is to be understood particle-wise. For notational clarity, we denote by $x^{(n)}_k$, the $n$th particle in the ensemble at iteration $k$, and by $(x)_i$ the $i$th component of a vector $x\in\R^d$.
%
%
\begin{algorithm}
\caption{Mirror consensus-based optimization}\label{alg:MirrorCBO}
\textbf{Input:}  Initial ensemble $\en{x}_{0}$, distance generating function $\mm$, hyperparameters $\sigma\geq 0, \alpha, \tau>0$ 
\begin{algorithmic}[1]
\State $\en{y}_{0} \in \partial \mm (\en{x}_{0})$
\While{\textbf{Not Terminate()}}
    \State $\en{m}_{k+1} = \textbf{ComputeConsensus}(\en{x}_{k}, \alpha_{k})$
    \State $\en{y}_{k+1} = \en{y}_{k} - \tau\, (\en{x}_{k} -  \en{m}_{k}) + \sigma \,\textbf{Noise}(\en{x}_{k} -  \en{m}_{k}, \tau)$
    \State $\en{x}_{k+1} = \nabla \phi^*(\en{y}_{k+1})$
    \State \textbf{PostStepRoutine()}
\EndWhile
\end{algorithmic}
\end{algorithm}
%
\paragraph{Computing the consensus point}
The first relevant ingredient of \cref{alg:MirrorCBO} is the computation of the consensus point. For numerical stability, we employ a LogSumExp trick, i.e., using the function\footnote{In our numerics we employ the efficient implementation of a LogSumExp function, provided by the Python packages \href{https://docs.scipy.org/doc/scipy/reference/generated/scipy.special.logsumexp.html}{\texttt{SciPy}} \cite{2020SciPy-NMeth} and \href{https://pytorch.org/docs/stable/generated/torch.logsumexp.html}{\texttt{PyTorch}}.}
\begin{align*}
\textbf{LogSumExp}(\en{x}) = \log\left(\sum_{n=1}^N \exp(x^{(n)})\right)
\end{align*}
we can rewrite the computation of the consensus point as presented in \cref{alg:consensus}.
\paragraph{Noise methods}
The next important ingredient is the \textbf{Noise} function. Here, we mainly consider the case of \emph{isotropic} noise and \emph{anisotropic} noise, as introduced in \cite{carrillo2021consensus}. This is implemented by choosing the function \textbf{Noise} to be given by \cref{alg:isonoise} or \cref{alg:anisonoise}, respectively, which we define on the particle level.

\paragraph{Hyperparameter updates and post-processing}
Since the performance of consensus based algorithms strongly depends on the choice of hyperparameters, one often employs update strategies for some of the parameters involved. A key parameter for CBO methods is the inverse temperature $\alpha$, which determines the trade-off between exploration and exploitation. Choosing $\alpha$ adaptively is known to be beneficial, providing more equal weight between particles at the beginning, and more strongly favoring the particle with the best function value towards the end. However, all cooling schedules available in the literature are chosen heuristically, and a rigorous justification or choice of optimal cooling schedule is still an open problem.
In most of our experiments, we only employ a very simple update rule for the parameter $\alpha$, see \cref{alg:mutlAlpha}. Namely, in each iteration we increase it by multiplying with a constant factor, which was similarly done in \cite{bungert2024polarized}. Other update strategies for the parameter $\alpha$ are for example based on the effective sample size, as proposed in \cite{carrillo2022consensus}, see \cref{alg:effAlpha}. Furthermore, it is often crucial for the performance to include certain numerical modifications, that are not directly motivated by the time-continuous scheme. In \cref{alg:MirrorCBO}, we aggregate all such hyperparameter updates and modifications in the functions \textbf{PostStepRoutines}. The concrete implementation of this function will be explained in the respective examples.
\begin{remark}
For very large values of $\alpha$ the consensus point can be effectively replaced by 
\begin{align*}
\en{m}(\en{x}) = \argmin_{x^{(n)}, n=1,\ldots,N} \obj(x^{(n)}),
\end{align*}
which helps the performance of the algorithm, especially in later stages of the iteration. This was also observed in \cite{carrillo2021consensus}.
\end{remark}

\paragraph{Termination criteria}
Many particle based methods terminate the algorithm, depending on the iteration behavior, e.g., when the consensus point or the whole ensemble do not move significantly between two iterations. In most of our experiments, we rather consider a constrained resources environment, where we cannot freely choose how long an algorithm should run. This translates to allowing only a fixed number of iterations and stopping the algorithm once this number is reached. As a further advantage, this allows for an easy comparison of the methods considered in the following experiments. Note, that this is typically a harder scenario than other termination criteria used in the literature.
\paragraph{General comments on the implementation}
All our experiments are implemented in \texttt{Python} \cite{van1995python}, mainly building upon the package \texttt{CBXPy} \cite{bailo2024cbx}, \texttt{NumPy} \cite{harris2020array} and \texttt{PyTorch} \cite{paszke2017automatic}. To ensure reproducibility, all experiment code files are provided in our GitHub repository\footnote{\url{https://github.com/TimRoith/MirrorCBX}}, where also the relevant hyperparameters are specified. For convenience, these hyperparameters are also included in \cref{sec:numapp}.

%
%
\subsection{Sparsity}\label{sec:sparse}

In this section we consider sparsity promoting mirror maps, by choosing an elastic net type functional
\begin{align*}
\mmel(x) = \frac{1}{2} \abs{x}^2 + \regp \abs{x}_1,
\end{align*}
with parameter $\regp\geq 0$. 
In particular, we can explicitly compute the inverse mirror map, which is given as the shrinkage operator, namely
\begin{align*}
(\nabla \mmel^* (y))_i = 
(\operatorname{shrink}(y, \lambda))_i:=
\begin{cases}
(y)_i - \lambda &\text{if } (y)_i > \lambda,\\
(y)_i + \lambda &\text{if } (y)_i < -\lambda,\\
0 &\text{otherwise}.
\end{cases}
\end{align*}

\subsubsection{Exact regularization property and dualized CBO}\label{sec:normsphere}
\begin{figure}
\centering
\includegraphics[width=\linewidth]{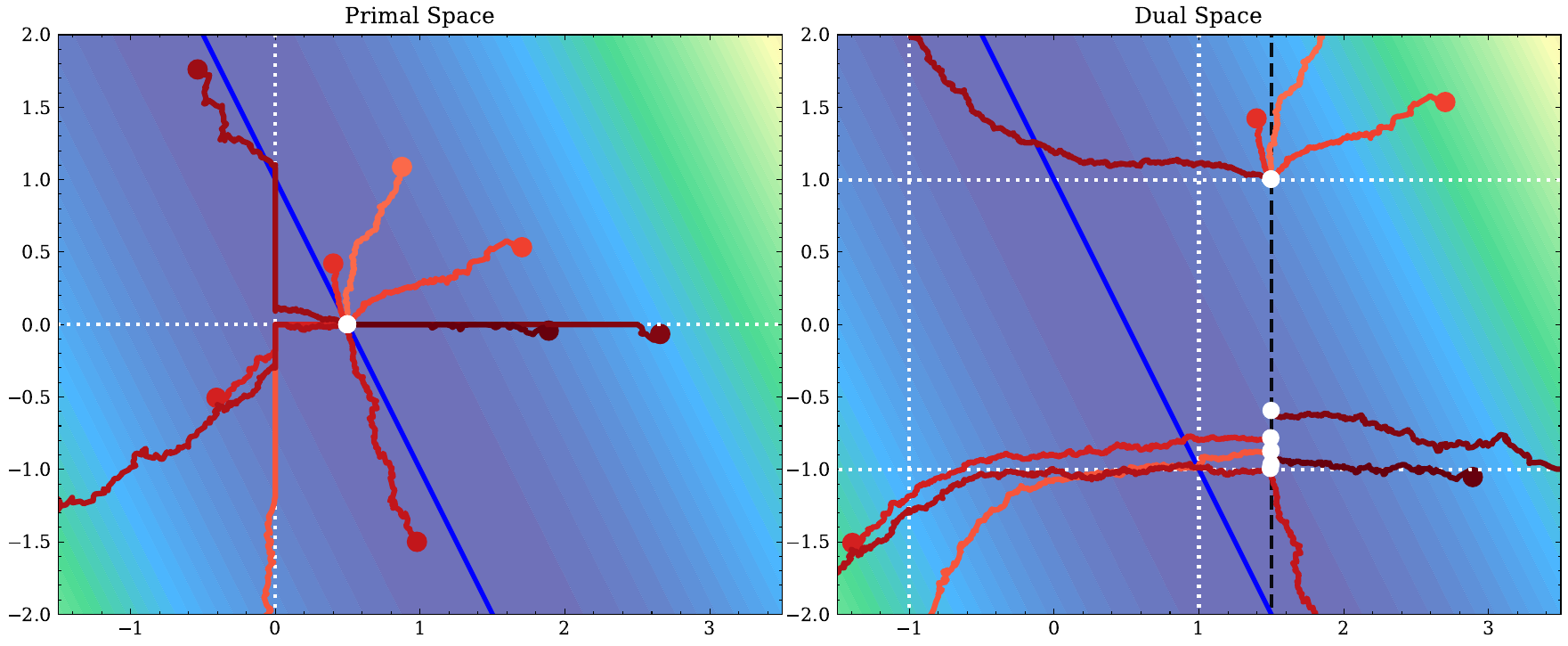}
\caption{Trajectories of the primal and dual particles for the experiment in \cref{sec:normsphere}. The contour illustrates the objective in \labelcref{eq:datafid}, the blue line visualizes the set $\{x:Ax=b\}$.
The particles in the primal space converge to the desired solution $x^*=(0.5,0)$ of the \labelcref{eq:base}. 
The dotted white lines in the primal space, represent the axes which the soft shrinkage operator projects to, i.e., $(x)_1 = 0$ and $(x)_2 = 0$. Respectively, this corresponds to whole intervals in the dual space, with $-1\leq (y)_1 \leq 1$ and $-1\leq (y)_2 \leq 1$. The trajectories of the particles highlight this relationship: while a particle $y$ in the dual space can move freely, governed by the drift and noise terms, once $\abs{(y)_i}\leq 1$ for some $i$, the corresponding particle $x$ in the primal space is glued to the axis $(x)_i = 0$. For more details, see \cref{tab:normsphere}.}
\label{fig:normsphere}
\end{figure}

The first example serves as a qualitative illustration of the MirrorCBO algorithm for this type of mirror map. 
Here, we especially want to highlight the connection to linearized Bregman iterations, as considered in \cite{yin2008bregman}. As cost function we choose a data fidelity type function 
\begin{align}\label{eq:datafid}
\obj(x) = \frac{1}{2}\, \abs{A x - \meas}^2
\end{align}
where in the context of inverse problems $\meas\in\R^{\tilde{d}}$ denotes a measurement vector and $A\in\R^{\tilde{d},d}$ the observation matrix. The set of global minima is given as $\{x:Ax = \meas\}$, where we also consider the case where $A$ has a non-trivial null-space, i.e., when we deal with an underdetermined system. Linearized Bregman iterations were introduced to approximately solve the basis pursuit
\begin{align}\label{eq:base}\tag{Basis Pursuit}
\min\{\abs{x}_1 : Ax = \meas\}.
\end{align}
For a regularization functional $\psi$, the update of linearized Bregman iterations can be written as
\begin{align}\label{eq:LinBreg}\tag{Linearized Bregman}
\begin{split}
p_{k+1} &= p_k - \tau\, \nabla \obj(x_k),\\
x_{k+1} &= \operatorname{prox}_{\psi}(p_{k+1}),
\end{split}
\end{align}
which coincides with the mirror descent iteration \labelcref{eq:MD_iteration} for the distance generating function $\phi=\frac12\abs{\cdot}^2+\psi$.
\begin{remark}\label{rem:breginit}
Bregman iterations are typically initialized at a point where $\psi(x) = 0$. There is no direct analogy to the CBO case, since we cannot initialize the whole ensemble at the same point, which would mean there is no particle evolution. Alternatively, one could initialize already sparse particles, to at least ensure that $\abs{x_0}_1$ is small for each particle. However, numerically this is not significantly different from initializing the ensemble $\en{x}$ random. This is due to the fact that already after the first step, the shrinkage operator induces sparsity if $\regp$ is big enough. 
\end{remark}
\paragraph{Exact regularization property}
In \cite{cai2009convergence, yin2010analysis} it is shown that for $\psi = \regp\abs{\cdot}_1$ and $x_0=0$, the iterates of the linearized Bregman iteration converge to a solution of the problem
\begin{align}\label{eq:minel}
\min\left\{\frac{1}{2}\abs{x}^2 + \regp \abs{x}_1 : Ax = \meas \right\}=
\min\{\phi_\mathrm{en}(x): Ax = \meas\}.
\end{align}
We further note that the linearized Bregman iteration for $\tilde{\mm} = \lambda\abs{\cdot}_1$ exactly corresponds to mirror descent with the elastic net mirror map we consider in this section. Therefore, the first interesting property here is that mirror descent recovers a solution to the problem $Ax=\meas$. This is opposed to variational regularization, which instead solves the problem
\begin{align*}
\min_x \frac{1}{2}\abs{Ax - \meas}^2 + \regp\abs{x}_1.
\end{align*}
Beyond this convergence property, \cite[Lemma 6.13]{benning2018modern} proves that if this algorithm arrives at an iterate $x_{k}$, such that $Ax_{k} = \meas$, it follows that 

\begin{align}\label{eq:linBregConv}
x_{k} \in &\argmin_{x\in\R^d} \left\{\mmel(x): Ax = \meas\right\}.
\end{align}
Furthermore, in \cite{yin2010analysis} it is shown that there exists a large enough but finite $\regp>0$, such that the solution of \labelcref{eq:minel} also solves the \labelcref{eq:base}. This property is also referred to as \emph{exact regularization}, see \cite{mangasarian1979nonlinear}. While our analysis does not provide an analogous result to the exact regularization property, we consider the numerical behavior of this situation. Namely, we consider the following setup,
\begin{align*}
A = (2,1), \qquad \meas = 1,
\end{align*}
with $\{x:Ax = \meas\} = \{(x_1,x_2): x_2 = 1- 2\, x_1\}$ and choose $\lambda=1$. In this setting, the solution to both, the \labelcref{eq:base}, and the elastic net version \labelcref{eq:minel} is given by $x^*=(0.5,0)$. 
In \cref{fig:normsphere} we observe that the particles in the primal space indeed converge to the desired point $x^*$. Naturally, this is not the case for standard CBO, where any point on the line $\{x:Ax = \meas\}$ is a valid global optimum.

\paragraph{Penalization strategy}

Alternatively, we can add the $\ell^1$-norm to our objective, i.e., consider
\begin{align*}
\tilde{\obj} = \obj + \regp \abs{\cdot}_1,
\end{align*}
which can be interpreted as a form of penalized CBO as in \cite{CBO_reg_giacomo}. In the context of $\ell^1$ regularization and compressed sensing, this formulation was studied in \cite{riedl2024leveraging}. Since the CBO dynamic does not require differentiability of the objective function, this is a valid modification. However, we note that the minimizers are not on the set $\{x:Ax = \meas\}$ anymore, and the solution strongly depends on the choice of $\regp$, which is not a hyperparameter of the optimizer but rather of the whole problem that is being solved. This effect can be observed \cref{fig:exactreg2}. We observe that both configurations of the $\ell^1$-penalized CBO find the global minimum of $\tilde{\obj}$ up to very high accuracy. However if $\regp$ is too large, this global minimum is far away from the desired point $x^*=(0.5,0)$. Nevertheless in this simple example with $d=2$, choosing $\regp$ small enough allows outperforming MirrorCBO. We attribute this to the following observations:
\begin{itemize}
\item The penalized CBO method allows choosing a very large parameter $\alpha>0$, which drastically increases the performance. For MirrorCBO we noticed numerically that this parameter cannot be chosen as freely, because otherwise there can occur an early collapse at a point that minimizes the objective function, but is not sparse. This is most likely due to the fact that the strength of the parameter $\regp$ of the mirror map is connected to the strength of $\alpha$.
\item As observed for linearized Bregman iterations in \cite{osher2010fast}, the particles in the MirrorCBO algorithms also encounter stagnation phases. This  happens when the respective dual particles are below the threshold $\regp$, which can be solved by \enquote{kicking} the dual particle out of the stagnation zone. We leave this possible modification of MirrorCBO for future work. Additonally, in\cref{sec:failexact} we visualize a failure case of the algorithm, which is also connected to such stagnation phases.
\end{itemize}
While the above observations slightly hinder the performance of MirrorCBO, there is a simple modification to overcome these issues. By regularizing the objective with an $\ell^0$ term, we can preserve sparsity of the solution, while also choosing $\alpha>0$ large for better performance, since we still recover $x^*=(0.5,0)$ as a global optimum. We note that the $\ell^0$-regularized objective with standard CBO does not yield satisfactory results. This is due to the fact, that for standard CBO the particles rarely have zero-entries and therefore the penalization has almost no effect. In contrast, the shrinkage operator in the MirrorCBO case actually sets entries of the particles to zero.
\begin{figure}[tb]%
\centering%
\begin{tikzpicture}
\begin{axis}[ymode=log,width=.6\textwidth,height=.35\textheight,
grid=major,
legend pos=outer north east,
legend cell align={left},
]
\addplot[line width=2pt, DualCBO] table[x expr=\coordindex+1, y index=0] {figs/sparsity/dual_CBO_diff_c.txt};
\addlegendentry{Dualized CBO}
\addplot[line width=2pt, PenalizedCBO, loosely dashdotted] table[x expr=\coordindex+1, y index=0] {figs/sparsity/reg_paramsL0_diff_c.txt};
\addlegendentry{$\ell^0$-Penalized CBO}
\addplot[line width=2pt, PenalizedCBO,] table[x expr=\coordindex+1, y index=0] {figs/sparsity/reg_params_diff_c.txt};
\addlegendentry{$\ell^1$-Penalized CBO, $\regp=.001$}
\addplot[line width=2pt, PenalizedCBO, densely dotted] table[x expr=\coordindex+1, y index=0] {figs/sparsity/reg_params_big_lamda_diff_c.txt};
\addlegendentry{$\ell^1$-Penalized CBO, $\regp=1.$}
\addplot[line width=2pt, MirrorCBO] table[x expr=\coordindex+1, y index=0] {figs/sparsity/mirror_params_diff_c.txt};
\addlegendentry{MirrorCBO}
\addplot[line width=2pt, MirrorCBO, densely dotted] table[x expr=\coordindex+1, y index=0] {figs/sparsity/mirror_paramsL0_diff_c.txt};
\addlegendentry{MirrorCBO + $\ell^0$}
\end{axis}
\end{tikzpicture}
\caption{Mean distance of the consensus point to the minimizer $x^*=(0.5,0)$ in the experiment described in \cref{sec:normsphere}. Each experiment uses $N=150$ particles and is averaged over $1000$ runs. All further hyperparameters are described in \cref{tab:exactreg2}.}\label{fig:exactreg2}
\end{figure}
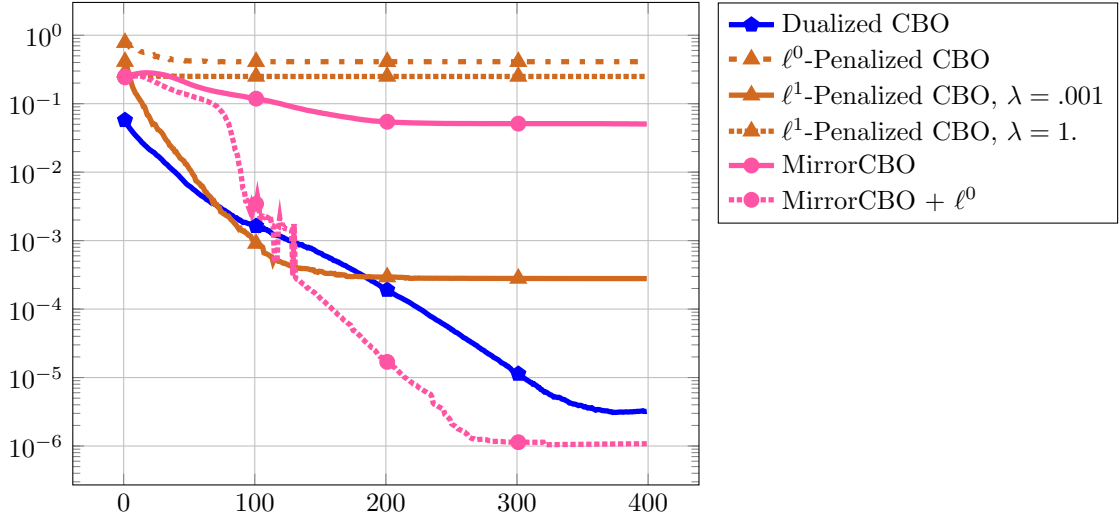

\paragraph{Dualized CBO}
A further interesting insight from \cite{yin2010analysis} is that linearized Bregman iterations in this setup are equivalent to gradient descent on the dual problem
\begin{align}\label{eq:dualcbo}
\min_{v\in\R^{\tilde d}} -\langle \meas, v\rangle +\mm^*(A^T v).
\end{align}
For convenience we consider the rescaled elastic net functional $\mmell(x) = \frac{1}{2\regp} \abs{x}^2 + \abs{x}_1$, where the convex conjugate can be computed as
\begin{align*}
\mmell^*(y) &= 
\left(
\left(\frac{1}{2\regp} \abs{y}^2\right)^{**} + 
\left(\abs{y}_1\right)^{**}
\right)^* 
= 
\left(\frac{1}{2\regp} \abs{y}^2\right)^{*} \square
\left(\abs{y}_1\right)^{*}
=
\inf_z \frac{\regp}{2} \abs{y - z}^2 + \chara_{B^\infty}(z)\\
&= \frac{\regp}{2} \abs{ y - \operatorname{Proj}_{B^\infty}(y)}^2,
\end{align*}
with $B^\infty=[-1,1]^{\tilde{d}}$ denoting the $\ell^\infty$ unit ball. Furthermore, in the above $\square$ denotes the infimal convolution and we refer to \cite{Bauschke_Combettes_2017} for further details and specifically to Proposition 13.24 (i) therein for the second equality. Denoting by $v^*$ the solution of the dual problem \labelcref{eq:dualcbo}, the solution to the original problem can be obtained as
\begin{align*}
x^* = \regp\operatorname{shrink}(A^T v^*, 1).
\end{align*}
We examine the performance of CBO run on the dual problem \labelcref{eq:dualcbo} in \cref{fig:exactreg2}. Here we notice that it outperforms the MirrorCBO algorithm for the primal problem, despite the fact that in the gradient based case they are equivalent. We note that in this case the particles evolve in a $d-1=1$ dimensional space, which might be beneficial for the algorithm. However, a practical drawback of this formulation is that it requires the evaluation of the adjoint operator $A^T$, which standard MirrorCBO does not need. Particle-based methods for inverse problems are particularly relevant when the forward operator is a closed-box and therefore, assuming that the adjoint is available might be unrealistic. Nevertheless, the question of how the dualized CBO formulation relates to the MirrorCBO version is an interesting future direction.

\subsubsection{Sparse deconvolution}\label{sec:deconv}
%
%
%
%
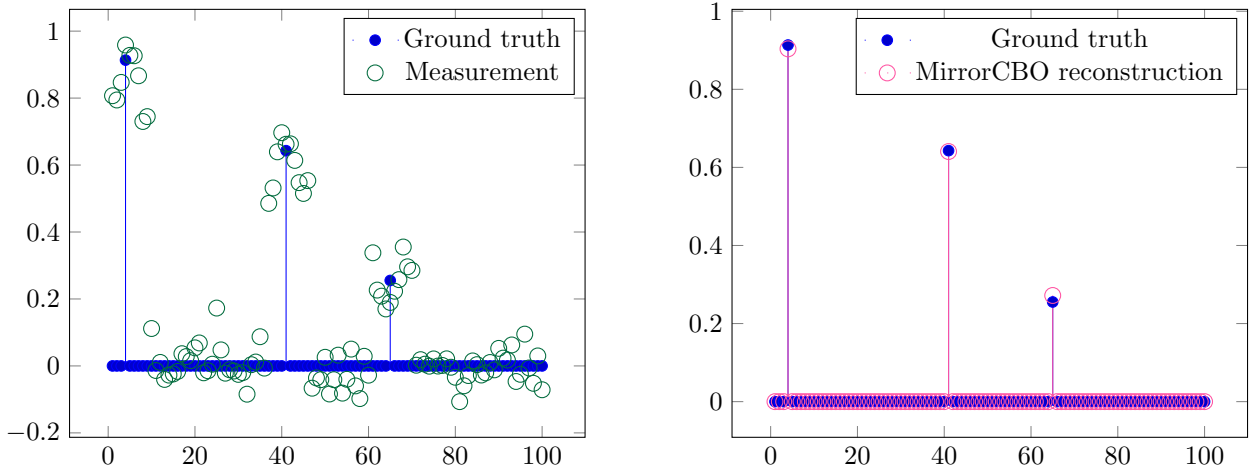
\begin{figure}[t]
\begin{subfigure}{.45\textwidth}
\begin{tikzpicture}
\begin{axis}
\addplot+[ycomb, blue] plot table[x expr=\coordindex+1, y index=0]{figs/sparsity/deconv_result_x.txt};
\addlegendentry{Ground truth}
\addplot+[cadmiumgreen, only marks,mark size=3, mark=o, fill=none] plot table[x expr=(\coordindex+1), y index=0]{figs/sparsity/deconv_result_meas.txt};
\addlegendentry{Measurement}
\end{axis}
\end{tikzpicture}
\end{subfigure}
\hfill%
%
%
\begin{subfigure}{.45\textwidth}
\begin{tikzpicture}
\begin{axis}
\addplot+[ycomb, blue] plot table[x expr=\coordindex+1, y index=0]{figs/sparsity/deconv_result_x.txt};
\addlegendentry{Ground truth}
\addplot+[ycomb, mirror, mark size=3, mark=o, fill=none] plot table[x expr=(\coordindex+1), y index=0]{figs/sparsity/deconv_result_c.txt};
\addlegendentry{MirrorCBO reconstruction}
\end{axis}
\end{tikzpicture}
\end{subfigure}
\caption{On the left: ground truth signal $x$ and corresponding noisy measurement for the convolution in \cref{sec:deconv}. On the right: reconstructed solution with MirrorCBO and ground truth for comparison.}\label{fig:deconv}
\end{figure}
%
%
For the next application, we consider sparse deconvolution with a one-dimensional signal $x\in\R^d$, which is assumed to be sparse. Our forward operator consists of a convolution $C:\R^d\to\R^d$ with a kernel $\kappa\in\R^{K}$, with $K \leq d$
\begin{align*}
(C(x))_i := (\kappa \ast x)_i := 
\sum_{j=0}^{K} (\kappa)_j\, (x)_{i - j},
\end{align*}
where we set $(x)_{i-j} = 0$ for $i-j < 0$ or $i-j > d$. For this experiment, we assume that the kernel has the form
$
(\kappa)_j := 
\exp\left(- \frac{j^2}{2\sigma_\kappa}\right).
$

The inverse problem we want to solve, for the operator $A=C$, can be stated as 
\begin{align*}
\meas^\noiselvl = Ax + \noise
\end{align*}
where $\meas^\noiselvl\in\R^{\tilde d}$ is the given measurement data, $x$ is the signal we want to retrieve and $\noise\in\R^{\tilde{d}}$ models measurement noise. We consider additive Gaussian noise and further assume that we have knowledge about the noise level $\noiselvl=\abs{\noise}$. We again choose $\obj(x):=\abs{Ax - b^\noiselvl}^2$ as the objective function. Regarding the insights from \cref{sec:normsphere}, we observe that also in this scenario the linearized Bregman iterations try to converge to a point $x$ with $Ax = \meas^\noiselvl$. However, in the noisy case, it is not meaningful to achieve an error $\abs{Ax -\meas^\noiselvl}$ below the noise level $\noiselvl$. For linearized Bregman iterations, this problem can be tackled by early stopping, i.e., stopping the iteration once the data fidelity term falls below the noise level. While a similar concept could be used for the MirrorCBO algorithm, one can not necessarily expect that at this iteration the consensus point has the desired sparsity properties. Among other factors, we attribute this to the initialization, since a zero-initialization, $\en{x}_0=0$, is not possible for the algorithm, see \cref{rem:breginit}. In the linearized Bregman case, this is required for the convergence to a solution of the basis pursuit. Using the insights from \cref{sec:normsphere}, we leverage an additional $\ell^0$ term to overcome the issue of fitting the noise. To do so, we also use an adaptive discrepancy principle \cite{morozov1966regularization} for the choice of the regularization parameter $\regp$, which is displayed in the pseudocode \cref{alg:adaptive-disc-principle}.
In this section, we consider dimensions bigger than $d=2$, as in the previous example. In order to prevent an early collapse of the ensemble, we adapt the strategy as employed in \cite{carrillo2021consensus}. Namely, every time a certain collapse criterion is fulfilled, we add noise to the ensemble, with a scale independent of the current ensemble state. This can also be interpreted as restarting the algorithm, with an initialization around the current iterate. We observed that this modification is crucial for the performance of all CBO algorithms. We summarize this function in \cref{alg:indepnoise}. Note that in \cref{alg:MirrorCBO}, this step is integrated into \textbf{PostStepRoutines}. For MirrorCBO, the noise is added onto the dual ensemble $\en{y}$. However, we can apply the same routine for standard CBO algorithms by instead adding it to the standard ensemble. In order to ensure convergence, we decrease the strength of this noise every time the resampling is applied.

For the problem setup, we choose a signal of length $d=100$ and sample $3$ random positions where the peaks appear. The strength of these peaks is distributed uniformly in $[0,1]$. We choose a kernel of size $K=10$ with variance $\sigma_\kappa=2.5$ and a noise level $\delta=0.05\, d$. In \cref{fig:deconv}, we display a visual representation of the problem in this setup. In \cref{tab:conv}, we compare the results of multiple runs with different CBO algorithms. We observe that MirrorCBO outperforms the other algorithms by a large margin. We especially see that including the prior information of sparsity into the problem helps the reconstruction since standard CBO fails to recover an acceptable solution in this case. This is in line with the result in \cite[Fig. 2]{riedl2024leveraging}, where a similar sparse reconstruction problem was considered, with random matrices, however with fewer particles, namely $N=10$ or $N=100$. While in \cite{riedl2024leveraging}, increasing the number of particles has no effect on their problem, we noticed an improvement in performance for our deconvolution example as displayed in \cref{tab:conv}. Furthermore, one observes that the dualized variant does not perform as MirrorCBO. We observed numerically that the hyperparameter selection for this method is difficult. Nevertheless, the dualized variant may have benefits, when the range of the operator $A$ has a lower dimension than its domain, which we intend to discuss in future work.
\begin{table}
\centering%
\begin{tabular}{p{5cm} l l}
\toprule
Algorithm    & Success rate   & $\abs{\en{m}_k}_0$ \\
\midrule
CBO ($N=1000$) & 0\% & 100.\\
\midrule%
\rowcolor{mirror!30}
MirrorCBO ($N=1000$)   & \textbf{90}\% & \textbf{3.4} \\
\rowcolor{mirror!10}
MirrorCBO ($N=100$)   & 69.0\% & 4.5 \\
Dualized CBO ($N=1000$) & 45\%  & 8.01\\
\midrule
Penalized CBO ($N=1000$) & 22\%  & 99.95\\
\bottomrule
\end{tabular}
\caption{Average performance over $100$ different initializations for the sparse deconvolution experiment in \cref{sec:deconv}. Note that for each run, new ground truth signals and noise vectors are sampled. Here, we count a run as successful if the consensus point at the last iterate fulfills $\abs{\en{m}_k - x_{\text{true}}}_\infty \leq 0.1$, where $x_{\text{true}}$ denotes the respective ground truth. Furthermore, we also display the $\ell^0$ value of $\en{m}_k$ at the last iterate, averaged over all runs, where $\abs{x_{\text{true}}}_0=3$. Each run uses $N=1000$ particles and is run for $5000$ iterations. All further hyperparameters are specified in \cref{tab:convp}.}\label{tab:conv}
\end{table}

\subsubsection{Sparse neural networks}\label{sec:nn}

In \cite{bungert2022bregman} linearized Bregman iterations were employed to obtain sparse neural networks. The benefit is that sparsity in the weight matrices yields better computational performance and overall decreases the CO$_2$ emissions that arise both in training and inference of the neural network, see \cite{hoefler2021sparsity}. We now want to obtain sparse neural networks by using a gradient-free training approach with MirrorCBO. While gradient-based methods typically require less network evaluations during the training phase, the sparsity of the weight matrices does not always transfer to the backpropagation, where we refer to \cite{hoefler2021sparsity} for a review of possible techniques. Compared to that, the forward pass (which is the only operation required for CBO) can exploit the sparsity directly.  

For our study, we follow the simple setup from \cite{carrillo2021consensus}, where CBO was used to train a network on the MNIST classification dataset \cite{lecun2010mnist}. We also refer to \cite{huang2023global,riedl2024leveraging,fornasier2022convergence} for further works using CBO to learn neural networks. The success of current machine learning approaches is also connected to the availability of network gradients through backpropagation \cite{rumelhart1986learning,rosenblatt1961principles}. Nevertheless, gradient-free training of neural networks remains an active field of research. We highlight two potential advantages of such optimization strategies:
\begin{itemize}
\item Gradient-based methods tend to get stuck in local minima. Gradient-free approaches, on the other hand, often are less prone to such effects, leading to better approximations of the desired parameters.
\item In a gradient descent-based learning framework, all components of the architecture must be differentiable. Beyond that, they must be arranged such that common effects like gradient vanishing or exploding \cite{bengio1994learning,hanin2018neural} are avoided. Therefore, efficient gradient-free training algorithms, also open up the possibility for new network architectures. In the context of sparse neural networks, the proposed approach here would relate to the setup in \cite{bungert2021neural}. 
\end{itemize}
We note that in this work, we do not address these challenges explicitly. The aim of this section is rather to show, that we can successfully train sparse neural networks with MirrorCBO. Apart from CBO, there are various meta-heuristics employed to train neural networks, e.g., particle-swarm optimization \cite{eberhart1995new}, simulated annealing \cite{haznedar2018training} or differential evolution \cite{storn1997differential}. We refer to \cite{kaveh2023application} for a recent review of meta-heuristics in machine learning. In the context of resource-friendly machine learning, gradient-free training has been employed for example in \cite{di2024gradient}.

We now describe the concrete setup used in this section. On top of the independent noise in \cref{alg:indepnoise}, we employ a further algorithmic modification, proposed in \cite{carrillo2021consensus}. Namely, we only consider a part of the ensemble for the consensus point computation as specified in the following algorithm. We only use the subsampled ensemble for the consensus point, but update the whole ensemble. In \cite{carrillo2021consensus} this is referred to as \enquote{full} updates.
We parametrize a simple network, as follows
\begin{align*}
h_\theta(x) = \operatorname{SoftMax}(\operatorname{BN}(\operatorname{ReLU}(Wx + b))),
\end{align*}
with a weight matrix $W\in \R^{10 \times 28^2}$ and bias vector $b\in\R^{10}$. Here, we employed the following functions for a vector $z\in\R^{\tilde{d}}$,
\begin{gather*}
\operatorname{ReLU}(z_i) := \max\{0, z_i\},\qquad
\operatorname{SoftMax}(z)_i := \frac{\exp((z)_i)}{\sum_{j=1}^{\tilde{d}} \exp((z)_j)},\qquad\\
\operatorname{BN}(z) := 
\frac{(z^{(b)})_i - \mu(z)_i}{\sqrt{\sigma(z)_i^2 + \epsilon}} \odot (\gamma_{\text{BN}})_i + (\beta_{\text{BN}})_i.
\end{gather*}
The function BN denotes the so-called batch-norm, which is employed since a single query to the network $h_\theta$ is evaluated on a batch of inputs $(z^{(1)}, \ldots, z^{(B)})$. This is not connected to the ensemble structure of particle-methods, but rather comes from the fact that for learning applications, multiple inputs are processed in parallel. The batch-norm has learnable parameters $\gamma_{\text{BN}}, \beta_{\text{BN}} \in\R^{\tilde{d}}$ and we employ the batch mean $\mu(z)_i := \frac{1}{B}\sum_{b=1}^B (z^{(b)})_i$ and variance $\sigma(z)_i := \frac{1}{B}\sum_{b=1}^B ((z^{(b)})_i - \mu(z))^2$, with $\tilde d =10$. This results in a total of $d=7840 + 3 \cdot 10 = 7870$ learnable parameters, which we collect in the parameter variable $\theta=(W, b, \gamma_{\text{BN}}, \beta_{\text{BN}})$. We note that this is a relatively small architecture, even for the MNIST dataset. For better comparison, we also train this network using a standard SGD setup. Given the training dataset $\mathcal{T}_{\text{train}} = \{(X_m,Y_m)\}_{m=1}^{M}$ which consists of $M=60,000$ labeled input-output pairs, the loss function is defined as
\begin{align}\label{eq:trainloss}
\obj(\theta)= \frac{1}{M} \sum_{m=1}^M 
\ell(h_\theta(X_m), Y_m),
\end{align}
where $\ell$ denotes the cross-entropy function. When evaluating this function, one usually does not take all training samples into account at once, but rather randomly divides the train set into so-called disjoint mini-batches $\mathcal{T}_{\text{train}} = \mathcal{B}_1\cup \ldots\cup \mathcal{B}_E$. In each step of the iteration, the sum in \labelcref{eq:trainloss} is then replaced by the sum over one mini-batch $\mathcal{B}_i$. After $E$ steps, every mini batch was considered once, which is referred to as an \emph{epoch}, and we resample the mini batches. For evaluation, we consider an additional labeled data set $\mathcal{T}_{\text{test}}$ which consists of $10,000$ input-output pairs, that are not part of the training set, which allows us to evaluate the accuracy of the trained neural network on unseen data. Furthermore, we are interested in the sparsity of all parameters, which we define as
\begin{align*}
\operatorname{S}(x) := \frac{\#\{i: (x)_i = 0\}}{d}
=
1 - \frac{\abs{x}_0}{d}.
\end{align*}
\begin{table}
\centering%
\begin{tabular}{p{3cm} l l}
\toprule
Algorithm    & Accuracy on test set   & Sparsity\\
\midrule
SGD & 90.2\% & 0\%\\
\midrule\midrule
CBO & \textbf{88.0}\% & 0\%\\
\midrule
\rowcolor{mirror!30}
MirrorCBO    & 86.4\% & \textbf{71.2}\% \\
\end{tabular}
\caption{Test accuracy and sparsity of neural network parameters trained with MirrorCBO and plain CBO on the MNIST dataset. As an indicator for the expressivity of the chosen architecture, we also include the accuracies obtained with SGD. The results are averaged over $10$ different runs, where we let each algorithm run for $40$ epochs. For MirrorCBO and CBO we use mini-batching in the data points, see the explanation below \cref{eq:trainloss}. Furthermore for both optimizers we only use parts of the ensemble to compute the consensus point, as proposed in \cite{carrillo2021consensus}, see \cref{alg:batching}.} 
\label{tab:splearn}
\end{table}%
In \cref{tab:splearn} we compare MirrorCBO in this setup with standard CBO. We first note that the accuracy for the standard case is in line with results reported in \cite{carrillo2021consensus}. In contrast to the other experiments, the final evaluation is not done on the consensus point, but rather on the particle that achieved the lowest energy during the iteration, $x^\text{best}$. This is done to preserve the desired sparsity properties in the MirrorCBO case. Regarding this, we see that the MirrorCBO algorithm achieves a very high level of sparsity. This is in particularly interesting, since the given network is already relatively small. Nevertheless, MirrorCBO effectively only uses less than $30\%$ of the available parameters. 
We observe a similar drop in accuracy as reported in \cite{bungert2022bregman}. In \cref{fig:spmatrix}, we visualize the weight matrices obtained with the sparse MirrorCBO learning approach.

\begin{figure}
    \centering
    \includegraphics[width=\linewidth]{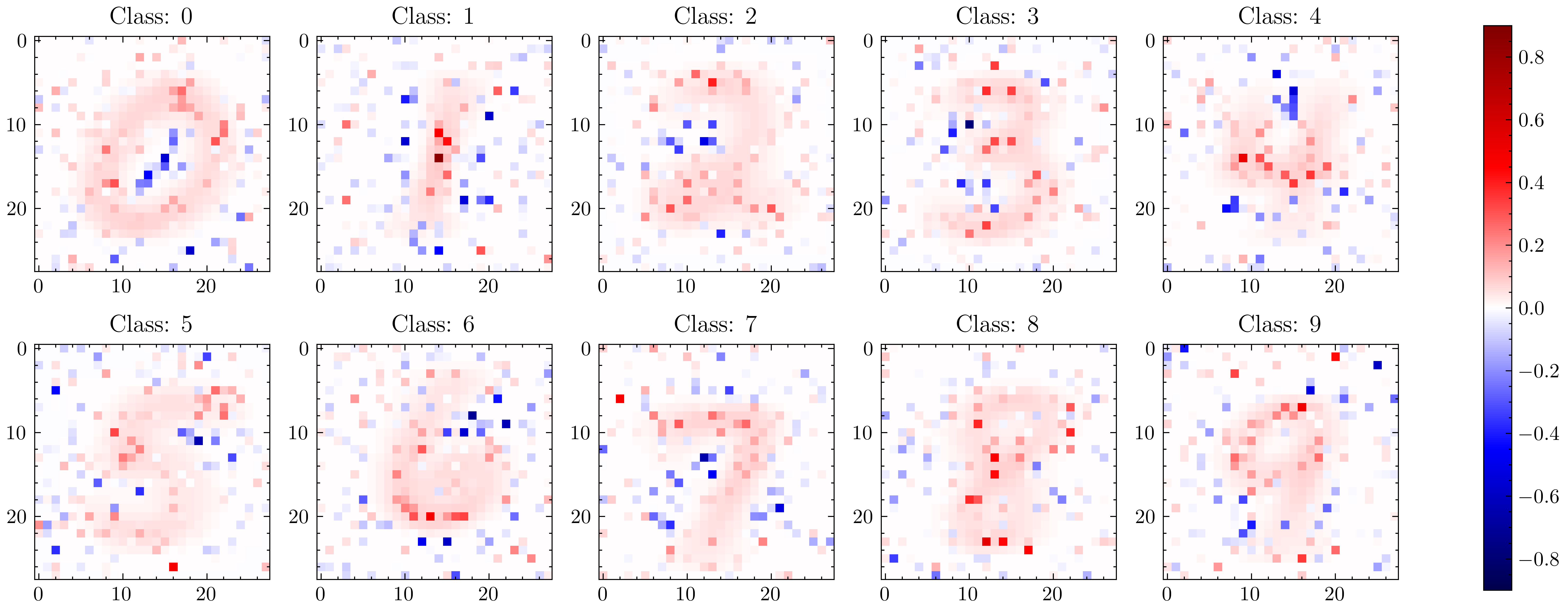}
    \caption{Visualization of the sparse weight matrices in \cref{sec:nn} obtained with MirrorCBO. For each class $i=0,\ldots, 9$, we show the row $W_{i, :}$ reshaped into a square $28\times 28$ matrix. Furthermore, we superimposed the mean of all training images of the respective class, to observe a correlation between the entries in the matrix and the respective class. For the neural network we considered, one can interpret each entry of the row $W_{i, :}$ as how important certain pixels are for the decision whether the image belongs to class $i$ or not.}
    \label{fig:spmatrix}
\end{figure}

\begin{remark}\label{rem:mlshort}
We address some shortcomings of our experiments here. The MNIST dataset is often deemed too trivial, to capture more intricate details of real machine learning applications. On top of that, we employ a very simple architecture that achieves only around $90\%$ test accuracy on MNIST, while various current approaches can easily achieve a validation error below $1\%$. The reasons, why we restrict ourselves to this small-scale setting, are two-fold. Firstly, while training neural networks is not the main focus of this work, it is still an important application. The results here shows that MirrorCBO can be used to obtain sparse neural networks. Secondly, particle-based optimization algorithms do not scale as well to very high-dimensional problems as gradient-descent.  

These shortcomings, paired with the promising results, offer a strong incentive for future work, exploring more difficult machine learning applications. Moreover, as mentioned in the introduction of this subsection, we expect that CBO can profit from the sparsity of the weight matrices. In the future, we aim at explicitly quantifying the FLOPS reduction and computational benefit of MirrorCBO compared to standard CBO.
\end{remark}
%
%
\subsection{Constrained optimization}\label{sec: constraints_alt}

Mirror descent can also be employed for constrained optimization problems, see, e.g., \cite{bubeck2015convex,stonyakin2019mirror,bostan2012asymptoticfixedspeedreduceddynamics,juditsky2023unifying}. As mentioned in the introduction, there appear subtle differences between some formulations of mirror descent and its lazy variant or Nesterov's dual averaging. In the following, we explore how MirrorCBO can be applied for various types of constraints. We briefly explain the constraints considered in the following and how the differences between mirror and the lazy formulation transfer to our CBO setting.
\begin{itemize}
\item Constraints on the simplex \cref{sec:simplex}: this type of constraint is a standard application for mirror descent. Most importantly, by employing the negative log-entropy as the distance generating function, one recovers so-called exponentiated gradient descent \cite{kivinen1997exponentiated}. Furthermore, in this case, lazy mirror descent and the standard variant coincide. We establish the CBO variant of this algorithm and show its numerical performance on a regression task.
\item Constraints on hypersurfaces and manifolds \cref{sec:constraints_numerics}: CBO with hypersurface constraints has received a considerable amount of attention in existing literature, which we review in the following. At first glance, the application of a mirror-based algorithm here seems unusual, since most constraints on hypersurfaces $\hyp$ are tied to non-convex distance-generating functions. While mirror descent was already considered for quasi-convex functions \cite{stonyakin2019mirror}, the application to hypersurface constraints is not well-studied up to our knowledge.For the implementation of the algorithm in \cref{alg:MirrorCBO}, the only relevant part is the computation of $\nabla \mm^*$. When choosing $\phi = \frac{1}{2}\abs{\cdot}^2 + \iota_\hyp$, this reduces to the projection onto $\hyp$. While the projection is not always unique, or given in a closed form, we discover in the following that MirrorCBO can still outperform many comparable algorithms in this context. For this choice of distance generating function, the mirror descent formulation as in \cite{bubeck2015convex} recovers projected gradient descent, while dual averaging (which is closer to MirrorCBO) yields a slightly different iteration. Therefore, we also compare MirrorCBO to the projected CBO method proposed in \cite{CBO_finance_projection}. However, to the best of our knowledge, this algorithm was not used for hypersurface constraints so far. 

Apart from exploring the versatility of MirrorCBO, this section also serves as a comparative overview of different constrained CBO methods, evaluating their numerical performance on different test problems.
\end{itemize}

\subsubsection{Simplex constraints and exponentiated gradient descent}\label{sec:simplex}

A popular setting is when the constrained set is the simplex
\begin{align*}
C = \left\{x\in\R^d : \sum_{i=1}^d (x)_i = 1, (x)_i\ge 0\right\}
\end{align*}
together with the natural choice of distance generating function as the \emph{negative log-entropy}
\begin{align*}
\mmnl(x) = \sum_{i=1}^{n} (x)_i \log((x)_i).
\end{align*}
In this case incorporating the simplex constraint via ${\mm}(x) := \mmnl(x) + \iota_C(x)$, we obtain
\begin{align*}
\partial \mm(x)=\{\log(x) + \mathbf{1}\}, \text{ for } x \in C,\qquad
\nabla \mm^*(y)=\frac{\exp(y)}{\sum_{i=1}^d \exp((y)_i)},
\end{align*}
where $\mathbf{1}=(1,\ldots,1)$. For mirror descent, this allows to express the update as
\begin{align*}
x_{k+1} &= 
\nabla{\mm}^*(y_k - \tau\, \nabla\obj(x_k)) = 
\nabla{\mm}^*(\log(x_k) + 1 - \tau\, \nabla\obj(x_k))\\
&=
\frac{x_k\odot \exp(-\tau \nabla\obj(x_k))}{\sum_{i=1}^d (x_k)_i\, \exp(-\tau (\nabla\obj(x_k)))_i}
\end{align*}
which is referred to as exponentiated gradient descent \cite{kivinen1997exponentiated}. This update scheme is especially attractive, since each operation can be computed efficiently. This is opposed to directly applying projected gradient descent, where in each step the projection onto the simplex $C$ needs to be computed. As a test problem, we consider robust regression, where we obtain noisy measurements as $\meas = Ax + \noise$ with additive Gaussian noise~$\noise$. The entries of the matrix $A\in\R^{\tilde d\times d}$ are randomly sampled, with the rows being standard normally distributed. The objective is then given as the $\ell^1$ distance of the measurement error
\begin{align*}
\obj(x) = \abs{Ax - \meas}_1.
\end{align*}
For our experiment we choose $d=100, \tilde{d}=200$, varying noise levels of $\abs{\noise} = \tilde{\noiselvl}\, \sqrt{\tilde{d}}$, and sample the ground truth $x$ on the unit simplex. 
\begin{remark}\label{rem:simplexinit}
We draw $x$ on the unit simplex by first obtaining $d$ independent samples $(z)_1, \ldots, (z)_d$ from the exponential distribution 
\begin{align*}
p(t) = 
\begin{cases}
\exp(-t) &\text{if } t \geq 0\\
0 & \text{else}
\end{cases}
\end{align*}
and then setting
\begin{align*}
x = \frac{z}{\sum_{i=1}^d (z)_i}\in C.
\end{align*}
\end{remark}
In \cref{fig:simplex}, we compare the performance between this  version of MirrorCBO and standard CBO. Since the desired ground truth vector fulfills the constraint, at least in the noiseless case, standard CBO can be expected to find the correct solution. However, we observe that especially in the higher noise case, incorporating the prior information of the simplex constraint yields a better reconstruction. An interesting aspect here, is that MirrorCBO is not affected by the choice of the hyperparameters $\tau$ and $\sigma$ in the same way as standard CBO. This can be observed in the choice of hyperparameters in \cref{tab:simplex}. We leave a more detailed study of the influence of the hyperparameters (as done for example in \cite[Fig. 2]{huang2023global}) for future work.

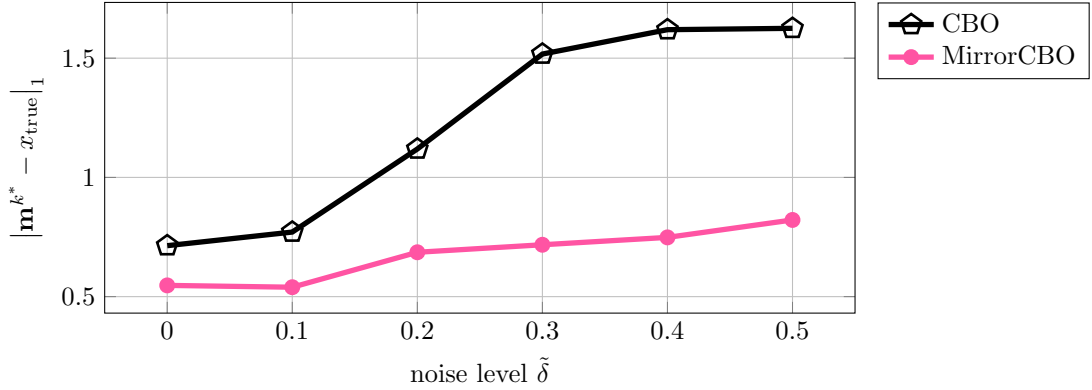
\begin{figure}[tb]%
\centering%
\begin{tikzpicture}
\begin{axis}[
width=.7\textwidth,height=.25\textheight,
grid=major,
legend pos=outer north east,
legend cell align={left},
xlabel = {noise level $\tilde{\noiselvl}$},
ylabel = $\abs{\en{m}^{k^*} - x_{\text{true}}}_1$
]
\addplot[line width=2pt, CBO, mark repeat={1}] table[x index=0, y index=1] {figs/simplex/CBOsweep_res.txt};
\addlegendentry{CBO}
\addplot[line width=2pt, MirrorCBO, mark repeat={1}] table[x index=0, y index=1] {figs/simplex/mirrorsweep_res.txt};
\addlegendentry{MirrorCBO}
\end{axis}
\end{tikzpicture} 
\caption{Error between the consensus point at the last iteration $k^*$ and the ground truth solution. Especially for higher measurement noise, MirrorCBO with the negative log entropy, outperforms the standard variant.}\label{fig:simplex}
\end{figure}

\subsubsection{Constraints on hypersurfaces and manifolds}
\label{sec:constraints_numerics}

As seen in the previous section, for simplex constraints, there is a natural choice of mirror map, which yields an easy update scheme. For equality constraints on sets $C$, where a similar choice is not available, we return to the function
\begin{align}\label{eq:constr}
\mm(x)= \frac{1}{2} \abs{x}^2 + \chara_{C}(x).
\end{align}
In this case we have that the gradient of the conjugate is the projection onto this set and
\begin{align*}
\nabla \mm(x)&= x \quad\text{ for } x\in C,\\
\nabla \mm^*(y) &= \argmin_x \frac{1}{2}\abs{x - z}^2 + \chara_{C}(y) = 
\proj_{C}(y).
\end{align*}
We want to consider the case of a hypersurface or manifold $C=\mathcal{H}$. Note, that typically $\mathcal{H}$ is not convex, and therefore the resulting distance generating function may not fit into our analytical framework. However, there are situations where the proximal mapping is still well-defined and can be computed at least approximately. We refer to \cite{hare2009computing} for a paper studying the proximal mapping for non-convex functions. We compare the mirror CBO algorithm to other constrained CBO variants, which we first review here. Methods can be classified into two groups. One based on adding a regularization term, and the other involving projections. 
\begin{enumerate}
\item \textbf{Regularization-based}: In this section, we only consider equality constraints, where the constrained optimization problem is formulated as follows: 
\begin{align*}
    &\min_{x \in \R^d} J(x)\,\text{ subject to }\\
    &g_i(x) = 0 \text{ for } i = 1, \cdots, m\,.
\end{align*}
The constraint set is then given as $\hyp =  \left\{ x\in \R^d: g_i(x) = 0 \,, \text{ for } i = 1, \cdots, m \right\}$. The respective regularization functional is defined as $G_p(x) = \sum_{i = 1}^m |g_i|^p(x)$, where $p=2$ in \cite{carrillo2024interacting,CBO_reg_jose}. The respective regularization parameter is denoted by $\regp>0$. Algorithms in this group are distinguished depending on how the regularization term is incorporated.
    
\begin{itemize}
\item \textbf{Penalized CBO} \cite{CBO_reg_giacomo}: So-called penalized CBO is obtained by replacing the objective by
\begin{align*}
\tilde\obj(x) := \obj(x) + \regp\, G_p(x)\,,\text{ where } p = 1 \text{ or } 2.
\end{align*}
The standard CBO algorithm is then applied to this problem, where additional care has to be taken in the choice of the regularization parameter $\regp$. Effectively, the regularization only enters the computation of the consensus point, which we denote by $\en{m}_{k,\regp}$.

\item \textbf{Drift-constrained CBO} \cite{carrillo2024interacting}: Unlike the penalized CBO method, authors in \cite{carrillo2024interacting} do not incorporate $G_p$ into the computation of the consensus point, but rather add another drift term $\regp \nabla G_p(x)$ for $p = 2$. This term promotes the particles to be concentrated around the set $\hyp$. Hence, the particles evolve according to following dynamics 
\begin{align*}
\d x_t^{(i)} = -\left(x_t^{(i)} - m_{\alpha}(\rho_t^N)\right)\de t - \regp  \nabla G_p(x_t^{(i)}\de t + \sqrt{2\sigma} |x_t^{(i)} - m_{\alpha}(\rho_t^N)| \de W_t^{(i)}\,.
\end{align*}
The idea to add a regularizing term on the drift was adapted from Cucker--Smale variants, where particles align on specific constraint sets such as a sphere in the Viczek--Fokker--Planck model, see \cite{acevessánchez2019hydrodynamiclimitskineticflocking, bostan2012asymptoticfixedspeedreduceddynamics, bostan2017reducedfluidmodelsselfpropelled, bostan2019fluidmodelsphasetransition}. As discussed in \cite{carrillo2024interacting}, adding a penalty term on the cost function negatively affects the performance of the algorithm. Due to the discrepancy between the actual constrained minimizer of the objective function and that of the regularized objective.
Numerical illustrations in \cite{carrillo2024interacting} substantiate the strength of this approach. At the time-discrete level, the dynamics are discretized via
\begin{align}\label{eq:driftconstr}\tag{Drift-Constr}
\begin{aligned}
\en{v}_k &=
\tau(\en{x}_k - \en{m}_k)
+
\tau \regp
\nabla G_p(\en{x}_k) 
\-
\sigma\, 
\textbf{Noise}(\en{x}_k - \en{m}_k, \tau),\\
\en{x}_{k+1} &= 
\en{x}_k - 
\left(
I + \tau \regp \nabla^2 G_p(\en{x}_k)
\right)^{-1}\en{v}_k.
\end{aligned}
\end{align}
\item \textbf{Combination} of the two methods \cite{CBO_reg_jose}: On top of using the regularized objective for the consensus point as in \cite{CBO_reg_giacomo} with parameter $\lambda_1$, the method proposed in \cite{CBO_reg_jose} adds a drift term $\lambda_2 \nabla G_p(x)$ as in \cite{carrillo2024interacting}. For a single quadratic constraint, $g(x) = \langle x, Ax\rangle - c$ with $p=2$, the following discretization is proposed in \cite{CBO_reg_jose},
\begin{align}\label{eq:combi}\tag{Combi}
\begin{gathered}
\en{v}_{k} = \en{x}_{k} -
\tau(\en{x}_k - \en{m}_{k,\regp_1})
+ 
\sigma\, 
\textbf{Noise}(\en{x}_k - \en{m}_{k,\regp_1}, \tau),\\
(I + 2 \tau\regp_2 g(\en{x}_k) \nabla g)(\en{x}_{k+1}) = \en{v}_k\,.
\end{gathered}
\end{align}
For certain cases, the solution of the above system of equation has an easy closed form solution, which makes the update feasible. The numerical performance of this method was not compared with the former two methods in the previous literatures. In this manuscript, we give an overview of the performance of each method with different constrained optimization tasks.
\end{itemize}
\item \textbf{Projection-based:} Here, we discuss methods that incorporate the constraint via an explicit projection onto the constraint set.
\begin{itemize}
\item \textbf{Hypersurface CBO} \cite{Fornasier2020ConsensusbasedOO, CBO_sphere, CBO_sphere2}:
The initial motivation of these methods was the case where the hypersurface is the sphere. Denoting by $\gamma$ the signed distance to the hypersurface, i.e. $\abs{\gamma(x)} := \operatorname{dist}(x,\hyp)$, the idea is to project the drift and the noise term into the tangent space, via $P(x) := \I_d - \nabla\gamma(x)\nabla\gamma(x)^T$. Together with an additional noise correction, this yields the following dynamics
\begin{align}\label{eq:hyper}\tag{HyperSurface}
\begin{aligned}
\d x_t^{(i)} = &-P(x_t^{(i)})) \left(x_t^{(i)} - m_{\alpha}(\rho_t^N)\right)\de t\\
&+\sigma |x_t^{(i)} - m_{\alpha}(\rho_t^N)| P(x_t^{(i)}) \de W_t^{(i)} \\
&-
\frac{\sigma^2}{2} \abs{x_t^{(i)} - m_{\alpha}(\rho_t^N)}^2
\Delta \gamma(x_t^{(i)}) \nabla \gamma(x_t^{(i)}) \de t
. 
\end{aligned}
\end{align}
At the time-discrete level, this requires an additional projection step to ensure that the particles fulfill the constraint $\gamma(x^{(i)}_t) = 0$.
\item \textbf{Projected CBO} \cite{CBO_finance_projection,beddrich2024constrained}: Projected CBO is the version that is the most similar to the mirror approach. Using the notation from \cref{alg:MirrorCBO}, the update can be written as 
\begin{align*}
\en{x}_{k+1} = 
\proj_\H\left(
\en{x}_{k} - \tau\, (\en{x}_{k} -  \en{m}_{k}) + \sigma \,\textbf{Noise}(\en{x}_{k} -  \en{m}_{k}, \tau)
\right).
\end{align*}
This is the CBO variant of projected gradient descent, where the update can be formulated as
\begin{align*}
x^{\text{PGD}}_{k+1} = 
\proj_\hyp
\left(x_{k}^{\text{PGD}} - \tau \nabla \obj(x^{\text{PGD}}_{k})
\right).
\end{align*}
For example, when $\hyp$ is a hyperplane, i.e., an affine subspace, we have that 
$\proj_\hyp(x+\tilde{x}) = \proj_\hyp(x) + \proj_\hyp(\tilde{x}) - c$ and $\proj_\H(-x)=-\proj_\H(x) + 2c$ for a constant vector $c$, with $\proj_\H(c)=-c$. This yields
\begin{align*}
x^{\text{PGD}}_{k+1} &= x^{\text{PGD}}_{k}-\proj_\hyp
\left(\tau \nabla \obj(x^{\text{PGD}}_{k})
\right) + c 
\\
&=
x_{k-1}^{\text{PGD}} - 
\left(
\proj_\hyp
\left(\tau \nabla \obj(x^{\text{PGD}}_{k-1})
\right)
+
\proj_\hyp
\left(\tau \nabla \obj(x^{\text{PGD}}_{k})
\right)
\right)
+2c
\\
&=
x_{k-1}^{\text{PGD}} - 
\proj_\hyp
\left(\tau \nabla \obj(x^{\text{PGD}}_{k-1})
+
\tau \nabla \obj(x^{\text{PGD}}_{k})
\right)
+c
\\
=\ldots&= 
x^{\text{PGD}}_{0} - \proj_\hyp
\left(
\tau\,
\sum_{i=0}^{k} \nabla \obj(x^{\text{PGD}}_{i})
\right) + c.
\end{align*}
Considering now the mirror descent update, we see that the dual variable can be expressed as $y^{\text{MD}}_{k+1} = x^{\text{MD}}_{0} - \tau \sum_{i=0}^k \nabla J(x^{\text{MD}}_{i})$ and thus 
\begin{align*}
x_{k+1}^{{\text{MD}}} = 
\nabla\phi^*(y_{k+1}^{{\text{MD}}})=\proj_\H(y_{k+1}^{{\text{MD}}})=
x^{\text{MD}}_{0} - \proj_\hyp\left(\tau\, \sum_{l=0}^k \nabla J(x^{\text{MD}}_{l})\right) + c.
\end{align*}
Therefore, if $\hyp$ is a hyperplane, we can show via induction that mirror descent is equivalent to projected gradient descent. Similarly, projected CBO is essentially equivalent to MirrorCBO for the choice of mirror map corresponding to \labelcref{eq:constr} when the constraint set is a hyperplane.
\end{itemize}
\end{enumerate}
\noindent%
Considering the above distinction, the mirror CBO approach falls into the second category of projection-based methods. In the following, we consider the Ackley function \cite{ackley2012connectionist},
\begin{align}\label{eq:Ack}\tag{Ackley}
f(x) = - a 
\exp\left(\frac{b}{\sqrt{d}} \abs{x-x^{\text{shift}}}\right) - 
\exp\left(\frac{1}{d} \sum_{i=1}^d \cos(2\pi\, c\, (x-x^{\text{shift}})_i\right) 
+ e + a
\end{align}
with parameters $a,b,c\in\R$ and a shift vector $x^{\text{shift}}\in\R^d$, which is then the global minimum of $f$ for unconstrained optimization over $\R^d$. We give a brief overview of the different constraints and respective objectives in \cref{tab:constobj}.
\begin{table}[]
\centering
\begin{tabular}{l|lllll}
\toprule
Constraint & Hyperplane  & Sphere & Quadric  & Non-smooth surface & Stiefel manifold\\
\midrule%
Objective & Ackley  & Ackley, Phase retrieval & Ackley & Hölder table & Ackley\\
\bottomrule
\end{tabular}
\caption{Constraints and the respective objectives considered in the following.}
\label{tab:constobj}
\end{table}

%
%
\paragraph{Constraints on hyperplanes.}

In the case when $\mathcal{H}$ is an affine hyperplane, i.e., there is $n_{\hyp}\in\R^{d}, d_\hyp\in\R$, such that 
\begin{align*}
\mathcal{H} = \{x\in\R^n: \langle n_\hyp, x\rangle = d_\hyp\},
\end{align*}
the function in \cref{eq:constr} is in fact convex and we can calculate
\begin{align}\label{eq:planeprox}
\nabla \phi^*(z) = z - \frac{\langle n_\hyp, z\rangle - d_\hyp}{|n_\hyp|^2} n_\hyp.
\end{align}
\begin{figure}%
\centering
\includegraphics[width=.8\linewidth, trim={2cm 0cm 2cm 1cm}, clip]{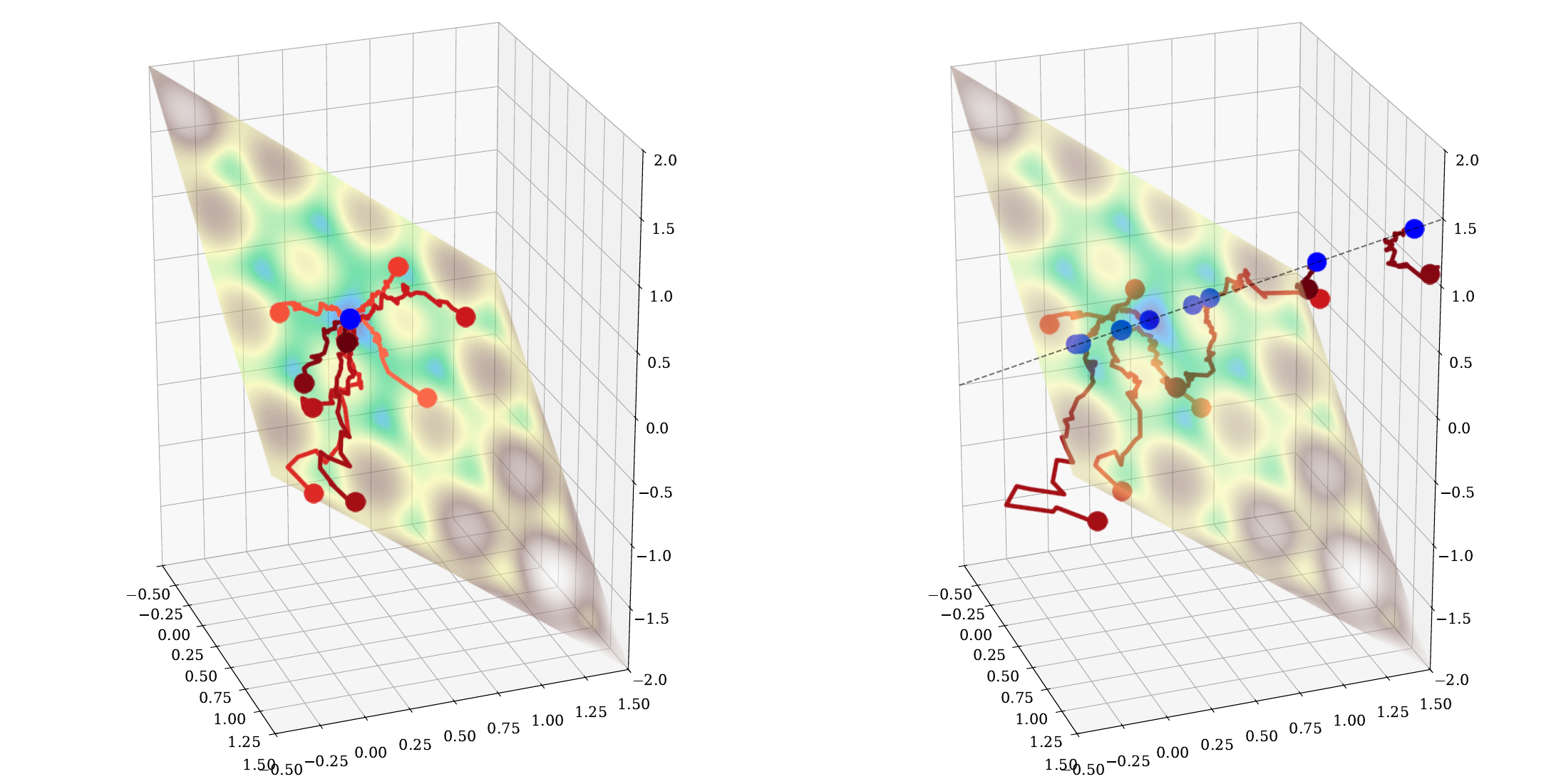}
\caption{The trajectories of the particles for the MirrorCBO algorithm constrained to a hyperplane. On the left, we observe that the particles in the primal space evolve on the hyperplane $\hyp$ and converge to the constrained minimizer. On the right, we observe that the particles in the mirror space are not constrained to the hyperplane. They do not converge to the constrained minimizer, but rather align along the shifted normal vector \{$t\cdot n_\hyp + x^\text{constr} : t\in\R\}$.}
\label{fig:plane}
\end{figure}%
We use the \labelcref{eq:Ack} function defined with parameters $a=20, b=0.1, c=1$ and shift vector $x^{\text{shift}}=(0.4, \ldots, 0.4)$, c.f. \cite[Section 5.2.2]{carrillo2024interacting}, and consider the hyperplane defined by $n_\hyp = (1,\ldots, 1), d_\hyp=2$. In the three-dimensional case $d=3$, the constraint minimizer is given as $x^{\text{constr}} = (1/3,1/3,1/3)$ and we visualize the dynamics in \cref{fig:plane}. We compare the performance of the algorithm to other constrained CBO method in \cref{fig:planecbo}. For more details on hyperparameters, see \cref{tab:plane}.
\begin{itemize}
\item \textbf{Projected CBO} \cite{CBO_finance_projection}: We choose the same projection as defined in \cref{eq:planeprox}. As discussed before, since $\hyp$ is a linear subspace, projected CBO is essentially the same as mirror CBO.
\item \textbf{Penalized CBO} \cite{carrillo2024interacting}: We consider the function
\begin{align}\label{eq:plane_penal}
g(x) = \frac{\langle n_\hyp, x\rangle - d_\hyp}{\abs{n_\hyp}}
\end{align}
and choose the penalty function $G_2 = g^2$. Further, we employ the parameter update rule from \cite[Algorithm 1]{CBO_reg_giacomo} for the regularization parameter. While in \cite{CBO_reg_giacomo} is often chosen with $p=1$, we choose $p=2$ for the following two reasons. For the drift-constrained case, we also consider $p=2$, and thus keeping $p$ fixed ensures comparability. Furthermore, in our experiments, choosing $p=2$ gave a slight advantage over $p=1$ for the penalized variant.
\item \textbf{Drift-constrained CBO} \cite{carrillo2024interacting}: We choose the same constraint function as defined in \cref{eq:plane_penal}, and compute
\begin{align*}
\nabla g(x) = \frac{n_\hyp}{\abs{n_\hyp}_2},\qquad
\nabla^2 g(x) = 0.
\end{align*}
Here, we consider the square of this term which corresponds to choosing $p=2$ for $G_p$.
\item \textbf{Combination} \cite{CBO_reg_jose}: This method was originally proposed for quadratic constraints. However, we can still use the discretization from \labelcref{eq:combi} together with the constraint given by \cref{eq:plane_penal} with $p=2$. Namely, the equation for the update can be solved in closed form, yielding the update
\begin{align*}
\en{x}_{k+1} = \en{v}_k - 2 \tau\regp_2 g(\en{x}_k) \frac{n_\hyp}{\abs{n_\hyp}_2}.
\end{align*}
\end{itemize}
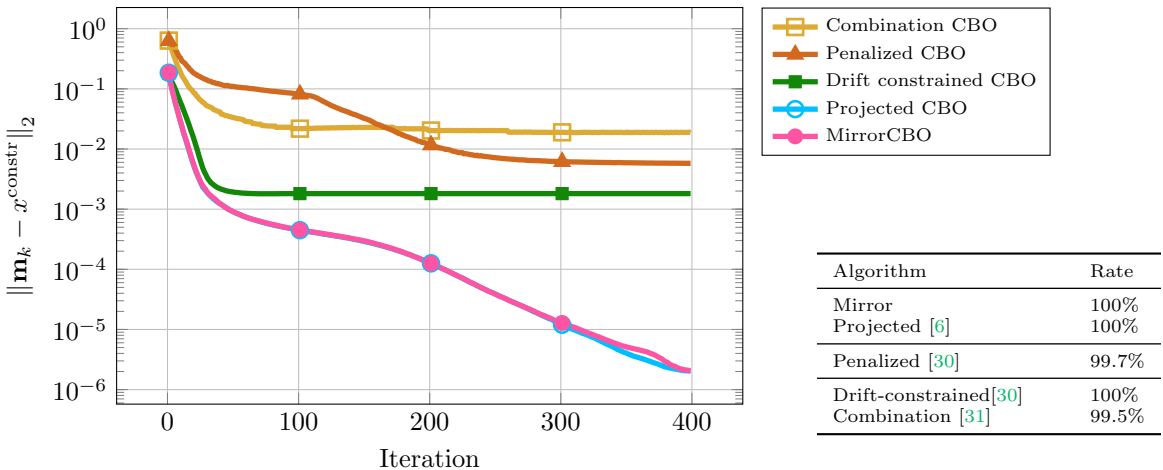
\begin{figure}[htb]%
\centering%
\begin{subfigure}{0.6\linewidth}
\begin{tikzpicture}
\begin{axis}[ymode=log,width=\textwidth,
height=.3\textheight,grid=major,
legend pos=outer north east,
legend style={font=\scriptsize},
legend cell align={left},
xlabel=Iteration, ylabel=$\norm{\en{m}_k - x^\text{constr}}_2$]
\addplot[line width=2pt, Combi] table[x expr=\coordindex+1, y index=0] {figs/plane/regcombi_diff_c.txt};
\addlegendentry{Combination CBO}
\addplot[line width=2pt, PenalizedCBO] table[x expr=\coordindex+1, y index=0] {figs/plane/penalized_diff_c.txt};
\addlegendentry{Penalized CBO}
\addplot[line width=2pt, Drift] table[x expr=\coordindex+1, y index=0] {figs/plane/driftconstr_diff_c.txt};
\addlegendentry{Drift constrained CBO}
\addplot[line width=2pt, ProxCBO] table[x expr=\coordindex+1, y index=0] {figs/plane/prox_params_diff_c.txt};
\addlegendentry{Projected CBO}
\addplot[line width=2pt, MirrorCBO] table[x expr=\coordindex+1, y index=0] {figs/plane/mirror_diff_c.txt};
\addlegendentry{MirrorCBO}
\end{axis}
\end{tikzpicture}
\caption{Distance between the consensus point and the minimizer.}\label{fig:planeq}
\end{subfigure}\hspace{1cm}%
\begin{subfigure}{.25\textwidth}
\scriptsize%
\begin{tabular}{p{3cm} l}
\toprule
Algorithm     & Rate    \\
\midrule
Mirror    & 100\% \\
Projected \cite{CBO_finance_projection}  & 100\% \\
\midrule
Penalized \cite{carrillo2024interacting} & 99.7\% \\
\midrule
Drift-constrained\cite{carrillo2024interacting} & 100\% \\
Combination \cite{CBO_reg_jose} & 99.5\% \\
\bottomrule
\end{tabular}
\vspace{.5cm}

\caption{Success rate, $\text{tol}=0.1$.}
\end{subfigure}%
\caption{Comparison of mirror CBO constrained on a hyperplane with $d=3$ to other constrained CBO methods. We plot the distance of the consensus point to the constrained minimum, averaged over $1000$ runs each. As expected, mirror and projected CBO behave very similar. For more details, see \cref{tab:plane}.}\label{fig:planecbo}
\end{figure}%
\begin{figure}%
\centering
\includegraphics[width=.8\linewidth, trim={1cm 1cm 0.5cm 1cm}, clip]{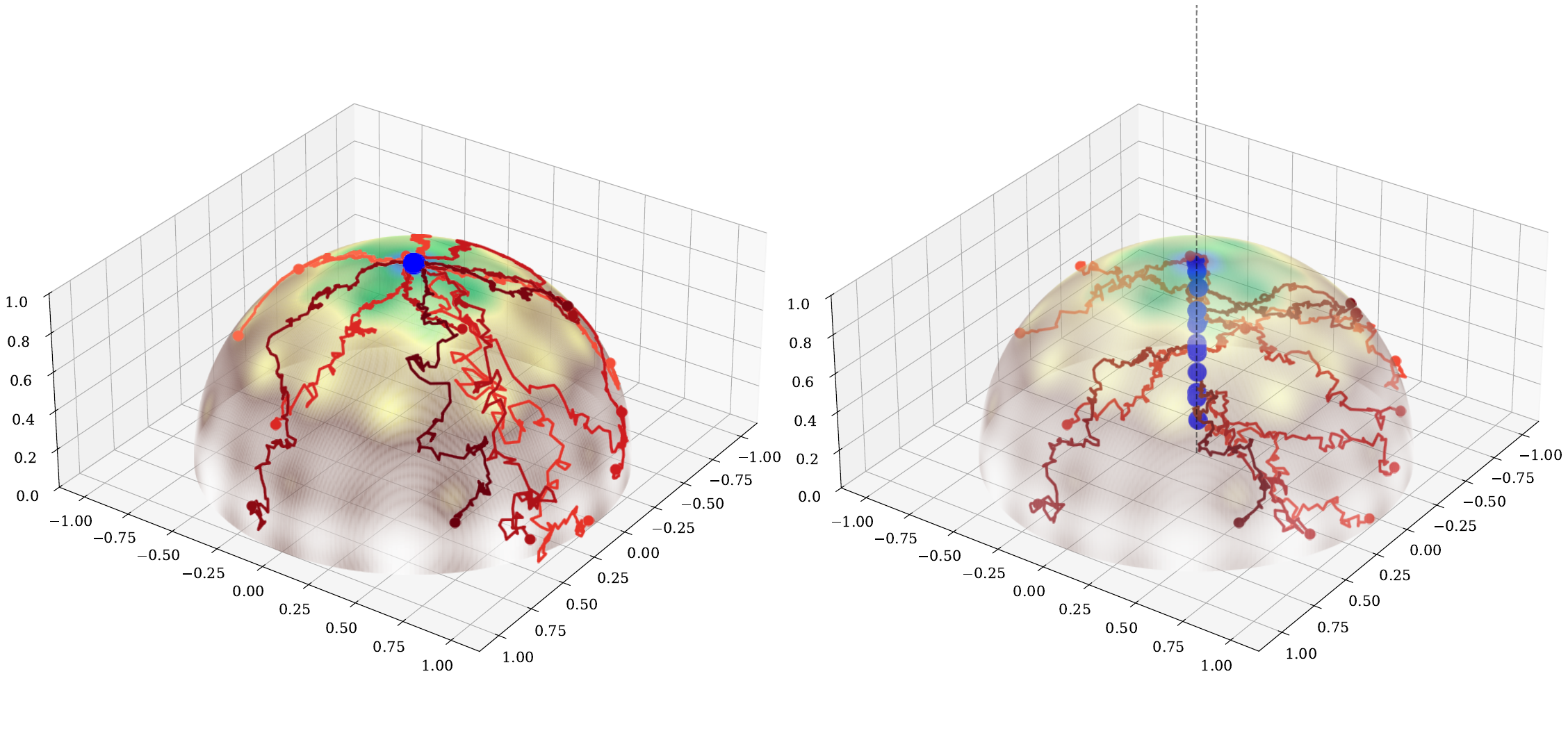}
\caption{We display the MirrorCBO evolution as in \cref{fig:plane}, but for the case where the hypersurface is a sphere. The dual particles on the right do not collapse, but align along the normal vector at the minimizer.}
\label{fig:sphere}
\end{figure}%

\paragraph{Constraints on spheres}

The $d$-dimensional hypersphere $\mathcal{S}^{d-1}$ is not a convex set and therefore also the function in \cref{eq:constr} is non-convex. Nevertheless, the algorithm is applicable, since we can explicitly compute
\begin{align}\label{eq:sphereprox}
\text{Proj}_{\mathcal{S}^{d-1}}(z) 
\in 
\begin{cases}
\left\{\frac{z}{\abs{z}}\right\} &\text{for } z\neq 0,\\
\mathcal{S}^{d-1} &\text{for } z = 0.
\end{cases}
\end{align}%
As a test function we now consider the \labelcref{eq:Ack} function for $d=20$, with parameters $a=20, b=0.1, c = 1$ and $x^{\text{shift}} = (0.4, \ldots, 0.4)$. This corresponds to the setup in \cite[Section 5.2.2]{carrillo2024interacting}, where the constraint minimizer on $\mathcal{S}^{19}$ is given as $x^{\text{constr}} = (1/\sqrt{20}, \ldots, 1/\sqrt{20})$. Again, this is a setting, where the constraint minimizer does not coincide with the global one. We compare the MirrorCBO algorithm to other constraint CBO methods, which we describe in the following.

\begin{itemize}
\item \textbf{Projected CBO} \cite{CBO_finance_projection}: We choose the same projection as defined in \cref{eq:sphereprox}.
\item \textbf{Penalized CBO} \cite{CBO_reg_giacomo}: We choose the function
\begin{align}\label{eq:sphere_penal}
g(x) = \abs{x} - 1
\end{align}
and employ the penalty function $G_2 = g^2$. As before, and employ the parameter update rule from \cite[Algorithm 1]{CBO_reg_giacomo} for the regularization parameter.
\item \textbf{Drift-constrained CBO} \cite{carrillo2024interacting}: We choose the same constraint function as defined in \cref{eq:sphere_penal}, and compute
\begin{align*}
\nabla g(x) = 2 x\qquad
\nabla^2 g(x) = 2 \I_d.
\end{align*}
Again we consider the square of this term, i.e., we choose $p=2$ for $G_p$. In \cite{carrillo2024interacting} the algorithm for this scenario includes a collapse prevention as in \cref{alg:indepnoise}. In order to achieve the same results as reported in \cite{carrillo2024interacting} we also include this for the experiments we conducted. Note, that we only include the independent noise for this algorithm and the combination method. Regarding the results in \cref{fig:sphere,tab:sphere}, we obtain values very close to the ones reported in \cite{carrillo2024interacting}.
\item \textbf{Combination CBO} \cite{CBO_reg_jose}: This setting is the same as considered in \cite{CBO_reg_jose}. We choose the same function $g$ as for the drift-constrained method above. Again the equation in \labelcref{eq:combi} can be solved in closed form yielding the update
\begin{align*}
\en{x}_{k+1} = 
\frac{\en{v}_{k+1}}{1 + 4 \tau \lambda_2 g(\en{x}_k)}.
\end{align*}
Furthermore, we also employ the collapse prevention from \cref{alg:indepnoise}.
\item \textbf{Hypersurface CBO} \cite{Fornasier2020ConsensusbasedOO,CBO_sphere,CBO_sphere2}: This is the standard setting of the method described in \cite{CBO_sphere}. In order to ensure a fair comparison to the other methods, that employ anisotropic noise, we implement the anisotropic version described in \cite{CBO_sphere2}. Our setup with a maximal number of iterations of $400$ is harder than the one considered in \cite[Tab. 1]{CBO_sphere2}, where the average number of iteration is reported to be $2561.4$. Furthermore, we do not consider the fast version of the algorithm, but its baseline proposed in \cite{CBO_sphere}. Nevertheless, we were able to find a set of hyperparameters that yields a performance that is in line with the results in \cite[Tab. 1]{CBO_sphere2}.
\end{itemize}

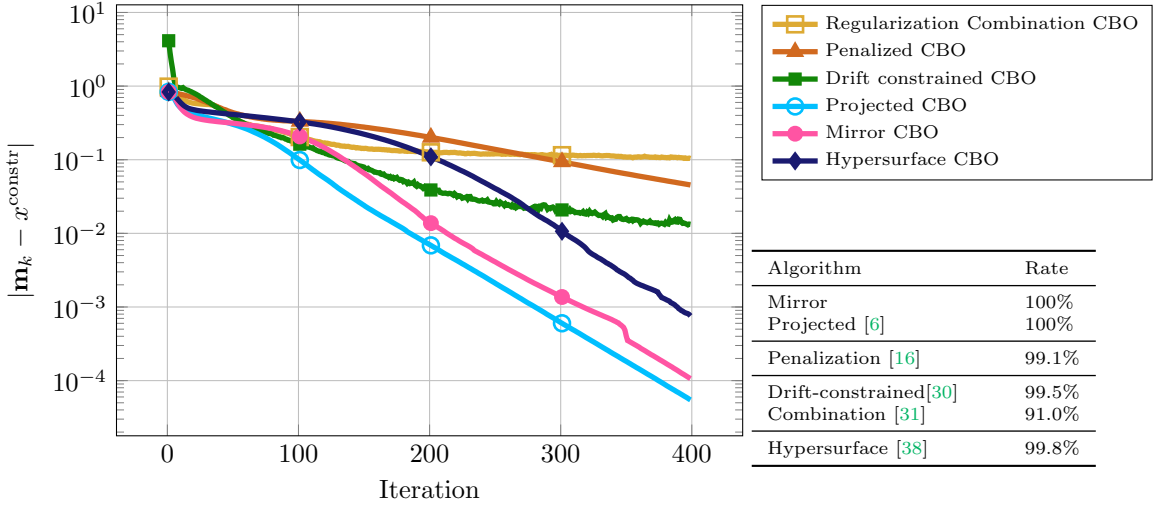
\begin{figure}[htb]%
\centering%
\begin{subfigure}{0.6\linewidth}
\begin{tikzpicture}
\begin{axis}[
ymode=log,
width=\textwidth,height=.32\textheight,
grid=major,legend pos=outer north east,
legend style={font=\scriptsize},
legend cell align={left},
xlabel=Iteration, ylabel=$\abs{\en{m}_k - x^\text{constr}}$]
\addplot[line width=2pt, Combi] table[x expr=\coordindex+1, y index=0] {figs/sphere/regcombi_diff_c.txt};
\addlegendentry{Regularization Combination CBO}
\addplot[line width=2pt, PenalizedCBO] table[x expr=\coordindex+1, y index=0] {figs/sphere/penalized_params_diff_c.txt};
\addlegendentry{Penalized CBO}
\addplot[line width=2pt, Drift] table[x expr=\coordindex+1, y index=0] {figs/sphere/driftconstr_diff_c.txt};
\addlegendentry{Drift constrained CBO}
\addplot[line width=2pt, ProxCBO] table[x expr=\coordindex+1, y index=0] {figs/sphere/prox_params_diff_c.txt};
\addlegendentry{Projected CBO}
\addplot[line width=2pt, MirrorCBO] table[x expr=\coordindex+1, y index=0] {figs/sphere/mirror_diff_c.txt};
\addlegendentry{Mirror CBO}
\addplot[line width=2pt, Hyper] table[x expr=\coordindex+1, y index=0] {figs/sphere/sphere_diff_c.txt};
\addlegendentry{Hypersurface CBO}
\end{axis}
\end{tikzpicture}
\caption{Distance between the consensus point and the minimizer.}\label{fig:sphereq}
\end{subfigure}
\begin{subfigure}{.25\textwidth}
\scriptsize%
\begin{tabular}{p{3cm} l}
\toprule
Algorithm     & Rate    \\
\midrule
Mirror    & 100\%  \\ 
Projected \cite{CBO_finance_projection}  & 100\% \\
\midrule
Penalization \cite{CBO_reg_giacomo} & 99.1\% \\
\midrule
Drift-constrained\cite{carrillo2024interacting} & 99.5\% \\
Combination \cite{CBO_reg_jose} & 91.0\% \\
\midrule
Hypersurface \cite{Fornasier2020ConsensusbasedOO} & 99.8\% \\
\bottomrule
\end{tabular}
\vspace{.5cm}

\caption{Success rate, $\text{tol}=0.1$.}\label{tab:sphere}
\end{subfigure}%
\caption{Comparison of MirrorCBO constrained on the sphere for $d=20$ to other constrained CBO methods. The results are averaged over 1000 runs. For further details, see \cref{tab:sphere}.}\label{fig:spherecbo}
\end{figure}%

\paragraph{Phase retrieval}

A very interesting application pointed out in \cite{Fornasier2020ConsensusbasedOO} is phase retrieval. The forward model of this inverse problem can be formulated as 
\begin{align*}
(y)_m = \abs{\langle f_m, x \rangle}^2 + (\varepsilon)_m =: (F(x))_m + (\varepsilon)_m,
\end{align*}
for $m=1,\ldots, M$, where $y\in\R^M$ is the obtained measurement and $\varepsilon\in\R^M$ denotes measurement noise. We consider the setting where the vectors $f_1,\ldots,f_M\in\R^d$ are real and known, where typically also complex frames are employed. 
In the classical setup $f_m$ is the $m$th column of the conjugate discrete Fourier transform matrix, i.e., $(f_m)_n = \exp(\imath 2\pi \frac{m n}{d})$, we refer to \cite{jaganathan2016phase} for a review on this topic. This problem arises, for example, in the context of diffractive imaging \cite{marchesini2003coherent}. To compare our method to the experiments conducted in \cite[Section 2.4.1]{Fornasier2020ConsensusbasedOO} we choose the phase retrieval setup as in \cite{balan2006signal}, where the vectors $f_m$ are sampled randomly on the sphere, which also allows providing recovery guarantees, by considering the oversampled case $M\geq d$. Here, the vectors $f_1,\ldots, f_M$ form a \textit{frame} of the vector space $\R^d$, if there exist scalars $A,B>0$ such that 
\begin{align*}
A \abs{x}^2 \leq \sum_{m=1}^M \abs{\langle f_m, x \rangle }^2 \leq B \abs{x}^2.
\end{align*}
For frames with $M\geq 2d -1$, \cite[Theorem 2.2]{balan2006signal} states that the measurement operator $F(x)$ restricted to the domain $\R^{d}\setminus \{\pm 1\}$ is injective. For convenience, we now repeat the reconstruction setup as in \cite{Fornasier2020ConsensusbasedOO,sunnen2023analysis}\footnote{We also refer to the MATLAB implementation \url{https://github.com/PhilippeSu/KV-CBO}. Our experiments use a reimplementation in Python. Furthermore, we acknowledge the useful explanations provided by Massimo Fornasier on this topic.}. The unconstrained optimization of an empirical risk type functional can be recast into a constrained optimization over $\mathcal{S}^d$ as follows: First define the zero-padded vector $f_m^{\to\sph^d} = (f_m|0)\in\R^{d+1}$. Using the lower frame bound, we obtain that for any $x\in\R^d$ we have
$\abs{x} \leq \sqrt{1/A} \abs{F(x)}_1^{1/2} =: R(x)$. This allows us to lift any $x\in\R^d$ onto the $d+1$ dimensional sphere by defining
\begin{align*}
x^{\to\sph^d} :=\frac{1}{R(x)} \left(x\middle| \sqrt{R(x)^2 - \abs{x}^2}\right).    
\end{align*}
For a given measurement $y$, we choose $R=\sqrt{\abs{y}_1/A}$ and consider the objective function $\obj:\R^{d+1}\to\R$
\begin{align*}
\obj(x^{\to\sph^d}) := 
\sum_{m=1}^M \abs{ \abs{\langle f_m^{\to\sph^d}, x^{\to\sph^d} \rangle }^2 - \frac{y_m}{R^2} }^2
\end{align*}
where $x^{\to\sph^d}$ is constrained to be on $\sph^d$. In \cref{fig:phase} we compare the MirrorCBO algorithm to the hypersphere CBO algorithm as proposed in \cite{CBO_sphere}. The reconstruction is only unique up to a multiplication with a global sign. After obtaining an approximate solution $\tilde{x}$, we use either $+x$, or $-x$ to evaluate the success, based on which is closer to the ground truth $x$. This alignment is only possible since we already know the ground truth, however it is common to evaluate the numerical performance, see \cite{Fornasier2020ConsensusbasedOO,sunnen2023analysis}. In \cref{fig:phase_noiseless} we consider a high-dimensional ($d=100$) but noise-less setting. Additionally, we compare the gradient-free methods with a gradient-descent based algorithm known as \emph{Wirtinger Flow} \cite{candes2015phase}. For the step sizes of this algorithm, we employed a simple backtracking scheme, which greatly improves the convergence speed. A detailed description of this algorithm is provided in \cref{sec:numapp}. In \cref{fig:phasenoisy}, we consider a lower dimension, however with additive Gaussian noise of varying noise levels. As before, we add a noise of strength $\abs{\varepsilon}= \tilde{\delta} \, \sqrt{d}$, where the factor $\tilde{\delta}$ is varied.

In general, if gradient information for the objective function is available, gradient-based methods tend to perform better. In \cref{fig:phase}, MirrorCBO is competitive with the gradient-based Wirtinger Flow if enough frame vectors are used. Furthermore, MirrorCBO performs similarly as the hypersurface variant. Apart from the constrained setting, one might be interested only in solving the phase retrieval. In this regard, we also employed standard CBO on the original problem. In both cases, this appears to be very effective. In \cref{fig:phasenoisy}, we observed numerically that the Wirtinger flow iterations is prone to running into local minima. In this scenario, CBO performs better, since it is less likely to get stuck in such sub-optimal states.

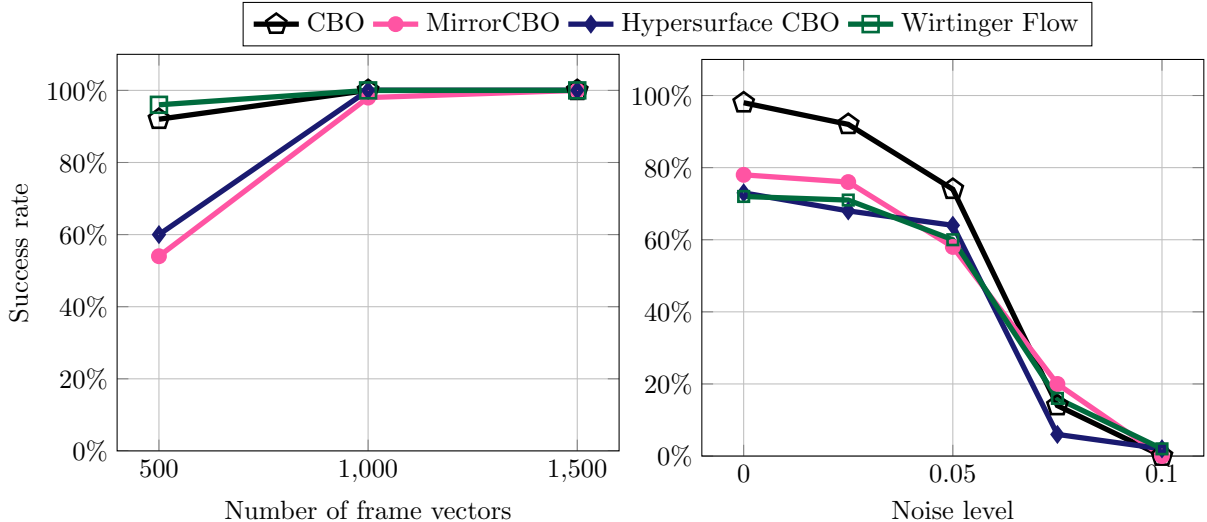
\begin{figure}[t]
\begin{subfigure}[t]{.5\textwidth}
\captionsetup{width=.9\linewidth}
\begin{tikzpicture}
\begin{axis}[width=\textwidth,height=.3\textheight,grid=major,
legend style={
                at={(0.25,1.02)},
                anchor=south west,
            }, 
legend columns=4,
yticklabels ={0$\%$, 20$\%$,  40$\%$, 60$\%$, 80\%, 100\%},
legend cell align={left},
ytick={0, 20, 40, 60, 80, 100},
xtick={500, 1000, 1500},
ymin=0, ymax=110,
xlabel=Number of frame vectors,
ylabel=Success rate]%
\addplot[line width=2pt, CBO, mark repeat={1}] table[] {figs/phase/cbo_scc.txt};
\addlegendentry{CBO};
\addplot[line width=2pt, MirrorCBO, mark repeat={1}] table[] {figs/phase/mirrorcbo_scc.txt};
\addlegendentry{MirrorCBO};
\addplot[line width=2pt, Hyper, mark repeat={1},] table[] {figs/phase/spherecbo_scc.txt};
\addlegendentry{Hypersurface CBO};
\addplot[line width=2pt, color=cadmiumgreen, mark=square,
mark options={line width=1pt, scale=1.5}] table[] {figs/phase/Wirtinger_scc.txt};
\addlegendentry{Wirtinger Flow};
\end{axis}
\end{tikzpicture}
\caption{Comparison in the noiseless case for $d=100$, $N=100$ particles and varying number of frame vectors for the CBO methods. The success rate was averaged over $50$ runs with a success tolerance of $0.05$.}\label{fig:phase_noiseless}%
\end{subfigure}%
\begin{subfigure}[t]{.5\textwidth}%
\captionsetup{width=.9\linewidth}%
\begin{tikzpicture}%
\begin{axis}[width=\textwidth,height=.3\textheight,grid=major,
legend style={at={(0.35,0.5)},anchor=west}, 
yticklabels ={0$\%$, 20$\%$,  40$\%$, 60$\%$, 80\%, 100\%},
ytick={0, 20, 40, 60, 80, 100},
xtick={0, 0.05, 0.1, 0.2},
xticklabels={0, 0.05, 0.1, 0.2},
ymin=0, ymax=110,
xlabel=Noise level]%
\addplot[line width=2pt, CBO, mark repeat={1}] table[] {figs/phase/cbo_noise_scc.txt};
\addplot[line width=2pt, MirrorCBO, mark repeat={1}] table[] {figs/phase/mirrorcbo_noise_scc.txt};
\addplot[line width=2pt, Hyper, mark repeat={1}] table[] {figs/phase/spherecbo_noise_scc.txt};
\addplot[line width=2pt, color=cadmiumgreen, mark=square,
mark options={line width=1pt}] table[] {figs/phase/Wirtinger_noise_scc.txt};
\end{axis}
\end{tikzpicture}
\caption{Comparison for the noisy case with $d=32$, $N=100$ particles and $M=128$ frame-vectors, the success rates are averaged over $50$ runs with a success tolerance of $0.05$.}\label{fig:phasenoisy}
\end{subfigure}
\caption{Comparison between MirrorCBO, hypersurface CBO and Wirtinger flow for different phase retrieval settings.}\label{fig:phase}
\end{figure}
%
%
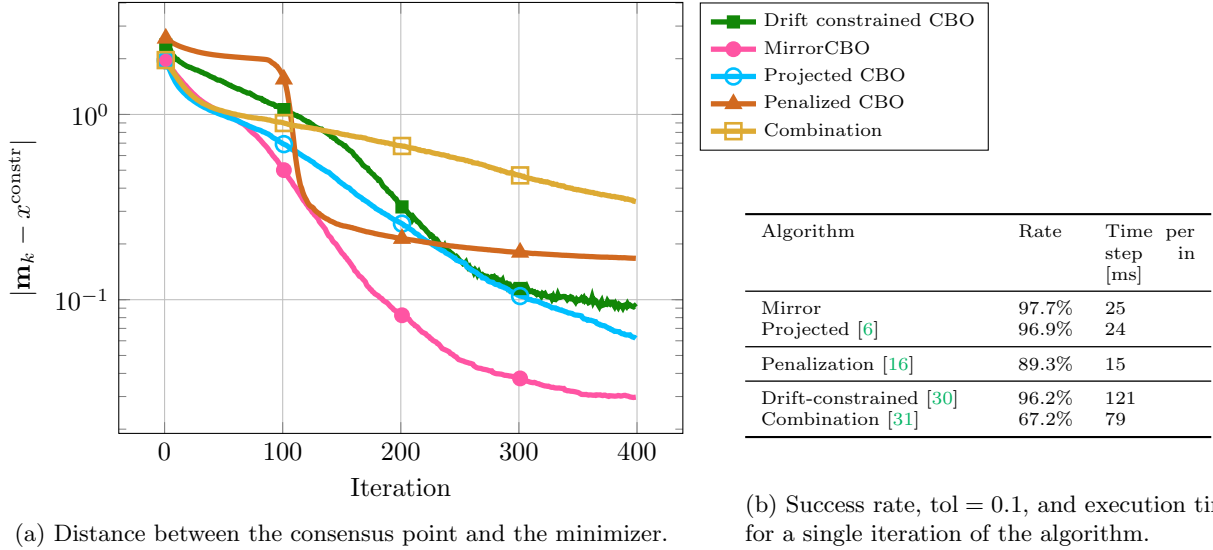
\begin{figure}[htb]%
\centering%
\begin{subfigure}{0.55\linewidth}
\begin{tikzpicture}
\begin{axis}[
ymode=log,
width=\textwidth,height=.32\textheight,
grid=major,legend pos=outer north east,
legend style={font=\scriptsize},
legend cell align={left},
xlabel=Iteration,
ylabel=$\abs{\en{m}_k - x^\text{constr}}$]
\addplot[line width=2pt, Drift] table[x expr=\coordindex+1, y index=0] {figs/quadrics/driftconstr_diff_c.txt};
\addlegendentry{Drift constrained CBO}
\addplot[line width=2pt, MirrorCBO] table[x expr=\coordindex+1, y index=0] {figs/quadrics/mirror_diff_c.txt};
\addlegendentry{MirrorCBO}
\addplot[line width=2pt, ProxCBO] table[x expr=\coordindex+1, y index=0] {figs/quadrics/prox_diff_c.txt};
\addlegendentry{Projected CBO}
\addplot[line width=2pt, PenalizedCBO] table[x expr=\coordindex+1, y index=0] {figs/quadrics/reg_diff_c.txt};
\addlegendentry{Penalized CBO}
\addplot[line width=2pt, Combi] table[x expr=\coordindex+1, y index=0] {figs/quadrics/combi_diff_c.txt};
\addlegendentry{Combination}
\end{axis}
\end{tikzpicture}
\caption{Distance between the consensus point and the minimizer.}
\end{subfigure}\hfill%
\begin{subfigure}{.4\textwidth}
\scriptsize%
\begin{tabular}{p{3cm} l p{1.2cm}}
\toprule
Algorithm     & Rate    & Time per step in [ms]\\
\midrule
Mirror    & 97.7\%  & 25\\ 
Projected \cite{CBO_finance_projection} & 96.9\% & 24\\
\midrule
Penalization \cite{CBO_reg_giacomo} & 89.3\% & 15 \\
\midrule
Drift-constrained \cite{carrillo2024interacting} & 96.2\% & 121 \\
Combination \cite{CBO_reg_jose} & 67.2\% & 79\\
\bottomrule
\end{tabular}
\vspace{.5cm}

\caption{Success rate, $\text{tol}=0.1$, and execution time for a single iteration of the algorithm.}\label{fig:parabolictab}
\end{subfigure}%
\caption{Comparison for parabolic constraints with $d=20$. For the success rate and the evolution of the error we averaged over $1000$ runs. For the run time computation, we process $100$ runs with $N=100$ particles per step, in parallel. }\label{fig:parabolic}
\end{figure}%
\paragraph{Constraints on general quadrics}

The previous examples are special cases of so-called quadric surfaces, which are defined as 
\begin{align*}
\mathcal{Q} = \{x\in\R^d: \langle x, Qx\rangle + \langle n_\quadr, x\rangle + c_\quadr = 0\}
\end{align*}
for a symmetric matrix $Q\in\R^{d\times d}$, a vector $n_\quadr\in\R^d$ and a scalar $c_\quadr$. While the projection for general quadrics is not available in closed form, there still exist fast algorithms to compute an approximate solution \cite{van2022quadproj}\footnote{We want to thank Loïc Van Hoorebeeck for the useful suggestions, concerning this topic.}. In our case, we employ a quasi projection as considered in \cite[Section 2.7]{van2022quadproj}, Algorithm 3 in the mentioned work. We use a combination of the retraction technique in \cite{borckmans2014riemannian} and a gradient-based projection described in \cite{van2022quadproj}. Note, that gradient-based here only refers to the gradient of the constraint, while the gradient of the objective function is still not used. We first consider the case of a paraboloid, where $Q$ is the diagonal matrix with $Q_i = 1$ for $i=1,\ldots, d-1$, $Q_d=0$ and $n_\quadr = (0,\ldots, 0, -1), c_\quadr=0$. As a test function, we again consider the \labelcref{eq:Ack} function, which then yields a similar setup as in \cite[Section 5.2.2, Case 2]{carrillo2024interacting}. We now review the setup for the algorithms we compare in this section. For all algorithms, we use the resampling strategy as mentioned in \cref{alg:indepnoise}.
\begin{itemize}
\item \textbf{Projected CBO} \cite{CBO_finance_projection}: For the projection, we employ the same quasi projection as in the MirrorCBO case. It is not directly clear how the resampling strategy in \cref{alg:indepnoise} should be performed in this scenario. Numerically, we observed that adding the independent noise \emph{after} the projection leads to far better results.
\item \textbf{Penalized CBO} \cite{CBO_reg_giacomo}: We consider the function
\begin{align}\label{eq:para}
g(x) = \langle x, Q x\rangle + \langle n_\quadr, x\rangle + c_\quadr
\end{align}
and consider the penalty $G_1 = \abs{g}$. Compared to the other the experiments, the choice of $p=1$ lead to better results here. We again employ the parameter update rule from \cite[Algorithm 1]{CBO_reg_giacomo} for the regularization parameter.
\item \textbf{Drift-constrained CBO} \cite{carrillo2024interacting}: We choose the same constraint function as defined in \cref{eq:para} and compute
\begin{align*}
\nabla g(x) = 2 Qx + n_\quadr\qquad
\nabla^2 g(x) = 2 Q.
\end{align*}
Here, we need to solve a linear system in \labelcref{eq:driftconstr}. Unfortunately, we are not aware of a closed form solution in this special scenario. Therefore, we fall back to the linear solver implemented in \texttt{NumPy}, which is parallelized over the particles.
\item \textbf{Combination} \cite{CBO_reg_jose}: We employ the same constraint as in \cref{eq:para}. Again the equation in \labelcref{eq:combi} can be solved in closed form yielding the update
\begin{align*}
\en{x}_{k+1} = 
\left(\I + 4 \tau \lambda_2 g(\en{x}_k)\, Q\right)^{-1}(\en{v}_k - 2 \tau \lambda_2 g(\en{x}_k)\, n_\quadr).
\end{align*}
\end{itemize}
In \cref{fig:parabolic} we compare the performance of the different algorithms. Here, we also briefly consider the question of the computational cost of each method. To do so, we consider how long a single iteration of each algorithm takes. The averaged times over $100$ iterations are displayed in \cref{fig:parabolictab}. We considered $N=100$ particles and $M=100$ runs, such that per iteration $10,000$ particles are being processed. In all cases we employ the array-level parallelization as provided by \texttt{NumPy}. The run times were performed on an Intel\textsuperscript{\textregistered}Core\textsuperscript{\texttrademark} i5 laptop with 10 cores and 16 GB RAM.
\begin{figure}[t]%
\centering
\includegraphics[width=.8\linewidth, trim={1cm 0.5cm 0.5cm 1cm}, clip]{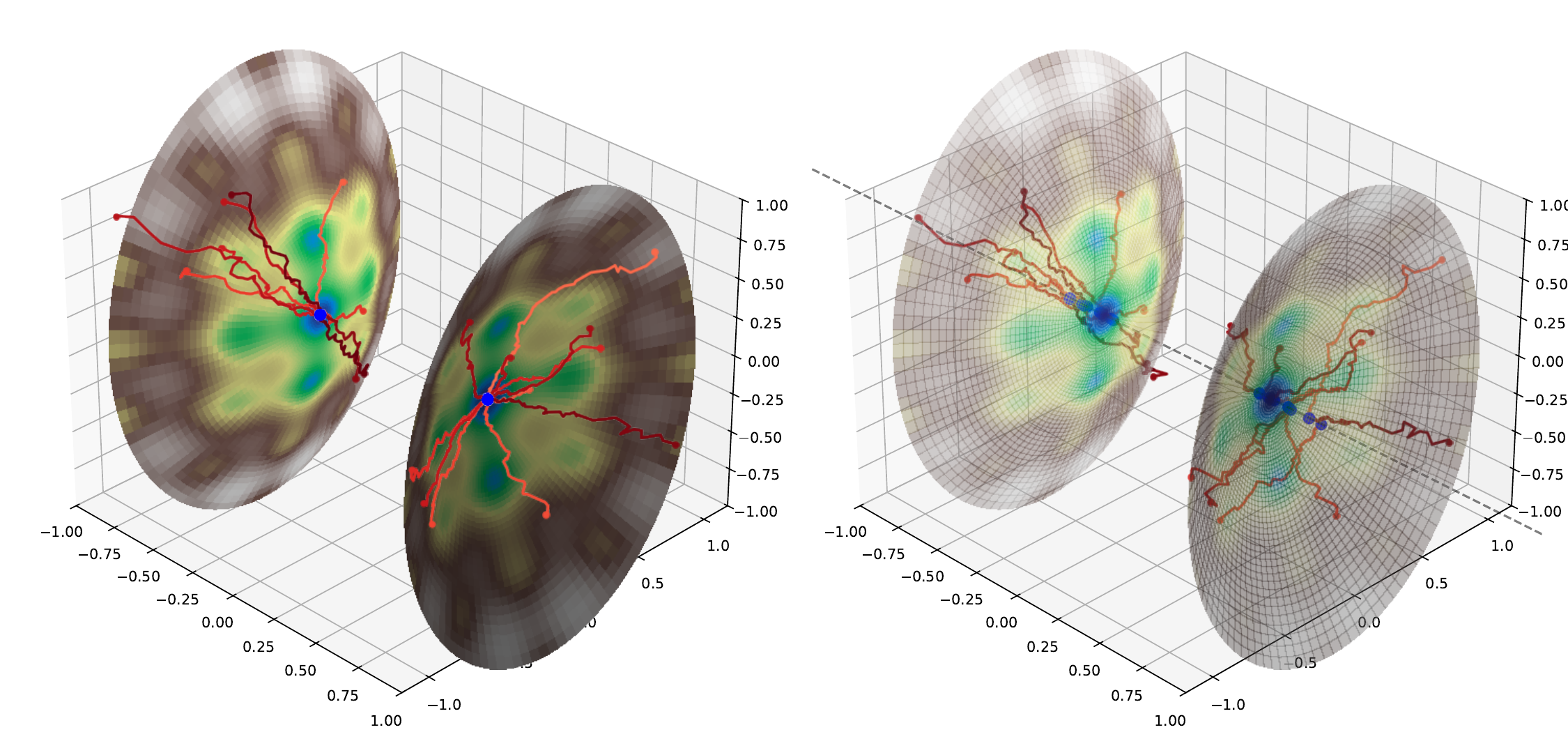}
\caption{The behavior of polarized MirrorCBO on a two-sheet hyperboloid. We visualize the dynamics of the primal particles on the left, and the one of the dual particles on the right.}
\label{fig:two_sheet}
\end{figure}%
A further interesting aspect of constraints on such hypersurfaces, is the case where the constrained minimizer is not unique. For example for $d=3$ considering the Ackley function with $x^\text{shift}=0$ on a two-sheet hyperboloid induced by 
\begin{align*}
Q = \text{diag}(4, -2, -1),
\quad n_\hyp=0, \quad c_\hyp=-1,
\end{align*}
there are two constrained minimizers 
\begin{align*}
x^{\text{constr}, 1} = (0.5, 0, 0),\qquad
x^{\text{constr}, 2} = (-0.5, 0, 0). 
\end{align*}
The MirrorCBO formulation allows transforming the algorithm into its polarized variant as proposed in \cite{bungert2024polarized}. Namely, we exchange the standard consensus point, by the polarized version
\begin{align*}
\en{m}(x^{(j)})
:= 
\frac{\sum_{i=1}^N \kappa(x^{(j)},x^{(i)}) x^{(i)} \exp(-\alpha \obj(x^{(i)}))}{\sum_{i=1}^N \kappa(x^{(j)},x^{(i)}) \exp(-\alpha \obj(x^{(i)}))}
\end{align*}
with a kernel $\kappa:\R^d\times\R^d\to[0,\infty)$, which we choose to be Gaussian. In \cref{fig:two_sheet} we visualize the trajectories of the polarized MirrorCBO algorithm. One observes that both constrained minima are found simultaneously in a single run of the dynamics.

\paragraph{Constraints on non-smooth surfaces} 

An advantage of incorporating the hypersurface constraints via projections is that the surface is not required to be smooth. To illustrate this, we consider the $L^\infty$ norm sphere in two dimensions, i.e., a square centered at the origin, 
\begin{align*}
\hyp = \{x\in\R^2: \abs{x}_\infty = 1\}.
\end{align*}
In this case a projection operation can be defined as $\operatorname{Proj}_\hyp := K \circ \operatorname{CLIP}$, where
\begin{align*}
\operatorname{CLIP}(y)_i := \operatorname{sign}(y)_i \min\{\abs{y}, 1\},\qquad 
K(y)_i := 
\begin{cases}
\operatorname{sign}((y)_i) &\text{if}\quad \abs{(y)_i} = \max_{j=1,\ldots,d} \abs{(y)_j},\\
(y)_i &\text{else}.
\end{cases}
\end{align*}
Here, we only conduct a qualitative example in $d=2$. We consider the so-called Hölder table function, defined as
\begin{align*}
J(x - x^\text{shift}):= - \frac{1}{\pi}\abs{\sin{(x)_1} \, \cos{(x)_2}}\, 
\exp(1 - \abs{x}_2)\,.
\end{align*}
Here, $x^\text{shift}=(0.2,0)$ denotes a vector to shift the global minimum of the function, which is at $(\frac{\pi}{2},0)-x^\text{shift}\approx (1.37,0)$. In \cref{fig:square}, we observe that the mirrored dynamics achieve to find the constrained minimizers $x^\text{constr}=(1, 0)$. While the particles in the dual space exhibit a rather continuous evolution, the particles in the primal space often perform large jumps to different sides of the rectangle.

\begin{figure}%
\centering
\includegraphics[width=.7\linewidth, trim={0cm 0cm 0.cm 0cm}, clip]{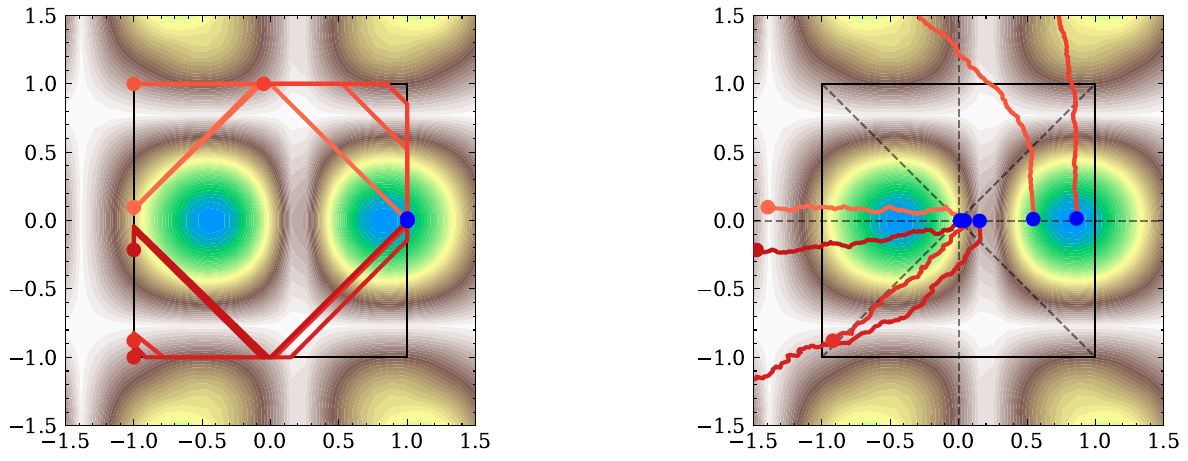}
\caption{The behavior of MirrorCBO on the $\ell^\infty$ norm sphere, i.e., a square. We visualize the dynamics of the primal particles on the left, and the one of the dual particles on the right.}
\label{fig:square}
\end{figure}%
%
%
%
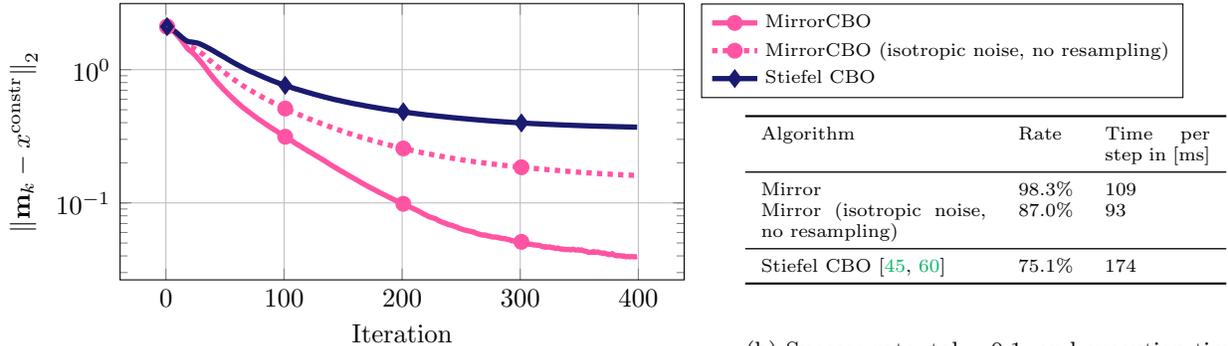
\begin{figure}[htb]%
\centering%
\begin{subfigure}{0.55\linewidth}
\begin{tikzpicture}
\begin{axis}[
ymode=log,
width=\textwidth,height=.23\textheight,
grid=major,legend pos=outer north east,
legend style={font=\scriptsize},
legend cell align={left},
xlabel=Iteration,
ylabel=$\norm{\en{m}_k - x^\text{constr}}_2$]
\addplot[line width=2pt, MirrorCBO] table[x expr=\coordindex+1, y index=0] {figs/Stiefel/mirror_diff_c.txt};
\addlegendentry{MirrorCBO}
\addplot[line width=2pt, MirrorCBO, dotted] table[x expr=\coordindex+1, y index=0] {figs/Stiefel/mirror_iso_diff_c.txt};
\addlegendentry{MirrorCBO (isotropic noise, no resampling)}
\addplot[line width=2pt, Hyper] table[x expr=\coordindex+1, y index=0] {figs/Stiefel/sphere_diff_c.txt};
\addlegendentry{Stiefel CBO}
\end{axis}
\end{tikzpicture}
\caption{Distance between the consensus point and the minimizer.}
\end{subfigure}\hfill%
\begin{subfigure}{.4\textwidth}
\scriptsize%
\begin{tabular}{p{3cm} l p{1.4cm}}
\toprule
Algorithm     & Rate    & Time per step in [ms]\\
\midrule
Mirror    & 98.3\%  & 109 \\
Mirror (isotropic noise, no resampling) & 87.0\%  & 93\\ 
\midrule
Stiefel CBO \cite{ha2022stochastic,CBO_Stiefel} & 75.1\% & 174\\
\bottomrule
\end{tabular}
\vspace{.5cm}

\caption{Success rate, $\text{tol}=0.1$, and execution time for a single iteration of the algorithm.}
\end{subfigure}%
\caption{Comparison for constraints on the Stiefel manifold with $k=10,d=5$, averaged over $1000$ runs.}\label{fig:Stiefel}
\end{figure}%
\paragraph{Constraints on the Stiefel manifold}

A further interesting class of constrained problems arise from considering the Stiefel manifold \cite{stiefel1935richtungsfelder}
\begin{align*}
\operatorname{St}_{k,d} := \{X\in \R^{k\times d}: X^T X = \I_d\}.
\end{align*}
In the context of consensus-based methods, this was explored in \cite{CBO_Stiefel,ha2022stochastic}. The approach therein derives an update similar to the one in \cite{Fornasier2020ConsensusbasedOO}, i.e., in \labelcref{eq:hyper}. However, this is not achieved by defining a signed distance function $\gamma$, but rather by directly computing the appearing terms. Their algorithm is obtained by making the following substitutions in \labelcref{eq:hyper}:
\begin{align*}
P(X) Z := Z - \frac{1}{2} (XZ^TX + XX^TZ),\qquad
\nabla \gamma(X) \leadsto X,\qquad
\Delta \gamma(X) \leadsto \frac{2d - k - 1}{2}\,.
\end{align*}
Note, that in the following we consider the case, where $d\geq k > 1$, and therefore the last expression above is strictly bigger than zero. For the mirror CBO case we use the identity 
\begin{align*}
\min_{\tilde{X}\in \operatorname{St}_{k,d}} \norm{X - \tilde{X}}_F = \norm{X - U V^T}_F
\end{align*}
where $\norm{\cdot}_F$ denotes the Frobenius norm and $U\Sigma V^T = X$ is the (thin) singular value decomposition of $X$. We refer to \cite{schonemann1966generalized} for general Procrustes problems arising in this context, and to \cite[Section 12.4.1]{golub2013matrix} or \cite[Theorem 4.1]{higham1988matrix} for a proof of the above statement. This yields that we can define a projection onto the Stiefel manifold as
\begin{align*}
\proj_{\operatorname{St}_{k,d}}(X) := UV^T.
\end{align*}
As a test function, we again choose the \labelcref{eq:Ack} function with the same parameters as before. 
By a slight abuse of notation, we apply a function defined on $\R^{\tilde{d}}$ with $\tilde d= dk$ to elements of the Stiefel manifold, by vectorizing the input, i.e.,
\begin{align*}
\obj(X) = \obj
\begin{pmatrix}
(X)_1\\
\vdots\\
(X)_d
\end{pmatrix},
\end{align*}
where $(X)_i\in\R^{k}$ denotes the $i$th column of a matrix $X$. For each run, we sample a random matrix $M$ on the Stiefel manifold, which we use as the shift vector $x^\text{shift}=M$, which then also is the constrained minimizer of this function. For the initialization, we sample matrices uniformly at random on the Stiefel manifold. For $k\geq d$, this can be achieved by first sampling $Z\sim\mathcal{N}(0,\I_{k\times d})$ and then computing
\begin{align*}
Z((Z^T Z)^{-1})^{\frac{1}{2}},
\end{align*}
see \cite[Theorem 2.2.1]{chikuse2012statistics}.
In \cref{fig:Stiefel} we compare the results of this setup for $d=10, k=5$. There are two differences between MirrorCBO and Hypersurface CBO for the setups of the results in \cref{fig:Stiefel}:
\begin{itemize}
\item The mirror CBO uses anisotropic noise, which is better suited for the high-dimensional problem we consider here. The extension of the hypersurface CBO iteration to anisotropic noise (as in \cite{CBO_sphere2} for the sphere) is not directly clear for the Stiefel manifold. Therefore, we use the algorithm of \cite{CBO_Stiefel} with isotropic noise.
\item The mirror CBO algorithm employs the resampling scheme from \cref{alg:indepnoise}. Again, it is not directly clear where a similar modification should be made for the algorithm in \cite{CBO_Stiefel}, therefore it is not used.
\item For better comparison, we also include the performance of MirrorCBO without the above modifications. However, we note that one of the strengths of MirrorCBO is its simplicity, which allows the easy incorporation of various enhancements.
\end{itemize}
%
%
%

%
%
%
\subsection{Limitations for preconditioning}\label{sec:precon}%
\begin{figure}[H]
\begin{subfigure}[t]{.48\textwidth}
\begin{tikzpicture}
\begin{axis}[
ymode=log,
width=\textwidth,height=.3\textheight,
grid=major,legend pos=outer north east,
legend style={
    at={(0.5,1.)},
    anchor=south,legend columns=1,
    font=\scriptsize
    },
legend cell align={left},
xlabel=$\tau$,
ylabel=Energy, ymode=log,ymax=10]
\addplot[line width=2pt, color=mirror, mark=*] table[x index = 0, y expr=\thisrowno{4}+1e-48,restrict x to domain=0:2] {figs/precon/precon.csv};
\addlegendentry{Precondtioned GD}
\addplot[line width=2pt, color=blue, mark=square*] table[x index = 5, y index=3, restrict x to domain=0:0.5] {figs/precon/precon.csv};
\addlegendentry{Gradient Descent}
\end{axis}
\end{tikzpicture}%
\captionsetup{width=.9\linewidth}%
\caption{For $50$ different initializations, we let gradient descent and its preconditioned variant run for $10$ steps. We display the mean of the objective value $f(x)$ at the last iterate. We show results for the admissible step sizes in $[0,0.5]$ and $[0,2]$, respectively.}
\end{subfigure}%
\begin{subfigure}[t]{.48\textwidth}
\begin{tikzpicture}
\begin{axis}[
ymode=log,
width=\textwidth,height=.3\textheight,
grid=major,legend pos=outer north east,
legend style={
    at={(0.5,1.)},
    anchor=south,legend columns=1,
    font=\scriptsize
    },
legend cell align={left},
xlabel=$\tau$,
ylabel=Energy,
ymode=log]
\addplot[line width=2pt, MirrorCBO, mark repeat={1},] table[x index = 0, y index=2] {figs/precon/precon.csv};
\addlegendentry{MirrorCBO}
\addplot[line width=2pt, color=blue, mark=square*] table[x index = 0, y index=1] {figs/precon/precon.csv};
\addlegendentry{CBO}
\end{axis}
\end{tikzpicture}
\captionsetup{width=.9\linewidth}%
\caption{For 50 different initializations, we let CBO and the preconditioning variant run for $100$ iterations. We display the mean objective value $f(\en{c})$ of the consensus point at the last iterate.}
\end{subfigure}
\caption{We show the effect of preconditioning for the simple setup in \cref{sec:precon}.
}\label{fig:precon}
\end{figure}
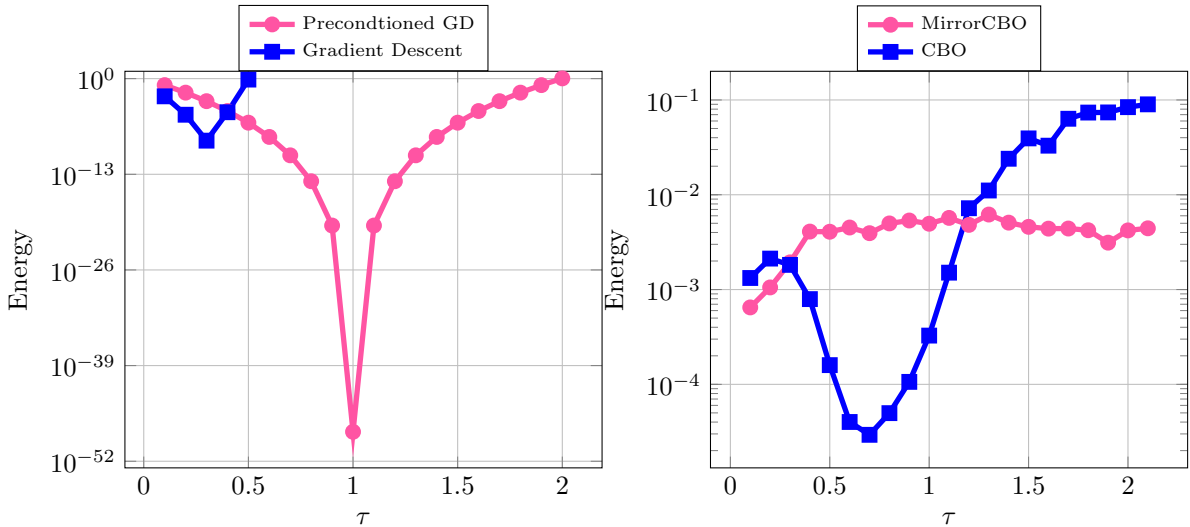%
Mirror descent has a strong connection to so-called preconditioned gradient descent. Assume, e.g., that the distance generating map is quadratic, i.e., $\phi(x) = \frac12\langle H x,x\rangle$ for some invertible and symmetric matrix~$H$. 
Then mirror descent can be expressed in terms of the primal variable as follows:
\begin{gather*}
x_{k+1} = x_k - \tau H^{-1}\nabla \obj(x_k).
\end{gather*}
If the preconditioning matrix $H$ is chosen appropriately, this strategy can greatly improve the convergence speed. Furthermore, one can choose the preconditioning matrix adaptively, e.g., as the Hessian of the objective function, where one recovers Newton's method. 
Interestingly, in the context of CBO-type methods, such preconditioning is neither necessary nor helpful. 
For this we first note that in the mean field regime and for sufficiently large $\alpha$, the associated Fokker--Planck equation  shows that the law of the particles approximately solves the Wasserstein gradient flow of the function $x\mapsto\frac12\abs{x-\hat x}^2$ where $\hat x$ is the global minimizer of $J$.
Consequently, most Hessian information about $J$ is lost in the mean field regime which is also reflected by the fact that the convergence theorem \cref{thm: exp_decay_V} just depends on the first order quantities in \cref{ass: J,ass: J2}.
Consequently, a preconditioning through the choice of a suitable mirror map is not necessary and can even have adverse effect. 
To see this, we consider the function $\obj(x) = \frac{1}{2}\langle x,Ax\rangle$ with $A=\text{diag}(4,2)$. 
The optimal preconditioning matrix is given as $H=A$, which in fact recovers Newton's method. In \cref{fig:precon}, we first show the effect of the preconditioning for standard gradient descent. In fact, we observe that for a step size $\tau = 1$ the method converges in one step, because then $x_{1}=x^0 - A^{-1}A x^0 = 0$. 
As expected, CBO performs well on this task whereas mirror CBO with the distance generating function $\phi(x)=\frac{1}{2}\langle x,Ax\rangle$ does not yield any speed-up but appears to be more stable with respect to the time step size.
%
%
%
\subsection{Beyond \cref{ass:bregman_phi}}\label{sec:numex}

There are some examples of interest for the choice of mirror map $\phi$ for which \cref{ass:bregman_phi} is not satisfied. For instance, for optimization with equality constraints on sets $C$, we would like to consider the distance generating function \labelcref{eq:constr} used in \cref{sec:constraints_numerics},
\begin{align*}
\mm(x)= \frac{1}{2} \abs{x}^2 + \chara_{C}(x).
\end{align*}
However, unless $C$ is a linear subspace (see \cref{ex: bregman_phi}), this choice of mirror map does not necessarily satisfy the upper bound in \cref{ass:bregman_phi} for all $y\in\R^d$. 
A natural question arises: Can we still expect (exponential) convergence of the MirrorCBO dynamics \labelcref{eq: meanfield-mirrorcbo-pde}? 

To build intuition, let us consider the unit ball $C = B_1(0)=:B \in \R^d$. For this choice of $\phi$ it is easy to verify that
\begin{align*}
    \phi^*(y) = 
    \begin{dcases}
    \frac{1}{2}\abs{y}^2  &\text{ if } y \in B\,, \\ 
    \abs{y} - \frac{1}{2} &\text{ else,}
\end{dcases}
\qquad
\text{and}
\qquad
\nabla \phi^*(y) = \begin{dcases}
    y &\text{ if } y \in B, \\ 
    \frac{y}{\abs{y}}  &\text{ else},
\end{dcases}
= \proj_B(y)\,.
\end{align*}
First let us consider the mirror flow
\begin{align*}
    \dot y(t)=-\nabla J (\nabla \phi^*(y))=-\nabla J(\proj_B(y))\,.
\end{align*}
For the objective function $\obj(x):=\frac12\abs{x}^2$, and with initial condition $y(0)=y_0\notin B$, we have
\begin{align*}
    \abs{y(t)}=
    \begin{cases}
        -t+\abs{y_0} &\text{ if } t\in [0,T_0]\,,\\
        e^{-t} &\text{ if } t>T_0\,,
    \end{cases}
\end{align*}
where $T_0=\abs{y_0}-1$.
In other words, we first observe algebraic decay to the unit ball, and then exponential decay to the global minimizer $\xhat=0\in B$ as soon as the evolution has arrived inside the constraint set.

For MirrorCBO,  \cref{prop: dynamics_V} and the resulting exponential decay from \cref{thm: exp_decay_V} do not necessarily hold without the upper bound in \cref{ass:bregman_phi}. However, the above observation on the mirror descent dynamics suggests that it may be possible still to achieve convergence, however potentially at a rate slower than exponential for an initial phase. We observe numerically that this is indeed the case, as demonstrated in \cref{fig:anaex}. 
\begin{figure}[t]
\centering
\includegraphics[width=0.7\linewidth]{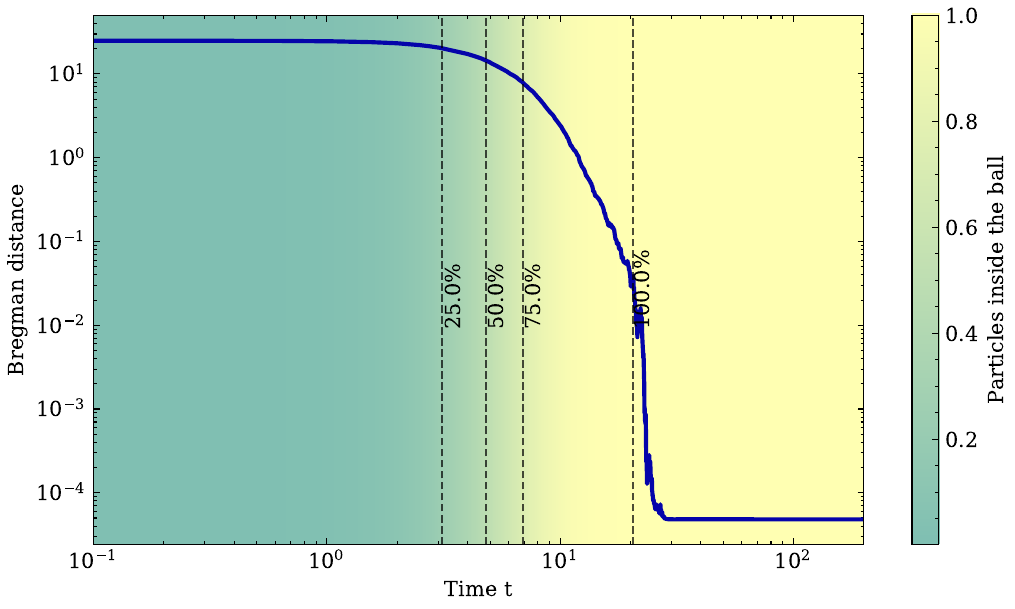}
\caption{We visualize the evolution of the Bregman distance in \cref{eq:discbreg}. The shaded color in the background indicates, what percentage of the dual ensemble lies within the ball $B$ at the given time.}
\label{fig:anaex}
\end{figure}
Here, we initialize the particles in the set 
\begin{align*}
\{x\in\R^2: 3\leq \abs{x} \leq 6, (x)_1\geq 0,\}
\end{align*}
and we plot the evolution of the expression
\begin{align}\label{eq:discbreg}
\sum_{n=1}^N D_\phi^{\en{y}^{(n)}_k}\left(0, \nabla\mm^*(\en{y}^{(n)}_k)\right)\,.
\end{align}
MirrorCBO evolves a swarm of particles communicating with each other via the weighted mean (whereas the mirror descent dynamics evolve each particle independently), and so it is an interesting question which proportion of the mass of the distribution may be required to have reached the constraint set for exponential convergence to kick in. In \cref{fig:anaex}, we visualize what percentage of the particles already fulfill the constraint at each time. We observe a rapid contraction of mass to the constraint set at on-set of exponential decay.

Next, we investigate this phenomenon theoretically. The Bregman distance corresponding to the distance generating function \labelcref{eq:constr} is given by
 \begin{align*}
    D_{\phi}^y(\xhat, \nabla \phi^*(y)) = \phi(\xhat) - \phi(\nabla \phi^*(y)) - \langle y, \xhat - \nabla \phi^*(y) \rangle 
    = \begin{dcases}
        \frac{1}{2}\abs{y}^2  &\text{ if } y \in B,\\
        \abs{y} - \frac{1}{2}  &\text{ else},
    \end{dcases}
    =:
    \xi(y),
\end{align*}
and we note that $\xi \in \mathcal{C}^{1,1}_
*(\R^d)$ and $\xi=\phi^*$.
This allows us to compute the Lyapunov function $V$ as follows
\begin{align*}
    V(t) &= \int_{\R^d} D_{\phi}^y(\xhat, \nabla \phi^*(y))\de\mu_t(y) 
    = \int_{\R^d} \xi(y)\de\mu_t(y) \\
    & = 
    \frac{1}{2}
    \int_{B} \abs{y}^2\de\mu_t(y) + 
    \int_{\R^d \setminus B} \left( \abs{y} - \frac{1}{2} \right)\de\mu_t(y).
\end{align*}
and its derivative is given by
\begin{align*}
    \dot V(t) 
    &= 
    -
    \int
    \left\langle
    \nabla\xi(y),
    \nabla\phi^*(y)-\mamu
    \right\rangle
    \de\mu_t(y)
    +
    \sigma^2
    \int
    \Delta\xi(y)\abs{\nabla\phi^*(y)-\mamu}^2
    \de\mu_t(y)
    \\
    & = \int_{B} -
    \langle y, y - m_{\alpha}(t)\rangle + d\sigma^2 |y - \mamu|^2\de\mu_t(y) - \int_{\R^d \setminus B} \left\langle \frac{y}{\abs{y}},\frac{y}{\abs{y}} - \mamu \right\rangle\de\mu_t(y) \\
    &\qquad 
     + \sigma^2(d-1) \int_{\R^d \setminus B} \frac{1}{|y|} \left|\frac{y}{\abs{y}}-\mamu\right|^2\de\mu_t(y)\\
    & = - \int_{B} \abs{y}^2\de\mu_t(y) + 
    \left\langle \mamu,\int_{B} y\de\mu_t(y)\right\rangle + d\sigma^2 \int_{B} \abs{y - \mamu}^2\de\mu_t(y)
    \\
    &\qquad
    - \mu_t(\R^d \setminus B) + \left\langle \mamu, \int_{\R^d \setminus B} \frac{y}{\abs{y}}\de\mu_t(y)\right\rangle
    + \sigma^2(d-1) \int_{\R^d \setminus B} \frac{1}{|y|} \left|\frac{y}{\abs{y}}-\mamu\right|^2\de\mu_t(y)\,.
\end{align*}
Let us now make the relatively strong assumption that for all $t\geq 0$ we have $\mamu = \hat x = 0$.
This would be the case, e.g., if $\alpha=\infty$ and $\hat x\in\supp\mu_t$ for all $t\geq 0$, or if both $J$ and $\mu_0$ are radially symmetric.
Under this assumption the derivative can be bounded by
\begin{align*}
    \dot{V}(t) \le  -(1-d\sigma^2) \int_{B} \abs{y}^2\de\mu_t(y) - \left(1-\sigma^2(d-1)\right) \mu_t(\R^d \setminus B)
\end{align*}
and from this we see that $\dot V(t)\leq 0$ for all $t\geq 0$ if $\sigma^2\leq\frac{1}{d-1}$ for $d>1$ or $\sigma^2<1$ for $d=1$. If in addition there is no noise ($\sigma=0$), similar to the mirror descent example above, we obtain exponential decay if $\supp\mu_t\subset B$: in this case we have
\begin{align*}
    \dot V(t) = -2 V(t)\,,
\end{align*}
and hence the usual exponential decay $V(t)=V(0)\exp(-2t)$ immediately follows. This is not surprising since under the assumption that $\supp\mu_t\subset B$, MirrorCBO reduces to regular CBO.

It turns out that even if some mass lies outside of $B$ one can prove exponential decay.
The inequality $\dot V(t)\leq -C V(t)$ for all $t\geq 0$ is equivalent to
\begin{align*}
    -(1-d\sigma^2)\int_B\abs{y}^2\de\mu_t(y)-(1-\sigma^2(d-1))\mu_t(\R^d\setminus B)
    \leq 
    -\frac{C}{2}\int_B\abs{y}^2\de\mu_t(y)-C\int_{\R^d\setminus B}\left(\abs{y}-\frac{1}{2}\right)\de\mu_t(y).
\end{align*}
This can be reordered to
\begin{align*}
    \left(\frac{C}{2}-(1-d\sigma^2)\right)
    \int_B\abs{y}^2\de\mu_t(y)
    +
    \int_{\R^d\setminus B}
    \left(C\abs{y}-\frac{C}{2}-1+\sigma^2(d-1)\right)
    \de\mu_t(y)
    \leq 0.
\end{align*}
We want to choose $C$ such that both terms on the left are non-positive.
For this let us assume that $\supp\mu_t\subset B_{R_0}(0)$ for all $t\geq 0$ and some $R_0>1$, and $\sigma=0$.
Then for the choice $C = \frac{2}{R_0}$ we get the desired inequality and, using Gronwall's inequality, we get $V(t)\leq V(0)\exp(-Ct)$ for all $t\geq 0$.
From this example it seems like one can hope to achieve exponential decay estimates as in \cref{prop: dynamics_V} even without the upper bound from \cref{ass:bregman_phi} if one assumes (or proves) some higher order moment bounds for the dynamics.
This will be a subject of future work.

\section*{Acknowledgment}
\addcontentsline{toc}{section}{Acknowledgment}

LB and DK would like to thank Konstantin Riedl for helpful conversations and references in the context of the well-posedness proofs and, furthermore, for his detailed comments on the first version of this article. 
TR acknowledges support from DESY (Hamburg, Germany), a member of the Helmholtz Association HGF. 
This research was supported in part through the Maxwell computational resources operated at Deutsches Elektronen-Synchrotron DESY, Hamburg, Germany. 
Parts of this study was carried out, while TR was visiting the group of FH at the California institute of technology, supported by the DAAD grant for project 57698811 \enquote{Bayesian Computations for Large-scale (Nonlinear) Inverse Problems in Imaging} and the host FH. 
TR further wants to thank Samira Kabri for many insightful discussions. 
LB and TR acknowledge funding by the German Ministry of Science and Technology (BMBF) under grant agreement No. 01IS24072A (COMFORT).
LB also acknowledges funding by the Deutsche Forschungsgemeinschaft (DFG, German Research Foundation) – project number 544579844 (GeoMAR). 
FH and DK are supported by start-up funds at the California Institute of Technology and by NSF CAREER Award 2340762.

\printbibliography[heading=bibintoc]

\clearpage
\begin{appendix}
\section{Further details on the numerical examples}\label{sec:numapp}

\subsection{Utilities for CBO methods}

In this section, we specify the concrete utility algorithms used for the numerical examples.

%
%
\begin{algorithm}[H]
\caption{Computes the consensus point with a LogSumExp trick.}\label{alg:consensus}
\begin{algorithmic}[1]
\Function{ComputeConsensus}{$\en{x}, \alpha>0$}
    \State $\en{c} = \exp(- \alpha \obj(\en{x}) -  \textbf{LogSumExp}(-\alpha \obj(\en{x})))$
    \State \Return $\sum_{n=1}^N c^{(n)}\, x^{(n)}$
\EndFunction
\end{algorithmic}
\end{algorithm}

%
%
\begin{algorithm}[H]%
\caption{The isotropic noise method, as proposed in \cite{pinnau2017consensus}.}\label{alg:isonoise}
\begin{algorithmic}[1]
\Function{IsotropicNoise}{$r\in\R^d, \tau > 0$}
    \State $z \sim \mathcal{N}(0, \I_{d\times d})$
    \State \Return $\sqrt{\tau}\, \abs{r}\, z$
\EndFunction
\end{algorithmic}
\end{algorithm}%
\begin{algorithm}[H]%
\caption{The anisotropic noise method, as proposed in \cite{carrillo2021consensus}.}\label{alg:anisonoise}
\begin{algorithmic}[1]%
\Function{AnisotropicNoise}{$r\in\R^d, \tau > 0$}
    \State $z \sim \mathcal{N}(0, \I_{d\times d})$
    \State \Return $\sqrt{\tau}\ r \odot z$
\EndFunction
\end{algorithmic}
\end{algorithm}%
%

%
%
\begin{algorithm}[H]
\caption{Simple update rule for the parameter $\alpha$.}\label{alg:mutlAlpha}
\begin{algorithmic}[1]
\Function{MultiplyAlpha}{$\eta_\alpha, \alpha_{\text{max}}$}
    \State $\alpha \gets \min\{\alpha\, \eta_\alpha, \alpha_{\text{max}}\}$
\EndFunction
\end{algorithmic}
\end{algorithm}%
%
%
\begin{algorithm}[H]
\caption{Effective sample size scheduler as proposed in \cite{carrillo2022consensus}}\label{alg:effAlpha}
\begin{algorithmic}[1]
\Function{EffectiveSampleSizeAlpha}{$\en{x}, \eta_\alpha, \alpha_{\text{max}}$}
    \State $e(\alpha):= 
    \left(\sum_{n=1}^N \exp(-\alpha\, \obj(x^{(n)}))\right)^2 - \eta_\alpha\,N\,\sum_{n=1}^N \exp(-2\alpha\, \obj(x^{(n)})) $
    \State Find $\alpha^*$ such that $e(\alpha^*) = 0$.
    \State $\alpha \gets \min\{\alpha^*, \alpha_{\text{max}}\}$
\EndFunction
\end{algorithmic}
\end{algorithm}%
\begin{remark}
In order to find the root in \cref{alg:effAlpha}, we employ a simple bisection algorithm.
\end{remark}
%
%
%
%
\begin{algorithm}[H]
\caption{Adaptive discrepancy principle.}\label{alg:adaptive-disc-principle}
\begin{algorithmic}[1]%
\State \textbf{Parameters}:  the parameter to update $\regp>0$, noise level $\noiselvl>0$, parameter range $(\regp_{\text{min}}, \regp_{\text{max}})$, factors $(\eta_{\text{decr}}\geq 1, \eta_{\text{incr}}\leq 1)$
\Function{DiscrepancyPrinciple}{}
    \If{$2\obj(\en{m}_k) < \noiselvl^2$}
        \State $\regp \gets \regp\,\eta_{\text{incr}} $
    \Else
        \State $\regp \gets \regp\,\eta_{\text{decr}} $
    \EndIf
    \State $\regp \gets \min\{\max\{\regp, \regp_{\text{min}}\}, \regp_{\text{max}}\}$
\EndFunction
\end{algorithmic}
\end{algorithm}%
%
%
\begin{algorithm}[H]%
\caption{Resampling as proposed in \cite{carrillo2021consensus}.}\label{alg:indepnoise}
\begin{algorithmic}[1]%
\State \textbf{Parameters}:  $\texttt{patience}\in\N, \texttt{tol}, \eta_{\text{indep}},\sigma_{\text{indep}}>0$, the latter is updated during the iteration.
\Function{IndependentNoise}{}
    \If{$\max_{j=0,\ldots,\texttt{patience}}\{\abs{\en{m}_{k-j} - \en{m}_{k-j-1}}\} < \texttt{tol}$}
        \State $\en{z} = (z^{(1)},\ldots, z^{(N)}), \quad z^{(n)}\sim \mathcal{N}(0, \I_{d\times d})$ for $n=1,\ldots, N$
        \State $\en{y}_k \gets \en{y}_k + \sigma_{\text{indep}}\, \sqrt{\tau}\, \en{z}$
        \State $\sigma_{\text{indep}} \gets \sigma_{\text{indep}}\, \eta_{\text{indep}}$
    \EndIf
\EndFunction
\end{algorithmic}
\end{algorithm}%
\begin{algorithm}[H]%
\begin{algorithmic}[1]
\caption{Compute consensus point on parts of the ensemble.}\label{alg:batching}
\State \textbf{Parameters}: batch size $b\leq N$.
\State Initialize indices as $I = \textbf{RandomPermute}([1:N])$
\Function{ComputeConsensusPartial}{$\en{x}, \alpha>0$}
\If{$\textbf{len}(I)\leq b$}
\State $I\gets \textbf{RandomPermute}([1:N])$
\EndIf
\State $\en{c} = \textbf{ComputeConsensus}(\en{x}[I[1:b]],\alpha)$
\State $I \gets I[b:\operatorname{len}(I)]$
\State \textbf{Return} $\en{c}$
\EndFunction
\end{algorithmic}
\end{algorithm}
\begin{remark}
In \cref{alg:batching}, we employ classical array-indexing notation, where for an index tuple $I=(i_1, \ldots, i_I)$ we write $\en{x}[I] = (x^{(i_1)}, \ldots, x^{(i_I)})$ and for an integer $b$ we have $[1:b]=(1, \ldots, b)$. Furthermore, $\textbf{len}(I)$ denotes the number of indices in $I$ and $\textbf{RandomPermute}$ denotes a random permutation of indices. Furthermore, the initialization of the index set $I$ in line $2$ only happens at the start of \cref{alg:MirrorCBO}. After that the function is called in the while loop of \cref{alg:MirrorCBO}, where in every iteration we discard the first $b$ entries. If the list is exhausted, i.e., there are less than $b$ indices left, we resample it.
\end{remark}%

\subsection{Failure case for sparsity regularization}\label{sec:failexact}

\cref{fig:extareg-failure} shows a graphic describing a failure case for MirrorCBO in the exact regularization context of \cref{sec:normsphere}. We observe the formation of two clusters in the primal space, 
\begin{itemize}
    \item (A): close to the point $(0.45, 0.1)$,
    \item (B): close to the point $(0.45, 0)$.
\end{itemize}
The cluster at (A) lies on the desired solution set, however, it is not sparse. The cluster at (B) is sparse, but it has a suboptimal objective value. Since both clusters are very close in $x_1$ direction, the consensus drift mainly provides information in $x_2$ direction. Since (A) has a better objective value, this only yields that the particles close to (B) slowly get pushed upwards, which can be observed in the dual space in \cref{fig:extareg-failure}. 
The noise term is not strong enough, to move point in (A) and therefore, the ensemble tends to collapse close to the non-sparse solution at (A).
\begin{figure}
\centering
\includegraphics[width=\linewidth]{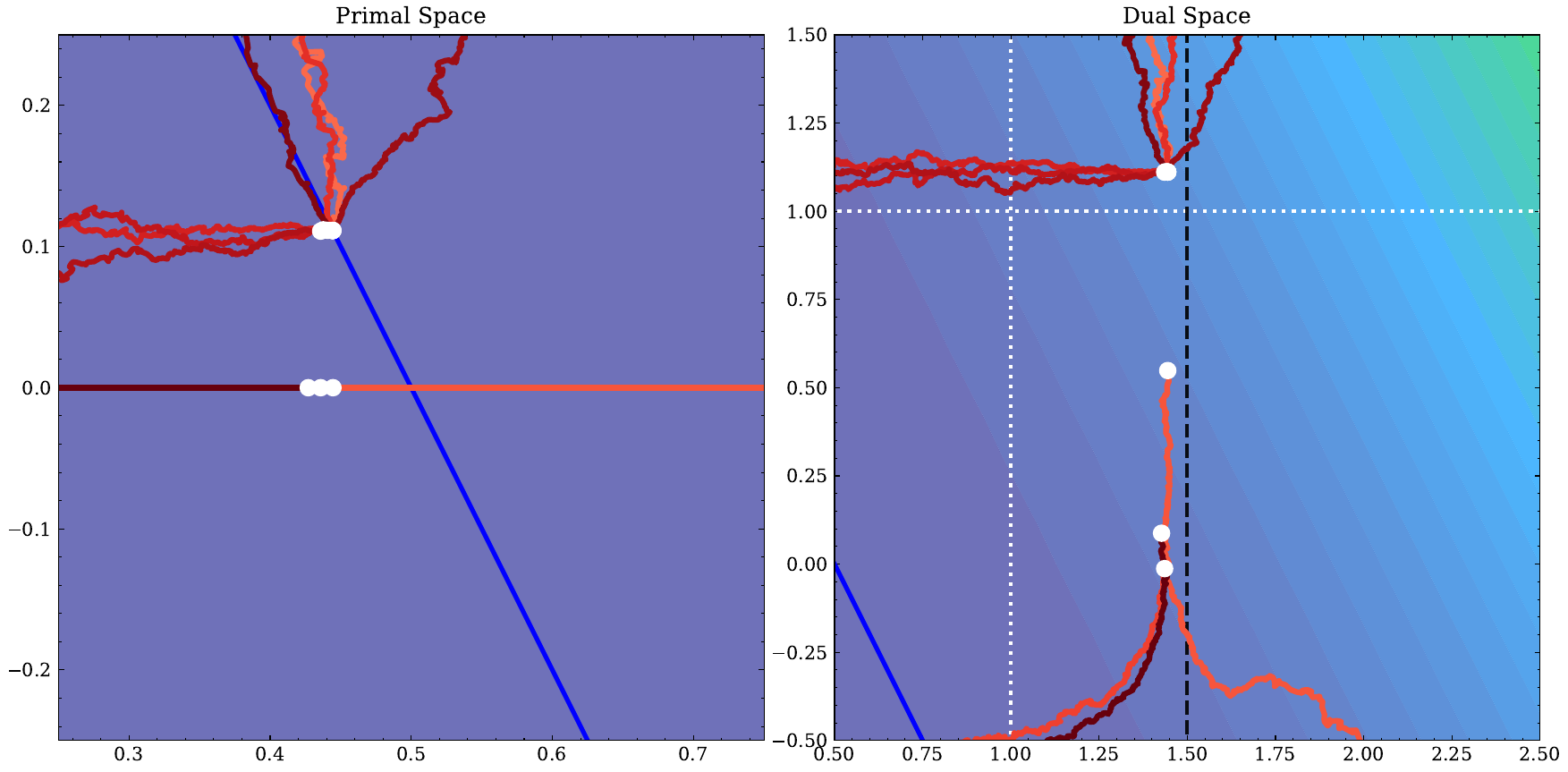}
\caption{Failure case example for the example in \cref{sec:normsphere}. The setup is the same as in \cref{fig:normsphere}, where we zoomed in the specific regions.}
\label{fig:extareg-failure}
\end{figure}

\subsection{Hyperparameters for \cref{sec:normsphere}}

\setlength\tabcolsep{3pt}%
\begin{table}[H]
\centering\tiny
\begin{tabular}
{p{1.5cm}|c|c|c|c|c|c|c|c|p{1cm}|p{2cm}|p{1cm}|p{2cm}}
\toprule%
Algorithm & $\tau$ & $\alpha$ & noise & $\sigma$ & M & N & d & $k_{\text{max}}$ & batching & scheduler & init & Miscellaneous\\
\midrule%
MirrorCBO & 0.01 & 10 & isotropic & 0.1 & 1 & 50 & 2 & 1000 & None & \cref{alg:mutlAlpha} multiply(1.05, 1e12) & $\mathcal{N}(0,I)$ & None\\
\bottomrule
\end{tabular}
\caption{Parameters for \cref{fig:normsphere}.}
\label{tab:normsphere}
\end{table}


\begin{table}[H]
\centering\tiny
\begin{tabular}
{p{2cm}|c|c|c|c|c|c|c|c|p{1cm}|p{2cm}|p{1cm}|p{2cm}}
\toprule%
Algorithm & $\tau$ & $\alpha$ & noise & $\sigma$ & M & N & d & $k_{\text{max}}$ & batching & scheduler & init & Miscellaneous\\
\midrule
MirrorCBO & 0.05 & 1.0 & isotropic & 0.5 & 100 & 150 & 2 & 400 & None & \cref{alg:mutlAlpha} multiply(1.05,1e4) & $\mathcal{N}(0,1)$ & None\\
\midrule
MirrorCBO + $\ell^0$ & 0.05 & 10 & isotropic & 0.5 & 100 & 150 & 2 & 400 & None & \cref{alg:mutlAlpha} multiply(1.05,1e12) & $\mathcal{N}(0,1)$ & $\lambda_{\ell^0}=.0001$\\
\midrule
Dualized CBO & 0.01 & 10 & isotropic & 1.2 & 100 & 150 & 2 & 400 & None & \cref{alg:mutlAlpha} multiply(1.05,1e12) & $\mathcal{N}(0,1)$ & $\lambda=1$\\
\midrule
Penalized CBO & 0.05 & 1e6 & isotropic & 0.5 & 100 & 150 & 2 & 400 & None & \cref{alg:mutlAlpha} multiply(1.05,1e15) & $\mathcal{N}(0,1)$ & $\lambda=0.001$\\
\midrule
Penalized CBO (bigger $\lambda$) & 0.05 & 1e6 & isotropic & 0.5 & 100 & 150 & 2 & 400 & None & \cref{alg:mutlAlpha} multiply(1.05,1e15) & $\mathcal{N}(0,1)$ & $\lambda=1$\\
\midrule
Penalized CBO + $\ell^0$ & 0.05 & 1e6 & isotropic & 0.5 & 100 & 150 & 2 & 400 & None & \cref{alg:mutlAlpha} multiply(1.05,1e15) & $\mathcal{N}(0,1)$ & $\lambda=0.001$\\

\bottomrule
\end{tabular}
\caption{Parameters for \cref{fig:exactreg2}.}
\label{tab:exactreg2}
\end{table}

\subsection{Parameters for \cref{sec:deconv}}


\begin{table}[H]
\centering\tiny
\begin{tabular}
{p{2cm}|c|c|c|c|c|c|c|c|p{1cm}|p{2cm}|p{1cm}|p{2cm}}
\toprule%
Algorithm & $\tau$ & $\alpha$ & noise & $\sigma$ & M & N & d & $k_{\text{max}}$ & batching & scheduler & init & Miscellaneous\\
\midrule%
MirrorCBO + $\ell^0$ & 0.02 & 100 & anisotropic & 3.0 & 1 & 1000 & 100 & 5000 & None & \cref{alg:mutlAlpha} multiply(1.05,1e18) & $\mathcal{U}(0, 1)$ &  $\lambda_{\ell^0}=0.001$\\
\bottomrule
\end{tabular}
\caption{Parameters for sparse deconvolution in  \cref{fig:deconv}.}
\end{table}


\begin{table}[H]
\centering\tiny
\begin{tabular}
{p{2cm}|c|c|c|c|c|c|c|c|p{1cm}|p{2cm}|p{1cm}|p{2cm}}
\toprule%
Algorithm & $\tau$ & $\alpha$ & noise & $\sigma$ & M & N & d & $k_{\text{max}}$ & batching & scheduler & init & Miscellaneous\\
\midrule%
MirrorCBO + $\ell^0$ &  0.02 & 100 & anisotropic & 3.0 & 1 & 1000 & 100 & 5000 & None & \cref{alg:mutlAlpha} multiply(1.05,1e18) & $\mathcal{U}(0, 1)$ & $\lambda_{\ell^0}=0.001$,  \cref{alg:indepnoise} resampling($\sigma=1.$, $\texttt{patience}=50$, $\eta = 0.9$)  \\
\midrule
MirrorCBO (fewer particles) + $\ell^0$ &  0.02 & 100 & anisotropic & 3.0 & 1 & 100 & 100 & 5000 & None & \cref{alg:mutlAlpha} multiply(1.05,1e18) & $\mathcal{U}(0, 1)$ & $\lambda_{\ell^0}=0.001$, \cref{alg:indepnoise} resampling($\sigma=1.$, $\texttt{patience}=50$, $\eta = 0.9$)\\
\midrule
Dualized CBO & 0.1 & 1e5 & anisotropic & 1. & 1 & 1000 & 100 & 5000 & None & \cref{alg:mutlAlpha} multiply(1.01,1e18) & $\mathcal{N}(0,1)$ & $\lambda=35$ \\
\midrule
Penalized CBO & 0.02 & 100.0 & anisotropic & 3.0 & 1 & 1000 & 100 & 5000 & None & \cref{alg:mutlAlpha} multiply(1.05,1e18) & $\mathcal{U}(0, 1)$ & $\lambda=2$, \cref{alg:indepnoise} resampling($\sigma=.5$, $\texttt{patience}=50$, $\eta = 0.99$)\\
\bottomrule
\end{tabular}
\caption{Parameters for sparse deconvolution in \cref{tab:conv}.}
\label{tab:convp}
\end{table}


\subsection{Hyperparameters for \ref{sec:nn}, sparse neural networks}


\begin{table}[h!]
\centering\tiny
\begin{tabular}
{p{2cm}|c|c|c|c|c|c|c|c|p{1cm}|p{2.5cm}|p{1cm}|p{2cm}}
\toprule%
Algorithm & $\tau$ & $\alpha$ & noise & $\sigma$ & M & N & d & $k_{\text{max}}$ & batching & scheduler & init & Miscellaneous\\
\midrule%
CBO & 0.1 & 50.0 & anisotropic & 0.1 & 1 & 100 & 7870 & None & batch size= 70, full updates & \cref{alg:effAlpha} EffectiveSampleSizeAlpha$(\cdot, 0.5, 1e7)$ & Default \texttt{PyTorch} init & \cref{alg:indepnoise} resampling($\sigma=.1$, $\texttt{patience}=1$, $\eta = 1.$, $\texttt{tol}=\infty$)\\
\midrule
MirrorCBO & 0.1 & 50.0 & anisotropic & 0.1 & 1 & 100 & 7870 & None & batch size= 70, full updates & \cref{alg:effAlpha} EffectiveSampleSizeAlpha$(\cdot, 0.5, 1e7)$ & Default \texttt{PyTorch} init & 
$\lambda=0.5$, 
\cref{alg:indepnoise} resampling($\sigma=.1$, $\texttt{patience}=1$, $\eta = 1.$, $\texttt{tol}=\infty$)\\
\midrule
SGD & 0.1 & None & None & None & None & None & 7870 & None & None & None & Default \texttt{PyTorch} init $\lambda=0.5$ & None\\
\bottomrule
\end{tabular}
\caption{Parameters for training sparse neural networks in \cref{tab:splearn}.}
\label{tab:splearnp}
\end{table}

\begin{table}[h!]
\centering\tiny
\begin{tabular}
{p{2cm}|c|c|c|c|c|c|c|c|p{1cm}|p{2.5cm}|p{1cm}|p{2cm}}
\toprule%
Algorithm & $\tau$ & $\alpha$ & noise & $\sigma$ & M & N & d & $k_{\text{max}}$ & batching & scheduler & init & Miscellaneous\\
\midrule%
MirrorCBO & 0.1 & 50.0 & anisotropic & 0.1 & 1 & 100 & 7870 & None & batch size= 70, full updates & \cref{alg:effAlpha} EffectiveSampleSizeAlpha$(\cdot, 0.5, 1e7)$ & Default \texttt{PyTorch} init $\lambda=0.5$ & 
$\lambda=0.5$, 
\cref{alg:indepnoise} resampling($\sigma=.1$, $\texttt{patience}=1$, $\eta = 1.$, $\texttt{tol}=\infty$)\\
\bottomrule
\end{tabular}
\caption{Parameters for training sparse neural networks in \cref{fig:spmatrix}.}
\label{tab:spmatrix}
\end{table}


\subsection{Hyperparameters for \ref{sec:simplex}, simplex constraints}

\begin{table}[H]
\centering\tiny
\begin{tabular}
{p{2cm}|c|c|c|c|c|c|c|c|p{1cm}|p{2.5cm}|p{1cm}|p{2cm}}
\toprule%
Algorithm & $\tau$ & $\alpha$ & noise & $\sigma$ & M & N & d & $k_{\text{max}}$ & batching & scheduler & init & Miscellaneous\\
\midrule%
CBO & 0.1 & 100.0 & anisotropic & 3.5 & 1 & 100 & 100 & 200 & None & \cref{alg:mutlAlpha} multiply(1.05,1e18) & Simplex initialization, see \cref{rem:simplexinit} & None\\
\midrule%
MirrorCBO & 5.0 & 100.0 & anisotropic & 4.0 & 1 & 100 & 100 & 200 & None & \cref{alg:mutlAlpha} multiply(1.05,1e18) & Simplex initialization, see \cref{rem:simplexinit} & None\\
\bottomrule
\end{tabular}
\caption{Parameters for the regression example on the simplex in \cref{fig:simplex}.}
\label{tab:simplex}
\end{table}

\subsection{Hyperparameters and details for \ref{sec:constraints_numerics}, hypersurface constraints}

In the following, we employ the \labelcref{eq:Ack} function with different parameters. We introduce the following abbreviations:
\begin{center}
\begin{tabular}{l|l}
Ackley-A     &  $a=20, b=0.1, c=2, x^\text{shift}=(0.4,\ldots, 0.4)$\\
Ackley-B     &  $a=20, b=0.1, c=2, x^\text{shift}$ sampled randomly\\
\end{tabular}
\end{center}


\begin{table}[H]
\centering\tiny
\begin{tabular}
{p{2cm}|c|c|c|c|c|c|c|c|p{1cm}|p{2.5cm}|p{1cm}|p{2.5cm}}
\toprule%
Algorithm & $\tau$ & $\alpha$ & noise & $\sigma$ & M & N & d & $k_{\text{max}}$ & batching & scheduler & init & Miscellaneous\\
\midrule%
MirrorCBO & 0.05 & 500.0 & isotropic & .25 & 1 & 50 & 2 & 400 & None & None & $\mathcal{N}(0,1)$ & Objective: Ackley-A\\
\bottomrule
\end{tabular}
\caption{Parameters for the visualization of MirrorCBO on a hyperplane in \cref{fig:plane}.}
\label{tab:planeparams}
\end{table}


\begin{table}[H]
\centering\tiny
\begin{tabular}
{p{2cm}|c|c|c|c|c|c|c|c|p{1cm}|p{2.5cm}|p{1cm}|p{2.5cm}}
\toprule%
Algorithm & $\tau$ & $\alpha$ & noise & $\sigma$ & M & N & d & $k_{\text{max}}$ & batching & scheduler & init & Miscellaneous\\
\midrule%
MirrorCBO & 0.1 & 500.0 & isotropic & 1.0 & 100 & 50 & 3 & 400 & None & \cref{alg:mutlAlpha} multiply(1.05,1e12) & $\mathcal{N}(0,1)$ & objective: Ackley-A\\
\midrule%
Projected CBO & 0.1 & 500.0 & isotropic & 1.0 & 100 & 50 & 3 & 400 & None & \cref{alg:mutlAlpha} multiply(1.05,1e12) & $\mathcal{N}(0,1)$ & objective: Ackley-A\\

\midrule%
Penalized CBO & 0.05 & 500.0 & isotropic & 0.5 & 100 & 50 & 3 & 400 & None &  \cref{alg:mutlAlpha} multiply(1.05,1e12) & $\mathcal{N}(0,1)$ & objective: Ackley-A, $p=2$\\%
\midrule%
Drift-constrained CBO & 0.1 & 500.0 & isotropic & 0.5 & 100 & 50 & 3 & 400 & None & \cref{alg:mutlAlpha} multiply(1.05,1e12) & $\mathcal{N}(0,1)$ & objective: Ackley-A, $p=2$\\
\midrule
Combination & 0.05 & 500.0 & isotropic & 1.2 & 100 & 50 & 3 & 400 & None & \cref{alg:mutlAlpha} multiply(1.05,1e12) & $\mathcal{N}(0,1)$ & objective: Ackley-A, $p=2$, $\lambda_1=200$, $\lambda_2=2$\\
\bottomrule
\end{tabular}
\caption{Parameters for different optimizers with the hyperplane constraint in \cref{fig:planecbo}.}
\label{tab:plane}
\end{table}


\begin{table}[H]
\centering\tiny
\begin{tabular}
{p{2cm}|c|c|c|c|c|c|c|c|p{1cm}|p{2.5cm}|p{1cm}|p{2.5cm}}
\toprule%
Algorithm & $\tau$ & $\alpha$ & noise & $\sigma$ & M & N & d & $k_{\text{max}}$ & batching & scheduler & init & Miscellaneous\\
\midrule%
MirrorCBO & 0.1 & 5.0 & anisotropic & 1.25 & 100 & 100 & 20 & 400 & None & \cref{alg:mutlAlpha} multiply(1.05,1e12) & Sphere & objective: Ackley-A\\
\midrule%
Projected CBO & 0.1 & 0
5.0 & anisotropic & 2.0 & 100 & 100 & 20 & 400 & None & \cref{alg:mutlAlpha} multiply(1.05,1e12) & Sphere & objective: Ackley-A\\

\midrule%
Penalized CBO & 0.05 & 5.0 & anisotropic & 2. & 100 & 100 & 20 & 400 & None &  \cref{alg:mutlAlpha} multiply(1.05,1e12) & Sphere & objective: Ackley-A, $p=2$\\%
\midrule%
Drift-constrained CBO & 0.1 & 5000.0 & anisotropic & 1.0 & 100 & 50 & 20 & 400 & None & \cref{alg:mutlAlpha} multiply(1.05,1e18) & $\mathcal{U}(-3,3)$ & objective: Ackley-A, $p=2$\\
\midrule
Combination & 0.1 & 5.0 & anisotropic & 1.0 & 100 & 100 & 20 & 400 & None & \cref{alg:mutlAlpha} multiply(1.05,1e18) & $\mathcal{U}(-3,3)$ & objective: Ackley-A, $p=2, \lambda_1=1/300,\lambda_2=1/0.9$\\
\midrule
Hypersurface CBO & 0.1 & 5.0 & anisotropic & 1.5 & 100 & 100 & 20 & 400 & None & \cref{alg:mutlAlpha} multiply(1.05,1e12) & Sphere & objective: Ackley-A\\
\bottomrule
\end{tabular}
\caption{Parameters for different optimizers with the sphere constraint in \cref{fig:spherecbo}.}\label{tab:sphereparams}
\end{table}

\setlength\tabcolsep{3pt}%
\begin{table}[H]
\centering\tiny
\begin{tabular}
{p{2cm}|c|c|c|c|c|c|c|c|p{1cm}|p{2cm}|p{2cm}|p{1.5cm}}
\toprule%
Experiment & $\tau$ & $\alpha$ & noise & $\sigma$ & M & N & d & $k_{\text{max}}$ & batching & scheduler & init & Miscellaneous\\
\midrule%
Mirror CBO & 0.1 & 0.001 & isotropic & 0.11 & 1 & 100 & 100 & 1e5 & None & multiply $(1.05, 1e15)$ & sphere(0.0, 1.0) & None\\
\midrule%
Hypersurface CBO &
0.1 & 2e15 & isotropic & 0.11 & 1 & 100 & 100 & 1e5 & None & None & sphere(0.0, 1.0) & None\\
\midrule%
Wirtinger &
10. & None & None & None & None & None & None & 1e5 & None & 100 & SpectralInit & Backtracking for $\tau$\\
\bottomrule
\end{tabular}
\caption{Parameters for the noiseless phase retrieval experiment in \cref{fig:phase_noiseless}.}
\end{table}


\setlength\tabcolsep{3pt}%
\begin{table}[H]
\centering\tiny
\begin{tabular}
{p{2cm}|c|c|c|c|c|c|c|c|p{1cm}|p{2cm}|p{2cm}|p{1.5cm}}
\toprule%
Experiment & $\tau$ & $\alpha$ & noise & $\sigma$ & M & N & d & $k_{\text{max}}$ & batching & scheduler & init & Miscellaneous\\
\midrule%
Mirror CBO & 0.3 & 0.001 & isotropic & 0.11 & 1 & 100 & 32 & 1e5 & None & multiply $(1.05, 1e18)$ & sphere(0.0, 1.0) & None\\
\midrule%
Hypersurface CBO &
0.1 & 1. & isotropic & 0.2 & 1 & 100 & 32 & 1e5 & None & multiply $(1.05, 1e18)$ & sphere(0.0, 1.0) & None\\
\midrule%
Wirtinger &
10. & None & None & None & None & None & 32 & 1e5 & None & None & SpectralInit & Backtracking for $\tau$\\
\bottomrule
\end{tabular}
\caption{Parameters for the noisy phase retrieval experiment in \cref{fig:phasenoisy}.}
\end{table}
In the following, we also detail the Wirtinger flow implementation used for the phase retrieval experiments. The spectral initialization and the basic gradient descent scheme are taken from \cite{candes2015phase}. We modified their step size selection scheme in favor of a simple backtracking scheme, which greatly improved the speed of convergence in our experiments.

%
%
\begin{algorithm}%
\caption{Spectral initialization, as proposed in \cite{candes2015phase}.}
\begin{algorithmic}[1]
\Function{SpectralInit}{$f\in\R^{M, d},y\in\R^M$}
    \State $\lambda \gets d \frac{\sum_{m=1}^M (y)_m}{\sum_{m=1}^M \abs{(f)_{m, :}}^2}$
    \State $Y\gets \frac{1}{M} \sum_{m=1}^M (y)_m (f)_{m, :} (f)_{m, :}^T$
    \State Set $v$ as the normalized eigenvector corresponding to the largest eigenvalue of $Y$
    \State \Return $\lambda z$
\EndFunction
\end{algorithmic}
\end{algorithm}%

%
\begin{algorithm}%
\caption{Wirtinger flow with backtracking.}
\begin{algorithmic}[1]
\Function{WirtingerFlow}{$f\in\R^{M, d},y\in\R^M, \tau_0\in(0,\infty)$}
\State $J(z):= \frac{1}{2M}\sum_{m=1}^M (\abs{\langle f_{m,:}, z\rangle}^2 - y)^2$
\State $z_0 = \textbf{SpectralInit}(f,y)$
\For{$i=0,...,k_{\text{max}}$}
\State $g_{k} = \frac{1}{M}\sum_{m=1}^M (\abs{\langle f_{m,:}, z_k\rangle}^2 - (y)_m)\ (f_{m,:}f_{m,:}^T) z_k$
\State $\tau \gets \tau_0$
\For {$j=0,...,50$}
\State $\tilde{z} \gets z_k - \tau\, g$
\If{$J(\tilde{z}) < J(z_k) + 0.1\, \tau\,
\abs{g}^2$}
\State \textbf{break}
\EndIf
\State $\tau\gets 0.2\, \tau$
\EndFor
\State $z_{k+1} = \tilde{z}$
\EndFor
\EndFunction
\end{algorithmic}
\end{algorithm}%


\setlength\tabcolsep{3pt}%
\begin{table}[H]
\centering\tiny
\begin{tabular}
{p{2cm}|c|c|c|c|c|c|c|c|p{1cm}|p{2cm}|p{2cm}|p{1.5cm}}
\toprule%
Experiment & $\tau$ & $\alpha$ & noise & $\sigma$ & M & N & d & $k_{\text{max}}$ & batching & scheduler & init & Miscellaneous\\
\midrule%
Mirror CBO & 0.3 & 0.001 & anisotropic & 2. & 100 & 100 & 20 & 400 & None & multiply $(1.05, 1e18)$ & $\mathcal{U}(-1,1)$ & objective: Ackley-A, resampling($\sigma=0.1$, $\texttt{patience}=5$, $\eta = .99$, $\texttt{tol}=1e-5$)\\
\midrule%
Projected CBO & 0.2 & 5000. & anisotropic & 2. & 100 & 100 & 20 & 400 & None & multiply $(1.05, 1e18)$ & $\mathcal{U}(-1,1)$ & objective: Ackley-A, resampling($\sigma=0.01$, $\texttt{patience}=5$, $\eta = .99$, $\texttt{tol}=1e-6$)\\
\midrule%
Penalized CBO & 0.2 & 5. & anisotropic & 2. & 100 & 100 & 20 & 400 & None & multiply $(1.05, 1e18)$ & $\mathcal{U}(-1,1)$ & objective: Ackley-A, $\lambda=0.003$, resampling($\sigma=0.01$, $\texttt{patience}=5$, $\eta = .99$, $\texttt{tol}=1e-6$)\\
\midrule%
Drift-constrained CBO & 0.1 &
50.0  & anisotropic & 2.0 & 100 & 100 & 20 & 400 & None & multiply $(1.05, 1e18)$ & $\mathcal{U}(-1,1)$ & 
objective: Ackley-A, $\lambda=100$, \cref{alg:indepnoise} resampling($\sigma=0.1$, $\texttt{patience}=5$, $\eta = .99$, $\texttt{tol}=1e-7$)\\
\midrule%
Combination &
0.1 & 5000.0 & anisotropic & 1.5 & 100 & 100 & 20 & 400 & None & multiply $(1.01, 1e18)$ & $\mathcal{U}(-1,1)$ & 
objective: Ackley-A, 
$\lambda_1=4/3$, $\lambda_2=2/3$,
\cref{alg:indepnoise} resampling($\sigma={1e-6}$, $\texttt{patience}=5$, $\eta = .99$, $\texttt{tol}=1e-6$)\\
\bottomrule
\end{tabular}
\caption{Parameters for different optimizers with the quadric constraint  \cref{fig:parabolic}.}
\end{table}


\setlength\tabcolsep{3pt}%
\begin{table}[H]
\centering\tiny
\begin{tabular}
{p{2cm}|c|c|c|c|c|c|c|c|p{1cm}|p{2cm}|p{2cm}|p{1.5cm}}
\toprule%
Experiment & $\tau$ & $\alpha$ & noise & $\sigma$ & M & N & d & $k_{\text{max}}$ & batching & scheduler & init & Miscellaneous\\
\midrule%
Mirror CBO & .2 & 5 & anisotropic & 1.5 & 100 & 100 & 50 & 400 & None & multiply $(1.05, 1e12)$ & Uniform on the manifold & objective: Ackley-B, resampling($\sigma=0.1$, $\texttt{patience}=5$, $\eta = .99$, $\texttt{tol}=1e-5$)\\
\midrule%
Mirror CBO (simpler) & .2 & 5 & isotropic & .2 & 100 & 100 & 50 & 400 & None & multiply $(1.05, 1e12)$ & Uniform on the manifold & objective: Ackley-B, None\\
\midrule%
Stiefel CBO & .2 & 5 & isotropic & .2 & 100 & 100 & 50 & 400 & None & multiply $(1.05, 1e12)$ & Uniform on the manifold & objective: Ackley-B, None\\
\bottomrule
\end{tabular}
\caption{Parameters for different optimizers with the Stiefel constraint  \cref{fig:Stiefel}.}
\end{table}


\setlength\tabcolsep{3pt}%
\begin{table}[H]
\centering\tiny
\begin{tabular}
{p{2cm}|c|c|c|c|c|c|c|c|p{1cm}|p{2cm}|p{2cm}|p{1.5cm}}
\toprule%
Experiment & $\tau$ & $\alpha$ & noise & $\sigma$ & M & N & d & $k_{\text{max}}$ & batching & scheduler & init & Miscellaneous\\
\midrule%
Mirror CBO & Varying & 10 & isotropic & 1. & 50 & 20 & 2 & 100 & None & None & $\mathcal{U}(-1,1) $& None \\
\midrule%
CBO & Varying & 10 & isotropic & 1. & 50 & 20 & 2 & 100 & None & None & $\mathcal{U}(-1,1) $& None \\
\bottomrule
\end{tabular}
\caption{Parameters for the preconditioning example in \cref{fig:precon}.}
\end{table}
\end{appendix}

\end{document}